\documentclass[twoside]{amsart}

\usepackage{amsmath, amsthm, amssymb, mathrsfs}
\usepackage{epsfig, graphicx}
\usepackage{verbatim,setspace}
\usepackage[colorlinks=true, citecolor=blue]{hyperref}

\newtheorem{theorem}{Theorem}
\newtheorem{corollary}[theorem]{Corollary}
\newtheorem{proposition}[theorem]{Proposition}
\newtheorem{lemma}{Lemma}

\numberwithin{equation}{section}

\newcommand{\T}		{\mathbb{T}}
\newcommand{\D}		{\mathbb{D}}
\newcommand{\R}		{\mathbb{R}}
\newcommand{\C}		{\mathbb{C}}
\newcommand{\N}		{\mathbb{N}}
\newcommand{\Z}		{\mathbb{Z}}
\newcommand{\Q}		{\mathbb{Q}}
\newcommand{\RS}		{\mathfrak{R}}

\newcommand{\map}	{\varphi}
\newcommand{\eqm}	{\omega_{\Delta}}

\newcommand{\z}		{\mathbf{z}}
\newcommand{\w}		{\mathbf{w}}
\newcommand{\tr}		{\mathbf{t}}
\newcommand{\ar}		{\mathbf{a}}
\newcommand{\br}		{\mathbf{b}}

\newcommand{\cp}		{\textnormal{cap}}
\newcommand{\const}	{\textnormal{const.}}
\newcommand{\diam}	{\textnormal{diam}}

\newcommand{\im}		{\textnormal{Im}}
\newcommand{\re}		{\textnormal{Re}}

\begin{document}

\title[Pad\'e approximants to elliptic-type functions]{Pad\'e approximants to  certain elliptic-type functions}

\author[L. Baratchart]{Laurent Baratchart}

\address{INRIA, Project APICS, 2004 route des Lucioles --- BP 93, 06902 Sophia-Antipolis, France}

\email{laurent.baratchart@sophia.inria.fr}

\author[M. Yattselev]{Maxim L. Yattselev}

\address{Corresponding author, Department of Mathematics, University of Oregon, Eugene, OR, 97403, USA}

\email{maximy@uoregon.edu}

\begin{abstract}
Given non-collinear points $a_1$, $a_2$, $a_3$, there is a unique compact $\Delta\subset\C$ that has minimal logarithmic capacity among all continua joining $a_1$, $a_2$, and $a_3$. For $h$ be a complex-valued non-vanishing Dini-continuous function 
on $\Delta$, we consider
\[
f_h(z) := \frac{1}{\pi i}\int_\Delta\frac{h(t)}{t-z}\frac{dt}{w^+(t)},
\]
where $w(z) := \sqrt{\prod_{k=0}^3(z-a_k)}$ and $w^+$  the one-sided value according to some orientation of $\Delta$. In this work we present strong asymptotics of diagonal Pad\'e approximants to $f_h$, as $n\to\infty$, and describe the behavior of the spurious pole and the regions of locally uniform convergence from a generic perspective.
\end{abstract}

\subjclass[2000]{42C05, 41A20, 41A21}

\keywords{Pad\'e approximation, orthogonal polynomials, non-Hermitian orthogonality, strong asymptotics.}

\maketitle

\section{Introduction} 

A truncation of continued fractions in the field of  Laurent series 
in one complex variable, Pad\'e approximants are among the oldest and simplest 
constructions in function theory \cite{JonesThron}.  These are 
rational functions of type\footnote{A rational function is said to be of type $(m,n)$ if it can be written as the ratio of a polynomial of degree at most $m$ and a polynomial of degree at most $n$.} $(m,n)$ that interpolate 
a function element at a given point  with order $m+n+1$. They were 
introduced  for the exponential function by Hermite \cite{Her}, 
who used them to prove the transcendency of $e$,
and later expounded more systematically by his student Pad\'e  \cite{Pade92}.
Ever since their introduction, Pad\'e approximants
have been an effective device in analytic number theory 
\cite{Her,Siegel,Skor03,KR}, and 
over the last decades they  became an important tool in
physical modeling and numerical analysis;
the reader will find an  introduction to such topics, as 
well as further references, in the  monograph 
\cite{BakerGravesMorris}, see also \cite{CeOk,DruMo,Pozzi}.

Still, convergence properties of Pad\'e approximants are not 
fully  understood as yet. Henceforth, for simplicity, we only discuss
approximants of type $(n,n)$ (the so-called \emph{diagonal} approximants)
which are  most natural since they treat poles and 
zeros on equal footing.  For restricted classes of functions,
such approximants were proven to converge, locally uniformly in the 
domain of analyticity, as $n$ goes large. 
These classes include Markov functions and rational perturbation thereof
\cite{Mar95,G75,Rakh77,BeckDerZh10}, 
Cauchy transforms of continuous non-vanishing functions on a 
segment \cite{Baxter,NuSi,Mag87} 
(in these examples interpolation takes place at infinity), and
certain entire functions such as Polya frequencies or
functions with smooth and fast decaying Taylor coefficients
\cite{ArmEd70,Lub85,Lub88} (interpolation being now at the origin). 
However, such favorable cases do not reflect the general situation which is
that Pad\'e approximants often fail to converge locally uniformly,
due to the occurrence of ``spurious'' poles
that may wander about the domain of analyticity. The so-called
Pad\'e conjecture, actually raised by Baker, Gammel and Wills \cite{Bak73},
laid hope for the next best thing namely convergence of a subsequence
in the largest disk of holomorphy, but this was eventually settled in the 
negative by D. Lubinsky \cite{Lub03}. 
Shortly after, a weaker form of the conjecture 
due to H. Stahl \cite{StahlC}, dealing with hyperelliptic functions, 
was disproved as well  by V. Buslaev \cite{Buslaev}.

Nevertheless, spurious poles are no obstacle to some
 weaker type of convergence. Indeed,
convergence in (logarithmic) capacity of Pad\'e 
approximants to functions with singular set of capacity zero
(the prototype of which is an  entire function)
was established by J.~Nuttall and Ch.~Pommerenke \cite{Nut90, Pom73}. 
Later, in his pathbreaking work \cite{St85,St85b,StahlOP,St89,St97} dwelling
on earlier study by J. Nuttall and  S. Singh \cite{NuSi,NutP}, H. Stahl proved
an analogous result for (branches of) functions $f$ having multi-valued 
meromorphic continuation over the plane deprived of a set of zero capacity 
(the prototype of which is an algebraic function). 
Here, there is an additional problem of identifying the convergence domain,
since the approximants are single-valued in nature but the approximated 
function is not. It turns out to be characterized, 
among all domains 
on which $f$ is meromorphic and single-valued, as one whose 
complement has minimal capacity.\footnote{This characterization is
up to a set of zero capacity only, but 
the union of all such domains is again a convergence domain, maximal with
respect to set-theoretic-inclusion, that we call \emph{the} convergence domain
in capacity of Pad\'e approximants to $f$.} As $n$ tends to 
infinity,
this complement attracts almost all the poles of the
Pad\'e approximant of order $n$ 
(that is, all but at most $o(n)$ of them). However, the actual limit set of 
the poles can be 
significantly larger, possibly the whole complex plane, which is the 
reason why uniform convergence may fail.
 
When the singular set of $f$
consists of finitely many branchpoints, the complement $\Delta$
of the convergence 
domain is a  (generally branched) system of analytic cuts 
without loop, whose loose ends are branchpoints of $f$, which is 
called an $S$-contour. Here the prefix ``S'' stands for
``symmetric'',  meaning  that the equilibrium potential of $\Delta$
has equal normal derivative from each side at every smooth point.  
Actually, this symmetry property expresses that the first order variation
of the capacity is zero under small distortions of the contour.

Stahl's work dwells on the classical and
fruitful connection between Pad\'e 
approximants and orthogonal polynomials:
if $f$ can be expressed as the Cauchy integral of a 
compactly supported (possibly complex)
measure $\nu$, then the denominator of the 
Pad\'e approximant of type $(n,n)$ to $f$ at infinity
is orthogonal to all polynomials of degree at most $n-1$ for the 
non-Hermitian\footnote{This means there is no conjugation involved, {\it i.e.}
the scalar product is $\langle g,h\rangle:=\int gh d\nu$.}
 scalar product defined by $\nu$ in $L^2(|\nu|)$. 
In the case of  Markov functions, $\nu$ is a positive measure supported on a 
real segment (a segment is the simplest example of an $S$-contour), 
so the orthogonality is in fact Hermitian and the limiting 
behavior of the denominator can be addressed using classical asymptotics of
orthogonal polynomials \cite{StahlTotik}.
In this connection, the work \cite{StahlOP} provides one with a 
non-Hermitian generalization to arbitrary $S$-contours of that part of the 
theory of orthogonal polynomials on a segment dealing with weak
(i.e., $n$-th root) asymptotics.\footnote{Note that $f$ can be written as 
the Cauchy integral of the difference between its 
values from each side of the $S$-contour, at least if the branchpoint have 
order $>-1$; if not, special treatment is needed at the 
endpoints, see \cite{StahlOP}.}

To determine subregions where uniform 
convergence of Pad\'e approximants takes place, if any, one needs to
analyze the behavior of \emph{all} the poles when $n$ goes large and
not just $o(n)$ of them. For functions with finitely many branchpoints, in light of 
the previous discussion, it is akin to
carrying over to a non-Hermitian context, over general $S$-contours, 
the Szeg\H{o} theory of strong asymptotics for orthogonal polynomials. 

On a segment $K$, strong asymptotics for non-Hermitian orthogonal polynomials
$q_n$ with respect to an absolutely continuous complex measure of the form 
$hd\omega_K$, with $\omega_K$ the equilibrium measure of $K$,
was obtained in \cite{Baxter,NuSi,NutP,Nut90} {\it via} the study of
certain singular integral equations (see \cite{AP02,AVA04,BY09c,BaYa10} for 
generalizations to varying weights over analytic arcs).
When the density $h$ is smooth and does not vanish, 
the asymptotics is similar to classical one:
up to normalization, $q_n$ is equivalent for large $n$
to $\Phi^n /S$,  locally uniformly outside of $K$, 
 where $\Phi$ conformally maps the complement of $K$ to the 
complement of the unit disk and $S$ is an auxiliary function,
the Szeg\H{o} function of $h$, which solves a Riemann-Hilbert problem
across $K$ and has no  zeros.
In particular $K$ attracts all zeros of $q_n$ asymptotically,
 which is equivalent to saying that there are no spurious poles
in Pad\'e approximation to the Cauchy transform of $hd\omega_K$.

When $h$ does have zeros (and even in the Hermitian case
if it has non-convex support), 
spurious poles appear whose number can sometimes be
estimated from the nature of the zeroing and the smoothness
of $\arg h$ \cite{NuSi,StD,Stahl91,BKT05}. But it is S.P.~Suetin,
for analytic non-vanishing densities on an $S$-contour comprised of
finitely many disjoint arcs (in particular on a finite union
of real intervals), 
who converted Nuttall's singular equation approach to
non-Hermitian orthogonality \cite{Nut90} into
an affine Riemann-Hilbert problem on a hyperelliptic Riemann surface 
and  who linked the occurrence of spurious poles to the outcome of
a Jacobi inversion process \cite{Suet00}. 
In the case of two arcs 
the Riemann surface is elliptic (i.e., it has genus 1), and there is 
at most one spurious pole whose recurrent 
behavior can be explained by the rational independence of the equilibrium 
weights of the arcs,
much for the same reason why a line with irrational slope embedded in 
the unit torus fills a dense subset of the latter \cite{Suet03}
(see \cite{Ap08,ApLy10} for a generalization). These results stress a parallel 
between non-Hermitian orthogonality and the
theory of Hermitian orthogonal polynomials on a system of curves initiated 
by H.~Widom \cite{Widom,GeVan,Peher}.

Now, Suetin's work tells us about spurious poles of functions with four 
branchpoints of order 2 in special position, namely  
the associated $S$-contour should consist of two disjoint arcs. 
In the present paper, drawing inspiration from \cite{Suet03}, we deal with the case of three branchpoints in arbitrary (non-collinear) position. The 
corresponding $S$-contour is a threefold, and thus we consider 
orthogonality on a non-smooth (in fact branched) contour. 
For a special class of Jacobi polynomials, such a setting was 
considered by Nuttall \cite{Nut86}, using a different method, and just recently by Mart\'inez Finkelshtein, Rakhmanov, and Suetin in \cite{M-FRakhSuet12}.
Here, using classical properties of singular integrals, we handle at little  
extra cost Dini-continuous non-vanishing densities 
(that may not be analytic). Moreover, we put our results in generic
perspective with respect to the location of the branchpoints 
employing differential geometric tools and properties of
quadratic differentials.

We first identify a convenient Riemann surface $\RS$ 
and a suitable curve $L\subset\RS$ 
over which we can lift non-Hermitian
orthogonality on the threefold into a Riemann-Hilbert problem. 
This step is somewhat more involved than in \cite {Suet00} but the surface 
we construct is still elliptic.

Next, to analyze the Riemann-Hilbert problem thus obtained in each degree $n$, 
we first solve it explicitly when the density is the reciprocal of a 
polynomial, in terms of (the two branches of) some auxiliary 
function $S_n$. 
Then, to handle an arbitrary Dini-continuous non-vanishing density, 
we approximate it 
by a sequence of reciprocal of polynomials and regard this case as
a perturbation of the previous one, using some singular  integral theory.

The function $S_n$
is holomorphic in $\mathfrak{R}\setminus L$ and plays here the role
of the product $\Phi^n/S$ which is the main term in the asymptotics 
of $q_n$ in the segment case.
It has at most one finite zero, given by the solution of a
Jacobi inversion problem (which depends on $n$). 
When this zero belongs to the first sheet of the covering, 
it generates a spurious pole nearby, whereas
there are no spurious poles when it belongs to the second sheet.

To describe the dynamics of this wandering zero, we proceed as in \cite{Suet03}
by mapping the Jacobi inversion problem to an equation on the
Jacobian variety of $\mathfrak{R}$ (which is a torus). There, 
the image of the zero evolves according to a discrete linear dynamical 
system whose coefficients
depend on the equilibrium weights of the arcs of the threefold.
The spurious pole
recurs in a dense manner if these equilibrium weights
are rationally independent, and eventually disappears or 
clusters to a union of disjoint arcs (resp. points)
if they are rationally dependent but one of them is irrational (resp. 
if they are all rational). We establish that the recurrent 
case is  generic in the measure-theoretic sense, and is one where the domain of convergence of the  sequence of Pad\'e approximants  is empty although some subsequence converges locally uniformly in the complement of the threefold. 
Still, the clustering case densely occurs and is one where the domain of 
convergence is nonempty.

The paper is organized as follows. In Section~\ref{sec:mr} we construct the 
surface $\RS$, we introduce our main objects of study
(sectionally holomorphic functions, Pad\'e approximants), and we
state our results. In Section~\ref{sec:ci}, we discuss Cauchy integrals and 
use them to construct Szeg\H{o} functions on the threefold.
Section~\ref{sec:rs} contains basic facts on
Abelian differentials, which are used in Section~\ref{sec:rhp}
to devise  certain sectionally meromorphic functions on $\mathfrak{R}$.
These are instrumental in Section~\ref{sec:sn} where we construct $S_n$,
derive formulae for the product and ratio of its branches, 
and analyze the behavior of its wandering zero. 
At last, in Section~\ref{sec:proofs}, we solve our 
initial Riemann-Hilbert problem in terms of $S_n$ and prove the announced 
asymptotics for non-Hermitian orthogonal polynomials and Pad\'e approximants
to Cauchy integrals on the threefold.

The author's motivation for writing the present paper has been twofold.
On the one hand, we aimed at carrying over to more general 
geometries and putting in generic perspective 
the mechanism behind the dynamics of spurious poles, first 
unveiled by Suetin, which offers beautiful connections with classical 
function theory on Riemann surfaces. On the other hand, we wanted to 
illustrate that Szeg\H{o}'s theory of orthogonal polynomials
can be generalized to both non-Hermitian and non-smooth
context, and does not owe that much to positivity.

\section{Main Results}
\label{sec:mr}

\subsection{Chebotar\"ev Continua}
\label{Ccont}
Let $a_1$, $a_2$, and $a_3$ be three non-collinear points in the complex plane $\C$. There exists a unique connected compact $\Delta=\Delta(a_1,a_2,a_3)$, called Chebotar\"ev continuum \cite{Grot30,Lav30,Lav34,Kuzmina}, containing these points
and having minimal \emph{logarithmic capacity} \cite{Ransford} among all continua connecting $a_1$, $a_2$, and $a_3$. It consists of three analytic arcs $\Delta_k$, $k\in\{1,2,3\}$, emanating from a common endpoint $a_0$, called the
Chebotar\"ev center, and ending at each of the given points $a_k$, respectively. It is also known that the tangents at $a_0$ of two adjacent arcs form an angle of magnitude $2\pi/3$. In what follows, we assume that the arcs $\Delta_k$ and the corresponding points $a_k$, are ordered clockwise with respect to $a_0$ (see Figure~\ref{fig:delta}). The interiors 
$\Delta_k^\circ:=\Delta_k\setminus\{a_0,a_k\}$ of
the arcs $\Delta_k$ can be described (see {\it e.g.}
\cite[Thm. 1.1]{Kuzmina}) as the (negative) critical trajectories of the 
quadratic differential 
\begin{equation}
\label{QD}
\frac{1}{\pi}\frac{z-a_0}{\prod_{k=1}^3(z-a_k)}(dz)^2.
\end{equation}
In other words, for any smooth parametrization $z_k(t)$, $t\in[0,1]$, of $\Delta_k$, it holds that
\begin{equation}
\label{eq:critr}\frac{1}{\pi}
\frac{z_k(t)-a_0}{\prod_{k=1}^3(z_k(t)-a_k)}(z^\prime_k(t))^2<0 \quad \mbox{for all} \quad t\in(0,1).
\end{equation}

In what follows, we denote by $D$ the complement of $\Delta$ in the extended complex plane $\overline\C$ and orient each arc $\Delta_k$ from $a_0$ to $a_k$. 
According to this orientation we distinguish the left $(+)$ and right $(-)$ 
sides of each $\Delta_k^\circ$ and therefore of 
$\Delta^\circ:=\cup_{k=1}^3\Delta_k^\circ=\Delta\setminus\{a_0,a_1,a_2,a_3\}$. 

\begin{figure}[!ht]
\centering
\includegraphics[scale=.5]{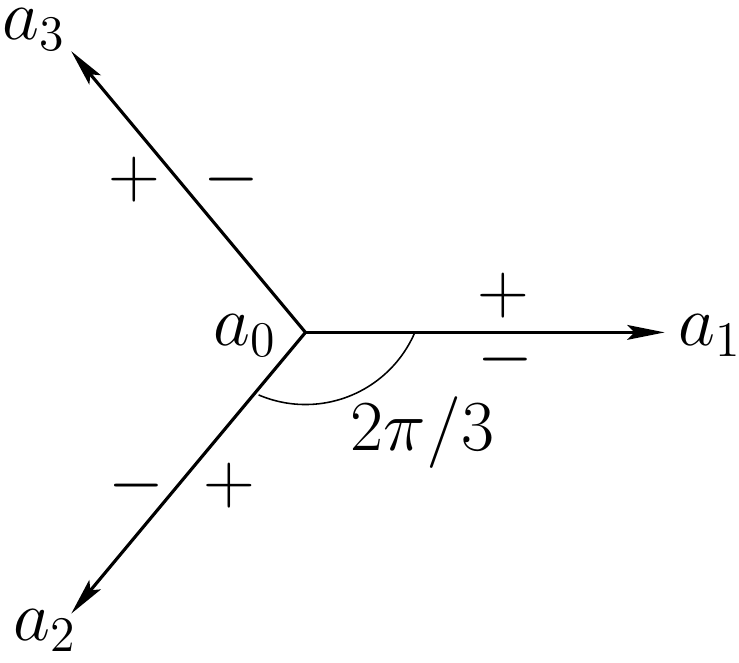}
\caption{\small Contour $\Delta$ consists of the arcs $\Delta_k$ which are oriented from $a_0$ to $a_k$ and are numbered clockwise. The left and right sides of the arcs $\Delta_k$ are labeled by signs $+$ and $-$, respectively.}
\label{fig:delta}
\end{figure}

On some occasions, it will be more useful to consider the boundary of $D$ as a limit of simple Jordan curves encompassing $\Delta$ and whose exterior domains exhaust $D$. Thus, we define $\partial D$ to be the ``curve'' consisting of two copies of each arc $\Delta_k^\circ$, the left and right sides, three copies of $a_0$, and a single copy of each $a_k$. We assume $\partial D$ to be oriented clockwise, that is, $D$ lies to the left of $\partial D$ when the latter is traversed in the positive direction.

The following function plays a prominent role throughout this work. We set
\begin{equation}
\label{eq:chebweight}
w(z) := \sqrt{\prod_{k=0}^3(z-a_k)}, \quad \frac{w(z)}{z^2}\to1 \quad \mbox{as} \quad z\to\infty,
\end{equation}
which is a holomorphic function in $D\setminus\{\infty\}$. It is easy to see that $w$ has continuous trace on $\partial D$ and it holds that $w^+=-w^-$, where $w^+$ and $w^-$ are the traces of $w$ from the left and right, respectively, on each $\Delta_k$.

\subsection{An Elliptic Riemann Surface}
Let $\RS$ be the Riemann surface defined by $w$. The genus of the Riemann surface of an algebraic function is equal to the number of branch points divided by two, plus one, minus the order of branching. Thus, $\RS$ has genus 1, that is,  $\RS$ is an elliptic Riemann surface. We represent $\RS$ as two-sheeted ramified cover of $\overline\C$ constructed in the following manner. Two copies of $\overline\C$ are cut along each arc $\Delta_k^\circ$ comprising $\Delta$. These copies are joint at each point $a_k$ and along the cuts in such a manner that the right (left) side of each $\Delta_k^\circ$ of the first copy, say $\RS^{(1)}$, is joined with the left (right) side of the respective $\Delta_k^\circ$ of the second copy, $\RS^{(2)}$, Figure~\ref{fig:gluing}. Thus, to each arc $\Delta_k$ in $\C$ there corresponds a cycle $L_k$ on $\RS$. 

We denote by $L$ the union $L_1\cup L_2\cup L_3$ and by $\pi$ the canonical projection $\pi:\RS\to\overline\C$. In particular, it holds that $\pi(L_k)=\Delta_k$. Each point in $\overline\C$ has two preimages on $\RS$ under $\pi$ except for the points $a_k$, $k\in\{0,1,2,3\}$, which have only one preimage.  For each $z\in D$ we set $z^{(k)}:=\pi^{-1}(z)\cap\RS^{(k)}$, $k\in\{1,2\}$, and write $\z$ for a generic point on $\RS$ such that $\pi(\z)=z$. We call two points \emph{conjugate} if they have the same canonical projection and denote the conjugation operation by the superscript $*$. Furthermore, we put $D^{(k)}:=\pi^{-1}(D)\cap\RS^{(k)}$. We orient each $L_k$ in such a manner that $D^{(1)}$ remains on the left when $L_k$ is traversed in the positive direction and these orientations induce an orientation on $L$, Figure~\ref{fig:torus}.

\begin{figure}[!ht]
\centering
\includegraphics[scale=.4]{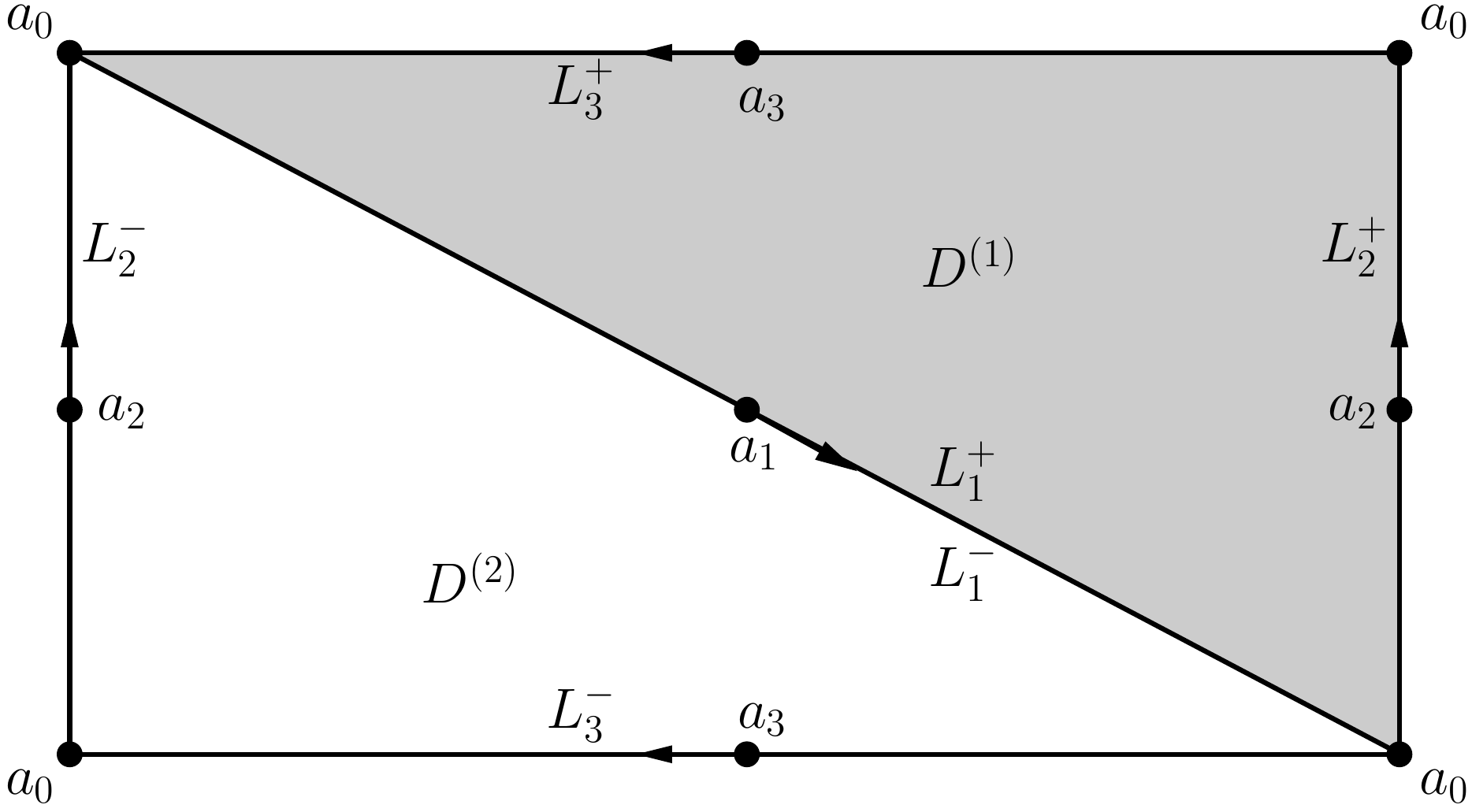}
\caption{\small Elliptic Riemann surface $\RS$ has genus 1 and therefore is homeomorphic to a torus. We represent $\RS$ as a torus cut along curves $L_2$ and $L_3$. In this case domains $D^{(1)}$ and $D^{(2)}$ can be represented as the upper and lower triangles, respectively.}
\label{fig:torus}
\end{figure}

We identify $D^{(1)}$ with $D$ and $L^+$ with $\partial D$. This means that we consider each function defined on $D$ as a function defined on $D^{(1)}$. In particular, we set
\begin{equation}
\label{eq:wrs}
w(\z) = \left\{
\begin{array}{rl}
w(z), &  \z\in D^{(1)}, \\
-w(z), & \z\in D^{(2)}.
\end{array}
\right.
\end{equation}
Clearly, $w$ extends continuously to each side of $L$ and the traces of $w$ on $L^+$ and $L^-$ coincide. Thus, $w$ is holomorphic across $L$ by the principle of analytic continuation. That is, $w$ is in fact a rational function over 
$\RS$ (a holomorphic map from $\mathfrak{R}$ into $\overline{\C}$). 

\subsection{Sectionally Meromorphic Functions}

A holomorphic (resp. meromorphic) function on $\RS\setminus L$ is called  sectionally holomorphic (resp. meromorphic). 
In this section we discuss some properties of sectionally  holomorphic functions on $\RS$ that we use further below. 

Let $r_h$ be a meromorphic function in $\RS\setminus L$ with continuous traces on $L$ that satisfy
\begin{equation}
\label{eq:bvpjumpdelta}
r_h^-=r_h^+\cdot(h\circ\pi),
\end{equation}
where $h$ is a continuous function on $\Delta\setminus\{a_0\}$ which extends continuously to each $\Delta_k$. The function $r_h$ gives rise to two meromorphic functions in $D$, namely,
\begin{equation}
\label{eq:conjfun}
r_h(z):=r_h(z^{(1)}) \quad \mbox{and} \quad r_h^*(z):=r_h(z^{(2)}), \quad z\in D,
\end{equation}
where we abuse notation in that we use $r_h$ to stand for both a function on 
$\mathfrak{R}$ and its restriction to $D$
(as mentioned before, we identify $D$ with $D^{(1)}$).
We call $r_h$ and $r_h^*$ the \emph{conjugate functions derived from $r_h$}.

\begin{figure}[!ht]
\centering
\includegraphics[scale=.5]{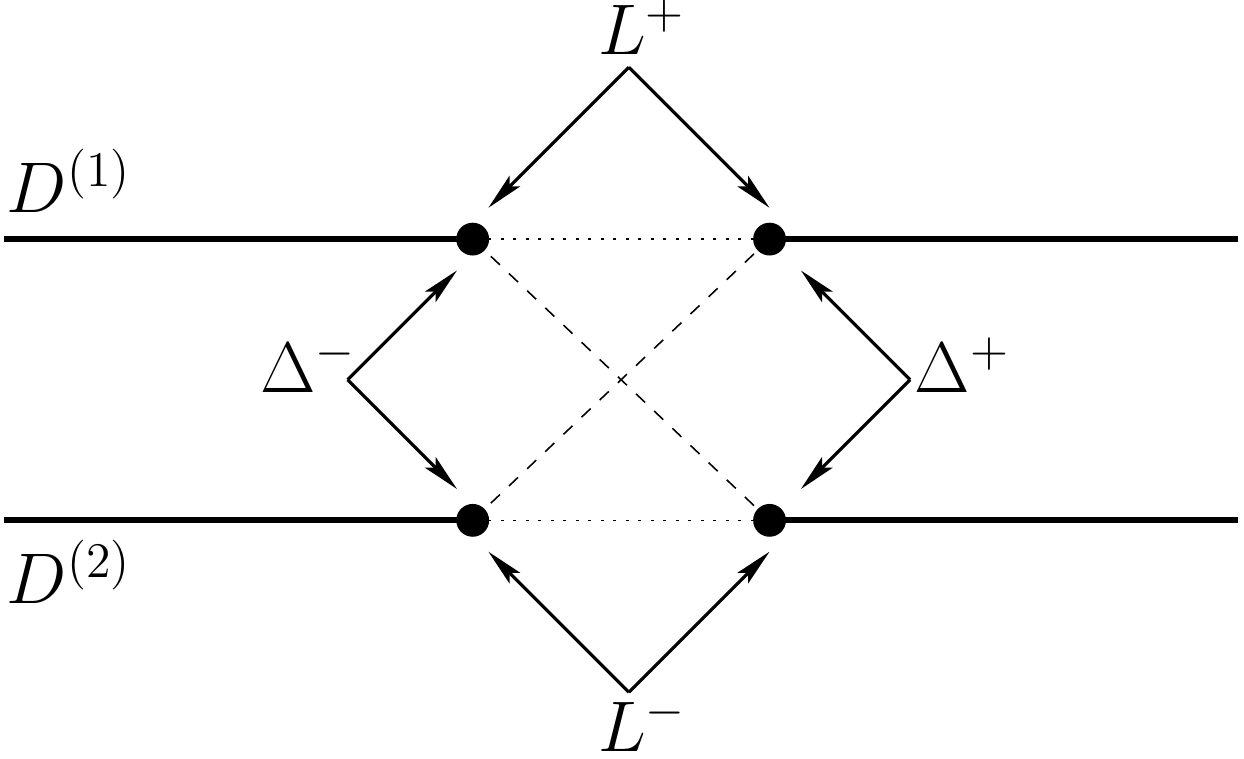}
\caption{\small Domains $D^{(1)}$ and $D^{(2)}$ are represented as upper and lower layers, respectively (two thick horizontal lines each). Each pair of disks joint by a dotted line represents the same point on $\Delta$ as approached from the left ($\Delta^-$) and from the right ($\Delta^+$). Each pair of disk joint by a punctured line represents the same point on $L$ as approached from the left ($L^-$) and from the right ($L^+$). The left and right sides are chosen according to the orientation of each contour in question.}
\label{fig:gluing}
\end{figure}
 Since the domains $D^{(k)}$ are ``glued'' to each other crosswise across $\Delta$ (see Figure~\ref{fig:gluing}),  boundary value problem \eqref{eq:bvpjumpdelta} gives rise to the following relation between the traces of $r_h$ and $r_h^*$ on $\Delta$:
\begin{equation}
\label{eq:bvp}
(r_h^*)^\pm = r_h^\mp h.
\end{equation}
Relations \eqref{eq:bvp} have two useful consequences. Firstly, if $h$ is holomorphic in some neighborhood of $\Delta$ then so is $r_h^*+hr_h$. Indeed, we only need to verify that this function has no jump on $\Delta$. The latter follows from \eqref{eq:bvp} and the computation:
\begin{equation}
\label{eq:sum}
(r_h^*+hr_h)^\pm = (r_h^*)^\pm+hr_h^\pm = hr_h^\mp + (r_h^*)^\mp = (r_h^*+hr_h)^\mp.
\end{equation}
Secondly, the product $r_hr_h^*$ is a rational function over $\overline\C$ 
as soon as $h$ is continuous. 
Indeed, as $r_hr_h^*$ is clearly meromorphic in $D$, we need only  
check the behavior across $\Delta$. If $h$ vanishes on a subset of positive 
linear measure of $\Delta$, then $r_h^*\equiv0$ by \eqref{eq:bvp} and 
Privalov's theorem. Otherwise  \eqref{eq:bvp} implies that 
$(r_hr_h^*)^+=(r_hr_h^*)^-$ a.e. on $\Delta^\circ$, and since 
 $r_hr_h^*$  is bounded we
get by a standard continuation principle \cite[Ch.~II, Ex.~12]{Garnett}. that it extends holomorphically across each $\Delta_k^\circ$. Finally, because it has bounded behavior near each $a_k$, the latter are removable singularities, as desired.

When
$r_h$ as above is not constant, we define its \emph{principal divisor} as
\begin{equation}
\label{eq:prdivisor}
(r_h) := \sum_l m_l\z_l - \sum_j k_j\w_j,
\end{equation}
meaning that $r_h$ has a pole (resp. zero) of multiplicity $k_j$ (resp.
$m_l$) at each $\w_j
\in D^{(1)}\cup D^{(2)}$ (resp. $\z_l\in D^{(1)}\cup D^{(2)}$), and that
$r_hr_h^*$ has a pole (resp. zero) of multiplicity $k_j$ (resp. $m_l$) 
at each $\pi(\w_j)$ (resp. $\pi(\z_l)$) if $\w_j\in L$ (resp. $\z_l\in L$),
while $r_h$ has finite non-zero value at any other point of $\RS$ 
including its two-sided boundary values on $L$. Then it is easy to see that
\begin{equation}
\label{eq:product}
(r_hr_h^*)(z) = \const\prod_{|\pi(z_l)|<\infty}(z-\pi(\z_l))^{m_l}\prod_{|\pi(w_j)|<\infty}(z-\pi(\w_j))^{-k_j}.
\end{equation}
In particular, $\sum_lm_l=\sum_jk_j$ as the number of poles and the number of 
zeros for a rational function over $\overline\C$ are the same. Thus definition
\eqref{eq:prdivisor} generalizes to sectionally meromorphic
functions on $\mathfrak{R}$ meeting \eqref{eq:bvpjumpdelta} the well-known
fact that non-constant rational functions on a compact Riemann surface have as 
many zeros as poles, counting multiplicities.
Note that $w$ is such a rational function and 
that its principal divisor is 
$(w)=\sum_{k=0}^3a_k-2\infty^{(1)}-2\infty^{(2)}$.

\subsection{Szeg\H{o}-type Functions}

The function $S_n$, introduced in this section, will provide the main term of the asymptotics of Pad\'e approximants to functions of the form \eqref{eq:CauchyT}. Before stating our first proposition, recall that a function $h$ is called Dini-continuous on $\Delta$ if
\[
\int_{[0,\diam(\Delta)]}\frac{\omega_h(\tau)}{\tau}d\tau<\infty, \quad \omega_h(\tau) := \max_{|t_1-t_2|\leq\tau}|h(t_1)-h(t_2)|,
\]
where $\diam(\Delta):=\max_{t_1,t_2\in\Delta}|t_1-t_2|$. 

\begin{proposition}
\label{prop:sn}
Let $h$ be a Dini-continuous non-vanishing function on $\Delta$. Then there exists $\z_n\in\RS$ such that $\z_n + (n-1)\infty^{(2)} - n\infty^{(1)}$ is the principal divisor of a function $S_n$ which is meromorphic in $\RS\setminus L$ and has continuous traces on $L$ from both sides  which satisfy
\begin{equation}
\label{eq:jumpL}
S_n^- = S_n^+\cdot(h\circ\pi).
\end{equation}
Moreover, under the normalization $S_n(\z)z^{-k_n}\to1$ as $\z\to\infty^{(1)}$, where $k_n=n-1$ if $\z_n=\infty^{(1)}$ and $k_n=n$ otherwise, $S_n$ is the unique function meromorphic in $\RS\setminus L$ with  principal divisor of the form $\w+(n-1)\infty^{(2)} - n\infty^{(1)}$, $\w\in\RS$, and continuous traces on $L$ that satisfy \eqref{eq:jumpL}. Furthermore, if $\z_n=\infty^{(1)}$ then $\z_{n-1}=\infty^{(2)}$ and $S_n=S_{n-1}$.
\end{proposition}

Observe that when $h\equiv1$ the principle of analytic continuation implies that
$S_n$ is simply a rational function over $\RS$  having $n$ poles at $\infty^{(1)}$ and $n-1$ zeros at $\infty^{(2)}$. The point $\z_n$ is then
determined by the geometry of $\RS$ as one cannot prescribe all  poles 
and zeros of rational functions over Riemann surfaces of non-trivial genus (see Section~\ref{subs:jip}). 

Note also that, when $h=1/p$, where $p$ is an algebraic polynomial non-vanishing 
on $\Delta$, equation \eqref{eq:sum} yields that $S_n+pS_n^*$ is a monic polynomial of degree $k_n$ if $2n>\deg(p)+2$.

Denote by $\map$ the conformal map of $D$ onto $\{|z|>1\}$ with
$\map(\infty)=\infty$, $\map^\prime(\infty)>0$. Then
\begin{equation}
\label{map}
\map(z) = \frac{z}{\cp(\Delta)} + \ldots,
\end{equation}
where $\cp(\Delta)$ is the \emph{logarithmic capacity} of $\Delta$, see \cite{Ransford} (often \eqref{map} serves as the definition of the logarithmic capacity of a continuum). Denote also by $\eqm$ the \emph{equilibrium (harmonic) measure} on $\Delta$, see \cite{Ransford}. It is known\footnote{
By \eqref{eq:critr} and the Cauchy formula, the right hand 
side of \eqref{eq:eqmeas} is a probability measure on 
$\Delta$, say $\mu$. 
The differential along $\Delta_k^\circ$ of its 
logarithmic potential 
$U^\mu(z):=-\int_\Delta \log|z-t|d\mu(t)$  is 
${\rm Re}\{(\int_\Delta d\mu(t)/(t-z))dz\}={\rm Re}\{(a_0-z)dz/w^\pm(z)\}=0$,
using \eqref{eq:critr} and Cauchy's formula again.
Hence $U^\mu$ is constant on $\Delta$, which is characteristic 
of the equilibrium potential.}
that $\eqm$ has the form 
\begin{equation}
\label{eq:eqmeas}
d\eqm(t) = \frac{i(t-a_0)dt}{\pi w^+(t)}, \quad t\in\Delta.
\end{equation}
In fact, $\map$ has an integral representation involving $\eqm$, see \eqref{eq:map} in Section~\ref{Gdiff}. Set
\begin{equation}
\label{G}
G_h:=\exp\left\{\int\log hd\eqm\right\}.
\end{equation}
It is clear, since $\eqm(\Delta)=1$, that $G_h$ is well-defined as long as a continuous branch of $\log h$ is used, which is possible since $h$ is continuous and does not vanish on $\Delta$ and the latter is simply connected. Hereafter, we put for simplicity  $z_n=\pi(\z_n)$.
\begin{proposition}
\label{prop:snasymp}
In the setting of Proposition~\ref{prop:sn}
we have that
\begin{equation}
\label{eq:conjfunSn1}
\frac{(S_nS^*_n)(z)}{(\cp(\Delta))^{2n-1}} = \xi_nG_h
\left\{
\begin{array}{ll}
(z-z_n)/|\map(z_n)|, & \z_n\in D^{(2)}\setminus\{\infty^{(2)}\}, \bigskip \\
\cp(\Delta), & \z_n=\infty^{(2)}, \bigskip\\
(z-z_n)|\map(z_n)|, & \z_n\in L\cup D^{(1)}\setminus\{\infty^{(1)}\},
\end{array}
\right.
\end{equation}
where $|\xi_n|=1$. Moreover, it holds that
\begin{equation}
\label{eq:conjfunSn2}
\frac{S_n^*(z)}{S_n(z)} = \frac{\xi_nG_h}{\map^{2n-1}(z)}\Upsilon(\z_n;z)
\left\{
\begin{array}{ll}
\displaystyle \frac{z-z_n}{\map(z)|\map(z_n)|}, & \z_n\in D^{(2)}\setminus\{\infty^{(2)}\}, \bigskip \\
\cp(\Delta)/\map(z), & \z_n=\infty^{(2)}, \bigskip\\
\displaystyle \frac{\map(z)|\map(z_n)|}{z-z_n}, & \z_n\in L\cup D^{(1)}\setminus\{\infty^{(1)}\},
\end{array}
\right.
\end{equation}
where $S_n$ and $S_n^*$ are the conjugate functions 
derived from $S_n$ and $\{\Upsilon(\ar;\cdot)\}_{\ar\in\RS}$ is a normal 
family of non-vanishing functions in $D$.
\end{proposition}

For $\N_1$ an arbitrary subsequence of the natural numbers and ${\bf Z}_1$ the derived set of $\{\z_n\}_{n\in\N_1}$, the sequence $\{S_n^*/S_n\}_{n\in\N_1}$ converges to zero geometrically fast on closed subsets of $D\setminus\{\pi({\bf Z}_1\cap D^{(1)})\}$ by \eqref{eq:conjfunSn2} (recall that $|\map|>1$ in $D$). On  the contrary, no convergence can take place on domains intersecting
$\pi({\bf Z}_1\cap D^{(1)})$. Hence the convergence properties of $\{S_n^*/S_n\}$ depend on the geometry of ${\bf Z}={\bf Z}(h)$, the set of the limit points of~$\{\z_n\}$
in $\R$.

Our next proposition qualitatively describes this geometry. The classification according to the rational independence of the numbers $\eqm(\Delta_k)$ is essentially due to Suetin \cite{Suet03} whose argument, originally developed to handle the case of two arcs rather than a threefold, applies here with little change. We complete the picture with generic properties of this classification which
are intuitively as expected,  although their proof is not so straightforward.
\begin{proposition}
\label{prop:asympzn}
In the setting of Proposition~\ref{prop:sn} it holds that ${\bf Z}=\RS$ when the numbers $\eqm(\Delta_k)$, $k\in\{1,2,3\}$, are rationally independent;  ${\bf Z}$ is the union of finitely many pairwise disjoint arcs when $\eqm(\Delta_k)$ are rationally dependent but at least one of them  is irrational; ${\bf Z}$ is a finite set of points when $\eqm(\Delta_k)$ are all rational. All the points $\z_n$ are mutually distinct in the first two cases and $\{\z_n\}={\bf Z}$ in the third one. The set of triples $(a_1,a_2,a_3)$ for which the numbers $\eqm(\Delta_k)$ are rationally dependent form a dense subset of zero measure in $\C^3$. Triples $(a_1,a_2,a_3)$ for which  $\eqm(\Delta_k)$ are rational are also dense.
\end{proposition}
We prove these propositions in Section~\ref{sec:sn} with all the preliminary work carried out in Sections~\ref{sec:ci}, \ref{sec:rs} and \ref{sec:rhp}. Moreover, in these sections one can find integral representations for $S_n$ and 
$\Upsilon(\ar;\cdot)$.

\subsection{Pad\'e Approximation}

Let $f$ be a function holomorphic and vanishing at infinity. Then $f$ can be 
represented as a power series
\begin{equation}
\label{eq:f}
f(z) = \sum_{k=1}^\infty \frac{f_k}{z^k},
\end{equation}
which converges outside of some disk centered at the origin. 
A \emph{diagonal Pad\'e approximant} of order $n$ to $f$ 
is a rational function $\pi_n=p_n/q_n$ of type $(n,n)$  
such that
\begin{equation}
\label{eq:linsys}
q_n(z)f(z)-p_n(z) = \mathcal{O}\left(1/z^{n+1}\right) \quad \mbox{as} \quad z\to\infty
\end{equation}
System \eqref{eq:linsys} is always solvable since it consists of $2n+1$ homogeneous linear equations with $2n+2$ unknowns, whose
coefficients are the moments $f_k$ in \eqref{eq:f}, no solution of which can be such that $q_n\equiv0$ (we may thus assume that $q_n$ is monic). A solution needs not be unique, but each pair $(\tilde{p}_n,\tilde{q}_n)$  meeting  \eqref{eq:linsys} yields the same rational function $\pi_n=\tilde{p}_n/\tilde{q}_n$. In particular, each solution of \eqref{eq:linsys} is of the form $(lp_n,lq_n)$, where $(p_n,q_n)$ is the unique solution of minimal degree. That is, a Pad\'e approximant is the unique rational function $\pi_n$ of type 
$(n,n)$ satisfying  
\begin{equation}
\label{eq:pin}
f(z)-\pi_n(z) = \mathcal{O}\left(1/z^{n+1+\sigma(\pi_n)}\right) \quad \mbox{as} \quad z\to\infty.
\end{equation}
where $\sigma(\pi_n)$ is the number of finite poles of $\pi_n$, counting multiplicity (see\cite{Pade92,BakerGravesMorris}).
Hereafter, when writing $\pi_n=p_n/q_n$, we always mean that $(p_n,q_n)$ is the solution of minimal degree. In the generic case where ${\rm deg}\,q_n=n$, observe that the order of contact of $\pi_n$ with $f$ at infinity is $2n+1$, which is generically maximal possible as a rational function of type $(n,n)$ has $2n+1$ free parameters.  Notice here that $\deg(p_n)<\deg(q_n)$ as $f$ vanishes at infinity. Equivalently, one could also regard $\pi_n$ as a continued fraction of order $n$ constructed from the series representation \eqref{eq:f} \cite{Mar95}, but we shall not dwell on this connection.

Let now $h$ be a complex-valued integrable function given on $\Delta$. 
We define the Cauchy integral of $h$ as
\begin{equation}
\label{eq:CauchyT}
f_h(z) := \frac{1}{\pi i}\int_\Delta\frac{h(t)}{t-z}\frac{dt}{w^+(t)}, \quad z\in D,
\end{equation}
where integration is taking place according to the orientation of each $\Delta_k$, i.e., from $a_0$ to $a_k$. Clearly, $f_h$ is a holomorphic function in $D$ that vanishes at infinity and therefore can be represented as in \eqref{eq:f}. Thus, we can construct the sequence of Pad\'e approximants to $f_h$ whose asymptotic behavior is described by the following two theorems.

In the first theorem we assume that $h\equiv1/p$, where $p$ is a polynomial non-vanishing on $\Delta$. This is not only a key step in our approach to the
general case, but it is also of independent interest since this assumption
assumption allows us to obtain non-asymptotic formula for the approximation error.

\begin{theorem}
\label{thm:bernstein-szego}
Let $\{\pi_n\}$, $\pi_n=p_n/q_n$, be the sequence of diagonal Pad\'e approximants to $f_{1/p}$, where $p$ is a polynomial non-vanishing on $\Delta$. Then
\begin{equation}
\label{eq:bernstein-szego}
\left\{
\begin{array}{rl}
\displaystyle \left(f_{1/p}-\pi_n\right) & = \displaystyle \frac{2}{w}\frac{S_n^*}{S_n+pS_n^*},  \bigskip \\
q_n &  = S_n+pS_n^*,
\end{array}
\right.
\end{equation}
for all $2n>\deg(p)+2$, where $S_n$ and $S_n^*$ are the conjugate functions 
derived from $S_n$ granted by Proposition~\ref{prop:sn}.
\end{theorem}

As we show in Section~\ref{subs:error}, the polynomial $q_n$ is orthogonal 
(in the non-Hermitian sense) to 
all algebraic polynomials of degree at most $n-1$  with respect to the weight  $h/w^+$ on $\Delta$. Thus, polynomials $q_n$ appearing in 
Theorem~\ref{thm:bernstein-szego} stand analogous to the well-known Bernstein-Szeg\H{o} polynomials on $[-1,1]$ \cite[Sec. 2.6]{Szego}. Moreover, since $1/w$ vanishes at infinity,  Cauchy theorem yields that
\[
\frac{1}{w(z)} = \frac{1}{2\pi i}\int_\Gamma\frac{1}{w(t)}\frac{dt}{z-t}
\]
for $z$ in the exterior of $\Gamma$, where $\Gamma$ is any positively oriented Jordan curve encompassing $\Delta$. By deforming $\Gamma$ onto $\partial D$, one can easily verify that $1/w=f_1$ and therefore polynomials $q_n$ for the case $h\equiv1$ can be viewed as an analog of the classical Chebysh\"ev polynomials. 

The second line in \eqref{eq:bernstein-szego} yields that 
$q_n=S_n(1+pS^*_n/S_n)$, and checking the behavior at infinity using
Proposition \ref{prop:sn} gives us
$\deg(q_n)=n$ unless $\z_n=\infty^{(1)}$ in which case 
$\deg(q_n)=n-1$, $q_n=q_{n-1}$, and $\z_{n-1}=\infty^{(2)}$. 
Hence, when analyzing the behavior of $q_{n_k}$ 
along a subsequence $\{n_k\}\subset\N$, we may assume that 
$\z_{n_k}\neq \infty^{(1)}$ for all $k$ 
upon replacing $n_k$ by $n_k-1$ if necessary.

Turning now to more general densities, let
 $h$ be a Dini-continuous non-vanishing function. Denote by
\begin{equation}
\label{omegan}
\omega_n =\omega_n(h) := \min_p\|1/h-p\|_\Delta,
\end{equation}
where the minimum is taken over all polynomials $p$ of degree at most $n$.
Clearly $\omega_n\to0$ as $n\to\infty$ by Mergelyan's theorem.

\begin{theorem}
\label{thm:pade}
Let $h$ be a complex-valued Dini-continuous non-vanishing function on $\Delta$ and $\{\pi_n\}$, $\pi_n=p_n/q_n$, be the sequence of Pad\'e approximants to $f_h$. Then
\begin{equation}
\label{asymptotics}
\left(f_h-\pi_n\right) = \frac2w\frac{S_n^*}{S_n}\frac{1+E_n^*}{1+E_n+\mathcal{O}(|\map|^{-n})},
\end{equation}
where $\mathcal{O}(|\map|^{-n})$ holds uniformly in $D$ and $E_n$ is a sectionally meromorphic function on $\RS\setminus L$
with at most one pole which is necessarily $\z_n$, and such that
\begin{equation}
\label{errorintegral}
\left(\int_{\partial D}\left(|(E_nl_n)(t)|^2 + |(E_n^*l_n)(t)|^2\right)\frac{|dt|}{|w(t)|}\right)^{1/2} \leq \const\omega_n
\end{equation}
with $l_n(t)\equiv1$ when $\z_n\notin O$ and $l_n(t)=t-z_n$ otherwise,
where $O$ is some fixed but arbitrary neighborhood of $L$ in $\RS$
(the constant in \eqref{errorintegral} depends on $O$ but is 
independent of $n$). Moreover,  if $\z_n\neq\infty^{(1)}$ and $n$ is large enough then 
$\deg(q_n)=n$.
\end{theorem}

Formulae \eqref{asymptotics} and \eqref{errorintegral} have the following  ramifications for the uniform convergence of Pad\'e approximants.

\begin{corollary}
\label{convsubs}
Assumptions being as in Theorem~\ref{thm:pade}, let 
$\N_1\subset\N$ be a subsequence such that $\{\z_n\}_{n\in\N_1}$ converges to 
$\z\in\RS$. 
\begin{itemize}
\item If $\z\in D^{(2)}\cup L$, then the Pad\'e approximants 
$\pi_n$ converge to $f$ 
geometrically fast on compact subsets of $D$ as $\N_1\ni n\to\infty$. 
\item If $\z\in D^{(1)}$, then the Pad\'e approximants $\pi_n$ converge 
to $f$ 
geometrically fast on compact subsets of $D\setminus \{z\}$
as $\N_1\ni n\to\infty$. Moreover, to each neighborhood $O$ of $\Delta$ in
$\C$, there is $n_O\in\N_1$ such that $\pi_n$ has exactly one pole
in $D\setminus O$ for $n\geq n_O$ and this pole converges to $z$
as $\N_1\ni n\to\infty$.
\end{itemize}
\end{corollary}

Theorems~\ref{thm:bernstein-szego},~\ref{thm:pade}, and Corollary~\ref{convsubs} are proven in Section~\ref{sec:proofs}. The following is an immediate consequence of Corollary~\ref{convsubs}.

\begin{corollary}
Under conditions of Theorem~\ref{thm:pade}, let  $\{\pi_n\}_{n\in\N'\subset\N}$ be a sequence of diagonal Pad\'e  approximants to $f$ and $\mathbf{Z}$ the set of accumulation points of $\{\z_n\}_{n\in\N'}$ on $\RS$. Then  $\{\pi_n\}_{n\in\N'}$ converges locally uniformly to $f$ on $D\setminus\pi(\mathbf{Z}\cap D^{(1)})$ and on no larger subdomain of $D$.
\end{corollary}

Our last theorem puts the preceding results in a generic perspective.

\begin{theorem}
There is dense subset of zero measure $E\subset\C^3$ such that, for $f$ as in \eqref{eq:CauchyT} with $\Delta$ the Chebotar\"ev continuum of $(a_1,a_2,a_3)\in \C^3$ and $h$ a Dini-continuous non-vanishing function on $\Delta$, the following holds:
\begin{itemize}
\item if $(a_1,a_2,a_3)\in \C^3\setminus E$, then the sequence of Pad\'e approximants to $f$ converges on no subdomain of $\overline{\C}\setminus\Delta$ but some subsequence converges locally uniformly to $f$ on $\overline{\C}\setminus\Delta$;
\item if $(a_1,a_2,a_3)\in E$, then the sequence of Pad\'e approximants to $f$ converges locally uniformly on $\overline{\C}\setminus(\Delta\cup A)$, where $A$ is either a finite (possibly empty) union of curves or finitely many points. In particular the domain of convergence is non-void.
\end{itemize}
\end{theorem}

\section{Cauchy Integrals}
\label{sec:ci}

For an analytic Jordan arc $F$ with endpoints $e_1$ and $e_2$, oriented from $e_1$ to $e_2$, define
\begin{equation}
\label{eq:wL}
w_F(z) := \sqrt{(z-e_1)(z-e_2)}, \quad \frac{w_F(z)}{z} \to 1 \quad \mbox{as} \quad z\to\infty,
\end{equation}
to be a holomorphic function outside of $F$ with a simple pole at infinity. Then $w_F$ has continuous traces $w_F^+$ and $w_F^-$ on the left and right sides of $F$, respectively (sides are determined by the orientation in the usual manner). For an integrable function $\phi$ on $F$, set
\begin{equation}
\label{eq:RL}
C_F(\phi;z) := \int_F\frac{\phi(t)}{t-z}\frac{dt}{2\pi i} \quad \mbox{and} \quad R_F(\phi;z) := w_F(z)C_F\left(\frac{\phi}{w_F^+};z\right),
\end{equation}
$z\in\overline\C\setminus F$. We also put
\begin{equation}
\label{eq:RDelta}
C_\Delta(z):= \int_\Delta\frac{\theta(t)}{t-z}\frac{dt}{2\pi i} \quad \mbox{and} \quad R_\Delta(\theta;z) := w(z)C_\Delta\left(\frac{\theta}{w^+};z\right),
\end{equation}
$z\in D$, where $\theta$ is an integrable function on $\Delta$. The following lemma will be needed later on.

\begin{lemma}
\label{lem:A2weight}
Let $\theta$ be a Dini-continuous function on $\Delta$ and $J$ be either $C_\Delta^\pm$ or $R_\Delta^\pm$. Then
\begin{equation}
\label{eq:BK}
\int_\Delta\frac{|J(t)|^2}{|w^+(t)|}|dt| \leq \const \int_\Delta\frac{|\theta(t)|^2}{|w^+(t)|}|dt|,
\end{equation}
\end{lemma}
\begin{proof}
It was shown in \cite[Sec. 3.2]{BY09c} that for a Dini-continuous function $\phi$ on $\Delta_k$, the functions $R_{\Delta_k}(\phi;\cdot)$ and $C_{\Delta_k}(\phi;\cdot)$ have unrestricted boundary values on both sides of $\Delta_k$ and the traces $R_\Delta^\pm(\phi;\cdot)$ and $C_{\Delta_k}^\pm(\phi;\cdot)$ are continuous. It is known \cite[Thm. 2.2]{BottcherKarlovich} that $|w_{\Delta_k}|^{-1/2}$ is an $A_2$-weight on $\Delta_k$. Hence, by \cite[Thm. 4.15]{BottcherKarlovich} it follows that
\begin{equation}
\label{eq:BK1}
\int_{\Delta_k}\frac{\left|C_{\Delta_k}^\pm(\phi;t)\right|^2}{|w_{\Delta_k}^+(t)|}|dt| \leq \const \int_{\Delta_k}\frac{|\phi(t)|^2}{|w_{\Delta_k}^+(t)|}|dt|,
\end{equation}
where $\const$ is a constant independent of $\phi$. By the very definition, see \cite[eq. (2.1)]{BottcherKarlovich}, $|w_{\Delta_k}|^{-1/2}$ is an $A_2$-weight if and only if $|w_{\Delta_k}|^{1/2}$ is also an $A_2$-weight. Thus, applying  \cite[Thm. 4.15]{BottcherKarlovich} as in \eqref{eq:BK1} only with $|w_{\Delta_k}|^{-1/2}$ replaced by $|w_{\Delta_k}|^{1/2}$ and $\phi$ replaced by $\phi/w_{\Delta_k}^+$, we get that
\begin{equation}
\label{eq:BK2}
\int_{\Delta_k}\frac{\left|R_{\Delta_k}^\pm(\phi;t)\right|^2}{|w_{\Delta_k}^+(t)|}|dt| \leq \const \int_{\Delta_k}\frac{|\phi(t)|^2}{|w_{\Delta_k}^+(t)|}|dt|.
\end{equation}
Moreover, \cite[Prop. 2.1]{BottcherKarlovich} yields that not only $|w_{\Delta_k}|^{\pm1/2}$ but also $|w|^{\pm1/2}$ is an $A_2$-weight on each $\Delta_k$. Thus, \eqref{eq:BK} is obtained by applying \eqref{eq:BK1} and \eqref{eq:BK2} on each $\Delta_k$ with $\phi=\theta_{|\Delta_k}$ and then taking the sum over $k$.
\end{proof}

The main purpose of this section is to study the so-called \emph{Szeg\H{o} function} of a given function on $\Delta$. Namely, let $h$ be a Dini-continuous non-vanishing function on $\Delta$. For a fixed $\kappa\in\{1,2,3\}$ and an arbitrary continuous branch of $\log h$, we define the constant $G_{h,\kappa}$ as
\begin{equation}
\label{eq:gm}
G_{h,\kappa} := \exp\left\{-m_1+\frac{\beta^1_\kappa}{\beta_\kappa}m_0\right\}, \quad m_j:=\frac{1}{\pi i}\int_\Delta\frac{t^j\log h(t)dt}{w^+(t)},
\end{equation}
and the \emph{Szeg\H{o} function} of $h$ as
\begin{equation}
\label{eq:szegofun}
S_{h,\kappa}(z):=\exp\left\{w(z)\left(C_\Delta\left(\frac{\log h}{w^+};z\right)-\frac{m_0}{\beta_{\kappa}}C_{\Delta_{\kappa}}\left(\frac{1}{w^+};z\right)\right)-\frac{\log G_{h,\kappa}}{2}\right\},
\end{equation}
$z\in D$, where under $\log G_{h,\kappa}$ we understand $-m_1+m_0(\beta^1_\kappa/\beta_\kappa)$,
\begin{equation}
\label{eq:betas}
\beta_k := \frac{1}{\pi i}\int_{\Delta_k}\frac{dt}{w^+(t)} \quad \mbox{and} \quad \beta^1_k := \frac{1}{\pi i}\int_{\Delta_k}\frac{tdt}{w^+(t)}, \quad k\in\{1,2,3\}.
\end{equation}
As $1/w(z)=1/z^2+\dots$, it holds by the Cauchy integral formula that
\[
\beta_1+\beta_2+\beta_3 = \frac{1}{2\pi i}\int_{\partial D}\frac{dt}{w(t)} = 0  \quad \mbox{and} \quad  \beta_1^1+\beta_2^1+\beta_3^1 = \frac{1}{2\pi i}\int_{\partial D}\frac{tdt}{w(t)} = -1.
\] 
\begin{proposition}
\label{prop:szego}
For $h$ and $\kappa$ as above, the constant $G_{h,\kappa}$ and the function 
$S_{h,\kappa}$ do not depend on the continuous branch of $\log h$ 
used to define them through \eqref{eq:gm} and \eqref{eq:szegofun}.  
Moreover $S_{h,\kappa}$ is  holomorphic in $D$, 
it has continuous boundary values on $\partial D$,  and
$S_{h,\kappa}(\infty)=1$.
\begin{equation}
\label{eq:szego}
h = \left\{
\begin{array}{ll}
\widetilde G_{h,\kappa}S_{h,\kappa}^+S_{h,\kappa}^- &  \mbox{on} \quad \Delta_\kappa^\circ, \bigskip \\
G_{h,\kappa}S_{h,\kappa}^+S_{h,\kappa}^-, & \mbox{on} \quad \Delta^\circ\setminus\Delta_\kappa,
\end{array}
\right. \quad \widetilde G_{h,\kappa}:=G_{h,\kappa}\exp\left\{\frac{m_0}{\beta_\kappa}\right\}.
\end{equation}
\end{proposition}
\begin{proof}
As we mentioned in Lemma~\ref{lem:A2weight}, it follows from \cite[Sec. 3.2]{BY09c} that for any Dini-continuous function $\phi$ on an analytic arc $F$,  $R_F(\phi;\cdot)$ has unrestricted boundary values on both sides of $F$, the traces $R_F^\pm(\phi;\cdot)$ are continuous, $R_F^+(\phi;e_k)=R_F^-(\phi;e_k)$, $k\in\{1,2\}$, and furthermore
\begin{equation}
\label{eq:SP}
R_F^+(\phi;t)+R_F^-(\phi;t) = \phi(t), \quad t\in F,
\end{equation}
where \eqref{eq:SP} is a consequence of Sokhotski-Plemelj formulae and the relation $w_F^+=-w_F^-$. 

Let now $\theta$ be a Dini-continuous function on $\Delta$. Observe that $w=w_{\Delta_k}w_{F_k}$ according to \eqref{eq:wL}, where $F_k:=(\Delta\setminus\Delta_k)\cup\{a_0\}$. Then
\[
R_\Delta(\theta;z) = \sum_{k=1}^3 w_{F_k}(z)R_{\Delta_k}\left(\frac{\theta}{w_{F_k}};z\right)
\]
according to \eqref{eq:RL}. This immediately implies that $R_\Delta(\theta;\cdot)$ has continuous trace on $\partial D$ with
\begin{equation}
\label{eq:onDelta}
R_\Delta^+(\theta;t)+R_\Delta^-(\theta;t) = \theta(t), \quad t\in\Delta,
\end{equation}
by \eqref{eq:SP}. Moreover, applying the Cauchy integral formula to $1/w$ on $\partial D$, we get
\begin{equation}
\label{eq:addconst1}
R_\Delta(\const;z) = \frac{\const}{2}
\end{equation}
for any constant. To describe the behaviour of $R_\Delta(\theta;\cdot)$ 
at infinity, define the moments
\begin{equation}
\label{eq:mom}
m_k = m_k(\theta) := \frac{1}{\pi i}\int_\Delta\frac{t^k\theta(t)}{w^+(t)}dt, \quad k\in\{0,1\}.
\end{equation}
Using the fact that $1/w=1/z^2+O(1/z^3)$ near infinity, one can readily verify that
\begin{equation}
\label{eq:addconst2}
m_0(\theta+\const) = m_0(\theta) \quad \mbox{and} \quad m_1(\theta+\const) = m_1(\theta)+\const
\end{equation}
for any constant. By developing $1/(t-z)$ at infinity in powers of $z$, we get that
\begin{equation}
\label{eq:nearinfty1}
R_\Delta(\theta;z) = w(z)\left(-\frac{m_0}{2z}-\frac{m_1}{2z^2}+O\left(\frac{1}{z^3}\right)\right)
\end{equation}
there. Analogously, one can check that
\begin{equation}
\label{eq:nearinfty2}
w_{F_\kappa}(z)R_{\Delta_\kappa}\left(\frac{1}{w_{F_\kappa}};z\right) = w(z)\left(-\frac{\beta_\kappa}{2z}-\frac{\beta^1_\kappa}{2z^2}+O\left(\frac{1}{z^3}\right)\right)
\end{equation}
near infinity. Thus,
\begin{eqnarray}
R_\kappa(z) &:=& R_\Delta(\theta;z)-\frac{m_0}{\beta_\kappa}w_{F_\kappa}(z)R_{\Delta_\kappa}\left(\frac{1}{w_{F_\kappa}};z\right) \nonumber  \\
\label{eq:nearinfty3} {} &=& w(z)\left(\frac{1}{2z^2}\left(m_0\frac{\beta^1_\kappa}{\beta_\kappa}-m_1\right)+O\left(\frac{1}{z^3}\right)\right) = \frac12\left(m_0\frac{\beta^1_\kappa}{\beta_\kappa}-m_1\right) + O\left(\frac1z\right)
\end{eqnarray}
near infinity by \eqref{eq:nearinfty1} and \eqref{eq:nearinfty2}. Moreover, it follows from \eqref{eq:SP} and \eqref{eq:onDelta} that
\begin{equation}
\label{eq:jump}
R_\kappa^++R_\kappa^- = \theta - \left\{
\begin{array}{ll}
m_0/\beta_\kappa, & \mbox{on} \quad \Delta^\circ_\kappa, \smallskip \\
0, & \mbox{on} \quad \Delta^\circ\setminus\Delta_\kappa.
\end{array}
\right.
\end{equation}

Finally, let $h$ be a Dini-continuous non-vanishing function on $\Delta$. As explained in \cite[Sec. 3.3]{BY09c}, any continuous branch of $\log h$ is itself
Dini-continuous. Fix such a branch and denote it by $\theta$. Observe that the difference between any two continuous determinations of $\log h$ is an integer multiple of $2\pi i$ and therefore $G_{h,\kappa}$ is well-defined by \eqref{eq:addconst2}. Moreover, the Szeg\H{o} function of $h$ defined in \eqref{eq:szegofun} is nothing else but
\begin{equation}
\label{eq:Sh}
\exp\left\{R_\kappa(z)-\frac12\left(m_0\frac{\beta^1_\kappa}{\beta_\kappa}-m_1\right)\right\}, \quad z\in D.
\end{equation}
As evident from \eqref{eq:addconst1} and \eqref{eq:addconst2}, $S_{h,\kappa}$ does not depend on the choice of the branch of $\log h$ as long as the branch is continuous and used in \eqref{eq:RDelta} and \eqref{eq:mom} simultaneously. Clearly, \eqref{eq:szego} follows from \eqref{eq:jump} and $S_{h,\kappa}(\infty)=1$ by \eqref{eq:nearinfty3}. The continuity of $S_{h,\kappa}$ on $\partial D$ is a consequence of continuity of $R_\kappa$ on $\partial D$. Obviously, $S_{h,\kappa}$ is holomorphic and non-vanishing in $D$ as it is an exponential of a holomorphic function.
\end{proof}

\section{Abelian Differentials and Their Integrals}
\label{sec:rs}

The following material is expository on  Abelian differentials 
on an elliptic Riemann surface. We use \cite{Bliss,Forster} as primary 
sources, limiting ourselves to the case at hand ({\it i.e.} $\RS$).

For each $k\in\{1,2,3\}$ set $\widetilde\RS_k:=\RS\setminus(L_k\cup L_{k+1})$ and $\widehat\RS_k:=\RS\setminus L_k$, where indices are computed modulo 3. 
It is easy to see ({\it cf.} Figure~\ref{fig:torus}), that each domain $\widetilde\RS_k$ is simply connected.

\subsection{Abelian Differentials of the First Kind}
A differential $d\Omega$ is called an \emph{Abelian differential of the first kind} on $\RS$ if the integral $\int^\z d\Omega$ defines a holomorphic multi-valued function on the whole surface. Since the genus of $\RS$ is 1, there exists exactly one Abelian differential of the first kind up to a multiplicative constant. This differential is given by
\[
d\Omega(\z):=\frac{dz}{w(\z)}
\]
as the principal divisor of $d\Omega$ should be integral and since it is known that the principal divisor of the differential $dz$ is given by $\sum_{k=0}^3a_k-2\infty^{(1)}-2\infty^{(2)}$. By $d\Omega_1$ we denote the Abelian differential of the first kind normalized to have period 1 on $L_2$ (i.e., we choose $L_2$ to be the so-called {\bf a}-cycle for $d\Omega$). That is,
\begin{equation}
\label{eq:adfk}
d\Omega_1(\z):=\frac{1}{2\pi i\beta_2}\frac{dz}{w(\z)}, \quad \beta_k=\frac{1}{\pi i}\int_{\Delta_k}\frac{dt}{w^+(t)}=\frac{1}{2\pi i}\oint_{L_k}d\Omega,
\end{equation}
$k\in\{1,2,3\}$. Moreover, it is known that
\begin{equation}
\label{eq:bw}
\im\left(\frac{\beta_3}{\beta_2}\right)>0, \quad \frac{\beta_3}{\beta_2}=\oint_{L_3}d\Omega_1,
\end{equation}
because $(L_2,L_3)$ is  positively oriented ({\it i.e.} we take $L_3$ to be the so-called {\bf b}-cycle for $d\Omega$). The numbers $1$ and $\beta_3/\beta_2$ are called the \emph{periods} of $d\Omega_1$. It is known that for any Jordan curve $\Gamma$ on $\RS$ the integral of $d\Omega_1$ along $\Gamma$ is congruent to 0 ($\equiv 0$) modulo periods of $d\Omega_1$. That is,
\[
\oint_\Gamma d\Omega_1=l+j\frac{\beta_3}{\beta_2}, \quad l,j\in\Z.
\]
It will be convenient for us to define\footnote{It is formally more appropriate but also more cumbersome to denote $\Omega_1$ by $\Omega_{1,3}$, where $\Omega_{1,k}$ is defined as in \eqref{eq:Omega1} using the differential of the first kind that has $L_{k-1}$ as the {\bf a}-cycle, $L_k$ as the {\bf b}-cycle, and $a_{k+1}$ as the initial bound for integration. This comment applies to all the differentials below where we do not explicitly specify the dependence on the choice of cycles.}
\begin{equation}
\label{eq:Omega1}
\Omega_1(\z) := \int_{a_1}^\z d\Omega_1, \quad \z\in\widetilde\RS_2,
\end{equation} 
where the path of integration except perhaps for the endpoint lies entirely in $\widetilde\RS_2$. Observe that $\Omega_1$ is a well-defined holomorphic function in the simply connected domain $\widetilde\RS_2$  since $1/w$ has a double zero at infinity. 
\begin{figure}[!ht]
\centering
\includegraphics[scale=.4]{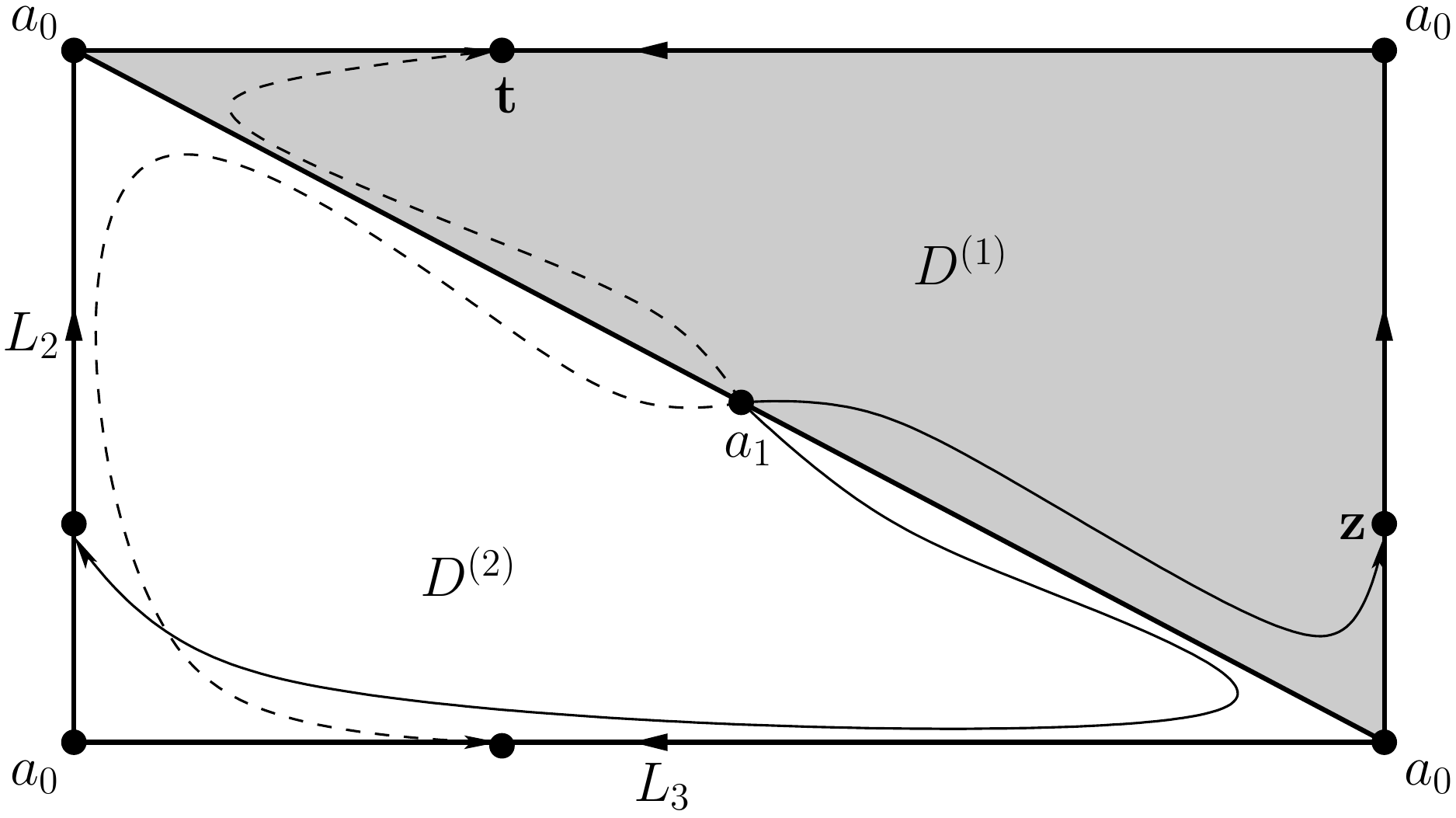}
\caption{\small Paths of integration of $d\Omega_1$ that start at $a_1$ and end at $\z\in L_2$ (solid lines) and $\tr\in L_3$ (dashed line).}
\label{fig:jumpOmega1}
\end{figure}
Furthermore, $\Omega_1$ has continuous traces on both sides of $L_2$ and $L_3$ and the jump of $\Omega_1$ there can be described by the relations
\begin{equation}
\label{eq:bvOmega1}
\Omega_1^+-\Omega_1^- = \left\{
\begin{array}{ll}
-\beta_3/\beta_2, & \mbox{on} \quad L_2, \smallskip \\
1, & \mbox{on} \quad L_3,
\end{array}
\right.
\end{equation}
as can be seen from Figure~\ref{fig:jumpOmega1} (as shown on the figure, the dashed path in $D^{(2)}$ can be deformed into a concatenation of the dashed path in $D^{(1)}$ and the loop $L_2$ traversed in the negative direction; the solid path in $D^{(2)}$ can be deformed into a concatenation of the solid path in $D^{(1)}$ and the loop $L_3$ traversed in the positive direction).

\subsection{Abelian Differentials of the Third Kind}
\label{subs:adthk}

An arbitrary Abelian differential is a differential of the form $rd\Omega$, where $r$ is a rational function over $\RS$. The principal divisor of $rd\Omega$ coincides with the principal divisor of $r$. Thus, $rd\Omega$ has only poles as singularities and the residue of $rd\Omega$ at a pole $\ar$ is equal to $\frac{1}{2\pi i}\oint_{\Gamma_\ar} rd\Omega$, where $\Gamma_\ar$ is a Jordan curve on $\RS$ that separates $\ar$ from the rest of the poles of $rd\Omega$ and is oriented so that $\ar$ lies to left of $\Gamma_\ar$ when the latter is traversed in the positive direction. 

In what follows, we are primarily interested in the following rational function:
\begin{equation}
\label{eq:raz}
r(\ar;\z) := \frac12\left(\frac{w(\z)+w(\ar)}{z-a}+z-a+A\right), \quad A:=\frac12\sum_{j=0}^3a_j,
\end{equation}
$\ar\in\RS$, $|a|<\infty$. Since $w(z) = z^2-Az+\cdots$ at infinity, it is an easy computation to verify by taking the appropriate limits that $r(\ar;\infty^{(2)}) = A-a$ and that $r(\infty^{(2)};\z) \equiv z$. That is, $r(\ar;\cdot)$ is, in fact, defined for all $\ar\in\RS\setminus\{\infty^{(1)}\}$ and is bounded near $\infty^{(2)}$ for all $\ar$ with finite canonical projection. Moreover, 
\begin{equation}
\label{eq:locallyto0}
r(\ar;\z)+a-A \rightrightarrows 0\quad \mbox{as} \quad \ar\to\infty^{(1)}
\end{equation}
in $\RS\setminus\{\infty^{(1)}\}$, where the sign $\rightrightarrows$ means ``converges locally uniformly''. Analogously, 
we get that
\begin{equation}
\label{eq:locallyto0again}
r(\br;\z)-r(\ar;\z) \rightrightarrows 0\quad \mbox{as} \quad \br\to\ar
\end{equation}
in $\RS\setminus\{\ar\}$. Summarizing, we have that
\begin{equation}
\label{eq:auxdiff}
r(\ar;\z)d\Omega(\z), \quad \ar\in\RS\setminus\{\infty^{(1)}\},
\end{equation}
defines a differential with two poles, $\infty^{(1)}$ and $\ar$. Let $\Gamma_{\infty^{(1)}}$ be a Jordan curve in $D^{(1)}$ that encompasses $\infty^{(1)}$ and separates it from $\ar$. Assuming that $\Gamma_{\infty^{(1)}}$ is oriented clockwise, we can compute the residue of $r(\ar;\z)d\Omega(\z)$ at $\infty^{(1)}$ as
\begin{equation}
\label{eq:computeresidue}
\frac{1}{2\pi i}\int_{\Gamma_{\infty^{(1)}}}r(\ar;\z)d\Omega(\z) = \frac{1}{4\pi i}\int_{\pi(\Gamma_{\infty^{(1)}})}\frac{dz}{z-a} + \frac{1}{4\pi i}\int_{\pi(\Gamma_{\infty^{(1)}})}\frac{zdz}{w(z)} = -1,
\end{equation}
where all the integrals are evaluated by the Cauchy integral formula for unbounded domains. Respectively, the residue of $r(\ar;\z)d\Omega(\z)$ at $\ar$ is equal to $1$. 

More generally, given two distinct points $\br_1$ and $\br_2$ on $\RS$, 
there is a differential $d\Omega(\br_1,\br_2;\z)$ called an 
\emph{Abelian differential of the third kind} having only two simple poles, 
$\br_1$ and $\br_2$, with residues $1$ and $-1$, respectively. Such a differential is unique up to a differential of the first kind. Thus, for $\br_1,\br_2\in\widetilde\RS_2$, there exists a unique differential of the third kind with 
period $0$ on $L_2$, that we denote by $d\Omega_0(\br_1,\br_2;\z)$ for 
brevity. It is known for such a differential that the $L_3$-period can be expressed through the Riemann relation:
\begin{equation}
\label{eq:riemannrelation}
\oint_{L_3}d\Omega_0(\br_1,\br_2;\z) = -2\pi i \int_{\br_1}^{\br_2}d\Omega_1(\z),
\end{equation}
where the path of integration for the integral on right-hand side of \eqref{eq:riemannrelation} lies entirely in $\widetilde\RS_2$. We shall also use another relation between the normalized Abelian integrals of the third kind, namely,
\begin{equation}
\label{eq:rr}
\int_{\br_1}^{\br_2}d\Omega_0(\br_3,\br_4;\z) = \int_{\br_3}^{\br_4}d\Omega_0(\br_1,\br_2;\z)
\end{equation}
for $\br_k\in\widehat\RS_2$, $k\in\{1,2,3,4\}$, where the paths of integration again lie in $\widetilde\RS_2$.

Assume now that at least one of $\br_1,\br_2$ belongs to $L_2\cup L_3$. 
Let $L_k^\prime$, $k\in\{2,3\}$, be two Jordan curves on $\RS$ intersecting
each other and $L_1$ once at the same point. Assume further that each $L_k^\prime$ is homologous to $L_k$ and coincides with the latter except in a neighborhood of $\br_j$ if $\br_j\in L_k$ where they are disjoint. In particular, the periods of $d\Omega$ remain the same on these new curves. For definiteness, we suppose that those of the points $\br_1,\br_2$ belonging to $L_2\cup L_3$ lie to the left of $L_2^\prime$ and $L_3^\prime$. Then \eqref{eq:riemannrelation} remains valid only with $L_3$ replaced by $L_3^\prime$, $d\Omega_0(\br_1,\br_2;\cdot)$ normalized to have zero period on $L_2^\prime$, and the path of integration for $d\Omega$ taken in $\RS\setminus\{L_2^\prime\cup L_3^\prime\}$. Clearly, \eqref{eq:rr} also holds only with the differentials of the third kind normalized to have zero period on $L_2^\prime$ and the paths of integration taken to lie in $\RS\setminus\{L_2^\prime\cup L_3^\prime\}$.

\subsection{Differentials $d\Omega_0(\ar,\infty^{(1)};\cdot)$}

In Section~\ref{sec:rhp}, we shall mainly work with differentials of the form $d\Omega_0(\ar,\infty^{(1)};\z)$. It easily follows from \eqref{eq:auxdiff} that
\begin{equation}
\label{eq:dOmega0inf1}
d\Omega_0(\ar,\infty^{(1)};\z) = \left(r(\ar;\z) +c(\ar)\right) d\Omega(\z),
\end{equation}
where the constant $c(\ar)$ is chosen so that the period on $L_2$ (or $L_2^\prime$) of the differential is equal to 0 and is clearly a continuous function of $\ar$. 
Moreover, it readily follows from a computation analogous to \eqref{eq:computeresidue} that $c(\ar)=a-A+\mathcal{O}(1/a)$ for $\ar$ in the vicinity of $\infty^{(1)}$. That is, $d\Omega_0(\ar,\infty^{(1)};\z)$ degenerates into a zero differential as $\ar\to\infty^{(1)}$ by \eqref{eq:locallyto0}.  

It is useful to observe that a general Abelian differential of the third kind is given by
\begin{equation}
\label{eq:generalthirdkind}
d\Omega_0(\br_1,\br_2;\z) = d\Omega_0(\br_1,\infty^{(1)};\z) - d\Omega_0(\br_2,\infty^{(1)};\z).
\end{equation}
Clearly, \eqref{eq:generalthirdkind} also provides a rational function representation for $d\Omega_0(\br_1,\br_2;\z)$ via \eqref{eq:dOmega0inf1}.

For $\ar\in\RS\setminus(L_2\cup\{a_1,\infty^{(1)}\})$, set
\begin{equation}
\label{eq:OmegaArF}
\Omega_0(\ar;\z) := \int_{a_1}^\z d\Omega_0(\ar,\infty^{(1)};\tr), \quad \z\in\widetilde\RS_2,
\end{equation}
where the path of integration, as usual, lies in $\widetilde\RS_2$ (or $\RS\setminus(L_2\cup L_3^\prime)$ when $\ar\in L_3$). Then $\Omega_0(\ar;\cdot)$ is 
analytic and multi-valued (single-valued modulo $2\pi i$) on 
$\widetilde\RS_2\setminus\{\ar,\infty^{(1)}\}$ (or $\RS\setminus(L_2\cup L_3^\prime\cup\{\ar,\infty^{(1)}\}$) with logarithmic singularities at $\ar$ and $\infty^{(1)}$. Moreover, analyzing the boundary behavior of $\Omega_0(\ar;\cdot)$ on $L_2$ and $L_3$ (or $L_3^\prime$) as in \eqref{eq:bvOmega1}, we see that $\Omega_0(\ar;\cdot)$ is analytic and multi-valued (single-valued modulo $2\pi i$) in $\widehat\RS_2\setminus\{\ar,\infty^{(1)}\}$, and that on $L_2$ it has the following jump:
\begin{equation}
\label{eq:bvOmegaa}
\Omega^+_0(\ar;\cdot)-\Omega^-_0(\ar;\cdot) \equiv -\oint_{L_3}d\Omega_0(\ar,\infty^{(1)};\tr) = 2\pi i\left(\Omega_1(\infty^{(1)})-\Omega_1(\ar)\right) \quad \mbox{(mod } 2\pi i),
\end{equation}
where the second equality follows from \eqref{eq:riemannrelation} and \eqref{eq:Omega1}. 

For $\ar=\infty^{(1)}$, we formally set
\begin{equation}
\label{eq:OmegaInfF}
\Omega_0(\infty^{(1)};\cdot):\equiv0 \leftleftarrows \Omega_0(\tr;\cdot) \quad (\mbox{mod }2\pi i) \quad \mbox{as} \quad \tr\to\infty^{(1)},
\end{equation}
where convergence holds locally uniformly in $\RS\setminus\{\infty^{(1)}\}$ by \eqref{eq:locallyto0}. Observe that under this convention \eqref{eq:bvOmegaa} still remains valid. 

For $\ar=a_1$, we simply change the initial bound of integration to some $b\in L_1\setminus\{a_1,a_0\}$. Clearly, \eqref{eq:bvOmegaa} remains valid in this case as well.

For $\ar\in L_2$, the construction of $\Omega_0(\ar;\cdot)$ is as follows. 
Define $\tilde\Omega_0(\ar;\cdot)$ as in \eqref{eq:OmegaArF} 
with $L_2$ replaced by any admissible $L_2^\prime$. This function is analytic 
and multi-valued in $\RS\setminus(L_2^\prime\cup\{\ar,\infty^{(1)}\})$  
and has a jump across $L_2^\prime$ whose magnitude is described by \eqref{eq:bvOmegaa}. Observe that the magnitude of the jump does not depend on the choice of $L_2^\prime$ and that for any $\z\in D^{(1)}$ ($\z\in D^{(2)}$) the curve $L_2^\prime$ can be chosen not to separate $\z$ and $\infty^{(1)}$ ($\z$ and $\infty^{(2)}$). Hence, $\tilde\Omega_0(\ar;\cdot)$ can be analytically continued to an analytic multi-valued function in $\widehat\RS_2\setminus\{\ar,\infty^{(1)}\}$, and we set $\Omega_0(\ar;\cdot)$ to be this function. Notice that \eqref{eq:bvOmegaa} is still 
at every point of $L\setminus\{\ar\}$.

One more important property of functions $\Omega_0(\ar;\cdot)$ is that
\begin{equation}
\label{eq:locallyuniform}
\Omega_0(\tr;\z) - \Omega_0(\ar;\z) \rightrightarrows 0 \quad (\mbox{mod }2\pi i) \quad \mbox{as} \quad \tr\to\ar
\end{equation}
in $\RS\setminus\{\ar\}$ for any $\ar\in\RS$, which follows from \eqref{eq:dOmega0inf1} and \eqref{eq:generalthirdkind} combined with \eqref{eq:locallyto0again}.

\subsection{Green Differential}
\label{Gdiff}

Another important way to normalize a differential of the third kind is to make its periods to be purely imaginary. For instance, we shall be interested in the so-called \emph{Green differential} given by
\[
dG(\z) := (z-a_0)d\Omega(\z) = \frac{(z-a_0)dz}{w(\z)}.
\]
Computing as in \eqref{eq:computeresidue}, one can easily check that 
$dG$ is a differential of the third kind having simple poles at $\infty^{(1)}$ 
and $\infty^{(2)}$ with residues $-1$ and $1$ respectively.
Moreover, by \eqref{eq:eqmeas},
\begin{equation}
\label{eq:ms-green}
\oint_{L_k} dG = -2\pi i\eqm(\Delta_k)=:-\omega_k,
\end{equation}
where, as before, $\eqm$ is the equilibrium measure on $\Delta$. In particular, it follows from \eqref{eq:betas} that
\begin{equation}
\label{eqandbetas}
\eqm(\Delta_k)=a_0\beta_k-\beta_k^1.
\end{equation}
Observe also that
\begin{equation}
\label{eq:relation}
d\Omega_0(\infty^{(2)},\infty^{(1)};\z) = dG(\z) +\omega_2d\Omega_1(\z)
\end{equation}
by uniqueness of a normalized Abelian differential of the third kind
with prescribed poles. 

Set
\begin{equation}
\label{eq:map}
\map_{a_1}(\z) := \exp\left\{\int^\z_{a_1}dG\right\}, \quad \z\in\widetilde\RS_2.
\end{equation}
Then $\map_{a_1}$ is a well-defined meromorphic function in $\widetilde\RS_2$ (the integral is defined modulo $2\pi i$) with a simple pole at $\infty^{(1)}$, a simple zero at $\infty^{(2)}$, otherwise non-vanishing, and with unimodular traces on $L_2\cup L_3$ that satisfy
\begin{equation}
\label{eq:greenboundary}
\frac{\map_{a_1}^+}{\map_{a_1}^-} = \left\{
\begin{array}{ll}
\exp\{\omega_3\}, & \mbox{on} \quad L_2, \smallskip \\
\exp\{-\omega_2\}, & \mbox{on} \quad L_3,
\end{array}
\right.
\end{equation}
where we obtain \eqref{eq:greenboundary} exactly as we derived \eqref{eq:bvOmega1}. In fact, $\map_{a_1}$ is the conformal map of $D$ onto $\{|z|>1\}$, $\map_{a_1}(a_1)=1$. It is known that
\begin{equation}
\label{eq:saympmap}
z/\map_{a_1}(z^{(1)})=z\map_{a_1}(z^{(2)})\to\xi_{a_1}\cp(\Delta) \quad \mbox{as} \quad z\to\infty,
\end{equation}
where $|\xi_{a_1}|=1$ and $\cp(\Delta)$ is the logarithmic capacity 
of $\Delta$. Here we indicate the dependence on the choice of the cycles 
and of the initial point of integration in \eqref{eq:map},
so that $\map_{a_1}$ will not be 
confused with $\map$ defined in \eqref{map}. Clearly,
\begin{equation}
\label{connectionmaps}
\map_{a_1}(z^{(1)})=\bar\xi_{a_1}\map(z) \quad \mbox{and} \quad \map_{a_1}(z^{(2)})=\xi_{a_1}\map^{-1}(z)
\end{equation}
for $z\in D$.

\subsection{Abel's Theorem and Jacobi Inversion Problem}
\label{subs:jip}

Given any arrangement of distinct points $\z_l,\w_j\in\RS$ and integers $m_l,k_j\in\N$, a \emph{divisor} is a formal symbol
\begin{equation}
\label{eq:divisor}
d := \sum_l m_l\z_l - \sum_j k_j\w_j.
\end{equation}
We define the degree of the divisor as $|d|:=\sum_lm_l-\sum_j k_j$. By Abel's theorem, $d$ is the principal divisor of a rational function on $\RS$ if,
and only if $|d|=0$ and
\begin{equation}
\label{eq:abel}
\sum_l m_l\Omega_1(\z_l) - \sum_j k_j\Omega_1(\w_j) \equiv 0 \quad \mbox{(mod periods)}.
\end{equation}
When $\z_l$  (resp. $\w_j$) belongs to $L_2\cup L_3$, we understand under $\Omega_1(\z_l)$ (resp. $\Omega_1(\w_j)$) its boundary values on either side of $L_2\cup L_3$ as they are congruent to each other. It is also known that the range of $\Omega_1$, as a multi-valued function on $\RS$, is the entire complex plane $\C$. Moreover, for any $a\in\C$ there uniquely exists $\z_a\in\RS$ such that
\begin{equation}
\label{eq:auxjip}
\Omega_1(\z_a) \equiv a \quad \mbox{(mod periods)}.
\end{equation}
The problem of finding $\z_a$ from $a$ is called the \emph{Jacobi inversion problem}. In particular, the unique solvability of this problem implies that there are no rational functions with a single pole on $\RS$. 

Using \eqref{eq:auxjip}, we see that for each $n\in\N\setminus\{1\}$ and $\gamma\in\C$, there uniquely exists $\z_n=\z_n(\gamma)\in\RS$ such that
\begin{equation}
\label{eq:jip}
\Omega_1(\z_n) + (n-1)\Omega_1(\infty^{(2)}) - n\Omega_1(\infty^{(1)}) + \gamma\frac{\beta_3}{\beta_2} =: l_n+j_n\frac{\beta_3}{\beta_2} \equiv 0 \quad \mbox{(mod periods)},
\end{equation}
$l_n,j_n\in\Z$. Observe that when $\gamma$ is an integer or an integer multiple of $\beta_2/\beta_3$, the constant $\gamma\beta_3/\beta_2$ is congruent to 0 modulo periods and therefore $\z_n+(n-1)\infty^{(2)}-n\infty^{(1)}$ is 
the principal divisor of a rational function on $\RS$ by \eqref{eq:abel}. 
In this case, notice also that if $\z_n=\infty^{(2)}$ then necessarily 
$\z_{n+1}=\infty^{(1)}$ and $l_{n+1}=l_n$, $j_{n+1}=j_n$.

Due to the integral expressions for $\Omega_1(\z_n)$ and $\Omega_1(\infty^{(2)})$, \eqref{eq:jip} can be easily rewritten as
\begin{equation}
\label{eq:jip-r1}
n\left(\Omega_1(\infty^{(2)})-\Omega_1(\infty^{(1)})\right)+(\gamma-j_n)\frac{\beta_3}{\beta_2} = l_n-\int_{\infty^{(2)}}^{\z_n}d\Omega_1.
\end{equation}
Again, by the very definition of $\Omega_1$, we have that
\[
\Omega_1(\infty^{(2)})-\Omega_1(\infty^{(1)}) = \int_{\infty^{(1)}}^{\infty^{(2)}}d\Omega_1 = \frac{1}{2\pi i} \oint_{L_3}d\Omega_0(\infty^{(2)},\infty^{(1)};\tr),
\]
where the second equality follows from the Riemann relation \eqref{eq:riemannrelation}. Now, using \eqref{eq:relation}, \eqref{eq:ms-green}, and \eqref{eq:bw}, we get that
\begin{equation}
\label{eq:Omega1infty12}
\Omega_1(\infty^{(2)})-\Omega_1(\infty^{(1)}) = \frac{1}{2\pi i}\oint_{L_3}\left(dG +\omega_2d\Omega_1\right) = \eqm(\Delta_2)\frac{\beta_3}{\beta_2}-\eqm(\Delta_3).
\end{equation}
Hence, by plugging \eqref{eq:Omega1infty12} into \eqref{eq:jip-r1} and rearranging the summands, we arrive at the equality
\begin{equation}
\label{eq:jip-r3}
\left(n\eqm(\Delta_2)-j_n+\gamma\right)\frac{\beta_3}{\beta_2} = n\eqm(\Delta_3)+l_n-\int_{\infty^{(2)}}^{\z_n}d\Omega_1.
\end{equation}
In particular, comparing the imaginary parts on both sides of \eqref{eq:jip-r3}, we get that
\begin{equation}
\label{eq:jip-r2}
\left(n\eqm(\Delta_2)-j_n+\re(\gamma)\right)\im\left(\frac{\beta_3}{\beta_2}\right) = -\im(\gamma)\re\left(\frac{\beta_3}{\beta_2}\right)-\im\left(\int_{\infty^{(2)}}^{\z_n}d\Omega_1\right).
\end{equation}
Thus, we obtain from \eqref{eq:jip-r2} that 
\begin{equation}
\label{eq:lambdan}
\lambda_n := 2\pi i\left(n\eqm(\Delta_2)-j_n+\gamma\right) = \lambda(\z_n)
\end{equation}
where
\begin{equation}
\label{lambda}
\lambda(\z) := -2\pi i\left(\im(\gamma)\overline{\left(\frac{\beta_3}{\beta_2}\right)}+\im\left(\int_{\infty^{(2)}}^\z d\Omega_1\right)\right)/\im\left(\frac{\beta_3}{\beta_2}\right).
\end{equation}
It follows from the definition of $\Omega_1$ and \eqref{eq:bvOmega1} that $\lambda$ is a continuous function in $\widehat\RS_2$ with continuous traces on both sides of  $L_2$ that satisfy
\begin{equation}
\label{eq:jumplambda}
\lambda^+-\lambda^- = 2\pi i.
\end{equation}
Moreover, it holds that
\begin{equation}
\label{eq:lambdanbounded}
|\lambda_n|\leq\const
\end{equation}
independently of $n$ since $\im(\beta_3/\beta_2)>0$ by \eqref{eq:bw} and $|\Omega_1|$ is uniformly bounded above in $\widetilde\RS_2$.
 
\subsection{Linear Functions} 

Here, we obtain several auxiliary representations for the linear function $z-a_1$ and its multiples. It holds that
\begin{eqnarray}
\label{eq:linearfactor1}
(z-a_1) &=& C_*\exp\left\{2\int_{\infty^{(2)}}^\z d\Omega_0(a_1,\infty^{(1)};\tr)-\omega_2\Omega_1(\z)\right\}\map_{a_1}^{-1}(\z) \\
\label{eq:linearfactor2}
{} &=& C_*\exp\left\{2\int_{\infty^{(1)}}^\z d\Omega_0(a_1,\infty^{(2)};\tr)+\omega_2\Omega_1(\z)\right\}\map_{a_1}(\z)
\end{eqnarray}
where
\begin{equation}
\label{eq:C1}
C_* := \xi_{a_1}\cp(\Delta)\exp\left\{-\omega_2\Omega_1(\infty^{(1)})\right\}.
\end{equation}
To verify \eqref{eq:linearfactor1}, denote the right-hand side of this expression by $E$. Then $E$ is a meromorphic function in $\widetilde\RS_2$ whose primary divisor is equal to $2a_1-\infty^{(1)}-\infty^{(2)}$. Moreover, $E$ has continuous traces on both sides of $L_2$ and $L_3$. 
\begin{figure}[!ht]
\centering
\includegraphics[scale=.4]{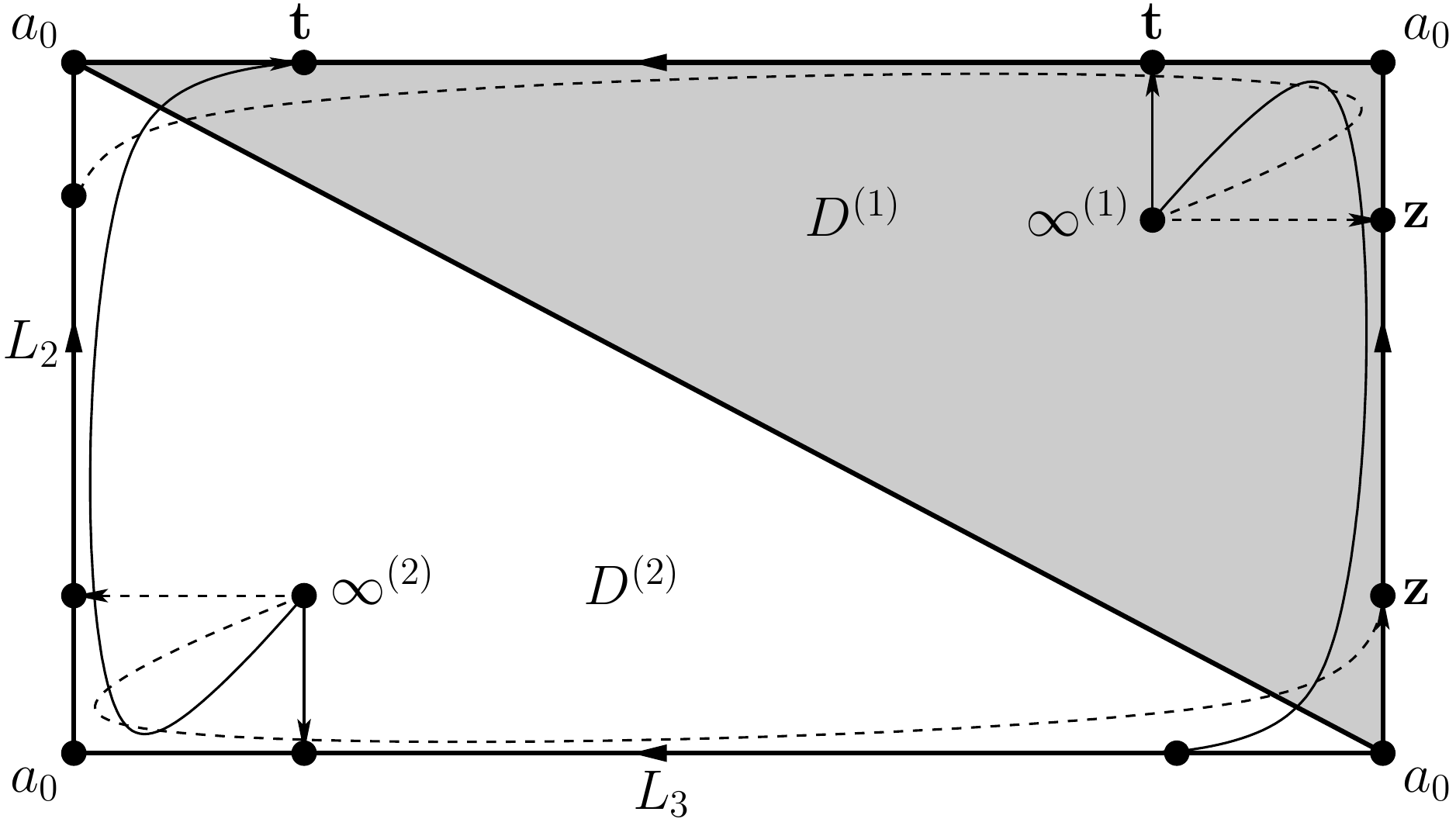}
\caption{\small Paths of integration for $d\Omega_0(a_1,\infty^{(1)};\cdot)$ ($d\Omega_0(a_1,\infty^{(2)};\cdot)$) that start at $\infty^{(2)}$ ($\infty^{(1)}$) and end at $\tr\in L_3$ (solid lines) and $\z\in L_2$ (dashed line).}
\label{fig:jumpphin}
\end{figure}
Examining these traces as in \eqref{eq:bvOmega1}, this time with the help of Figure~\ref{fig:jumpphin}, we get on $L_2$ that
\begin{eqnarray}
\frac{E^+}{E^-} &=& \exp\left\{-2\oint_{L_3}d\Omega_0(a_1,\infty^{(1)};\tr)+\omega_2\frac{\beta_3}{\beta_2}\right\}\frac{\map_{a_1}^-}{\map_{a_1}^+} \nonumber \\
\label{eq:nojump1}
{} &=& \exp\left\{2\pi i\left(2\Omega_1(\infty^{(1)})+\eqm(\Delta_2)\frac{\beta_3}{\beta_2}-\eqm(\Delta_3)\right)\right\}\equiv1
\end{eqnarray}
where the first equality follows from \eqref{eq:bvOmega1}, 
the second from \eqref{eq:riemannrelation} and \eqref{eq:greenboundary}, 
while the last is a consequence of \eqref{eq:Omega1infty12} and the 
fact that $\Omega_1(\infty^{(2)})=-\Omega_1(\infty^{(1)})$. Moreover on $L_3$
\begin{equation}
\label{eq:nojump2}
\frac{E^+}{E^-} = \exp\left\{2\oint_{L_2}d\Omega_0(a_1,\infty^{(1)};\tr)-\omega_2\right\}\frac{\map_{a_1}^-}{\map_{a_1}^+}  \equiv1 
\end{equation}
by \eqref{eq:bvOmega1}, \eqref{eq:greenboundary} and since $d\Omega_0(a_1,\infty^{(1)};\cdot)$ has zero period on $L_2$. Hence, $E$ is a rational function over $\RS$ such that $(E)=2a_1-\infty^{(1)}-\infty^{(2)}$. That is, $E(\z)=\const(z-a_1)$. Now, it is easy to verify by considering the behavior of $E$ near $\infty^{(2)}$ and using \eqref{eq:saympmap} that $C_*$ is chosen exactly so \eqref{eq:linearfactor1} holds.

The validity \eqref{eq:linearfactor2} can be shown following exactly the same steps.

In another connection, using properties of $\Omega_0(\ar;\cdot)$ it is easy to show by analyzing the boundary behavior on $L_2$ that
\begin{equation}
\label{eq:reprP}
\frac{z-a}{a_1-a} = \exp\left\{\Omega_0(\ar;\z)+\Omega_0(\ar;\z^*)\right\},
\end{equation}
where, as usual, $a=\pi(\ar)$. Let now $b$ be a point in a punctured neighborhood of $a_1$ with respect to which we defined $\Omega_0(a_1;\cdot)$. Then
\begin{equation}
\label{eq:reprPa1}
C(b)\frac{z-a_1}{b - a_1}  = \exp\left\{\Omega_0(a_1;\z)+\Omega_0(a_1;\z^*)\right\},
\end{equation}
where $C(b)$ is the normalizing constant. Clearly,
\begin{eqnarray}
C(b) &:=& \exp\left\{\Omega_0(a_1;b)+\Omega_0(a_1;b^*)\right\} = \exp\left\{\Omega_0(a_1;b^*)\right\} \nonumber \\
{} &=& \exp\left\{\int_b^{b^*}d\Omega_0(a_1,\infty^{(1)};\tr)\right\} = \exp\left\{\left(\int_{\infty^{(2)}}^{b^*}-\int_{\infty^{(2)}}^b\right)d\Omega_0(a_1,\infty^{(1)};\tr)\right\} \nonumber \\
\label{Ca1tilde}
{} &=& -\exp\left\{-\omega_2\Omega_1(b)\right\}/\map_{a_1}(b),
\end{eqnarray}
where we used \eqref{eq:linearfactor1} as well as the continuity of $C(b)$ as a function of $b$ in a neighborhood of $a_1$ and the fact that $C(a_1)=-1$ to derive \eqref{Ca1tilde}. To see that $C(a_1)=-1$, pick $b$ on $L_1$. The integration path $\Gamma(b)$ from $b^*$ till $b$ can be chosen so that
$\pi(\Gamma(b))$ is a Jordan curve through $\pi(b)$ since $b$ and $b^*$ have the same canonical projection. Using \eqref{eq:dOmega0inf1} and \eqref{eq:raz} we can write mod $2\pi i$ that
\[
\int_b^{b^*}d\Omega_0(a_1,\infty^{(1)};\tr) = \int_{\pi(\Gamma(b))}\left(\frac{w(t)}{2(t-a_1)}+t+c\right)\frac{dt}{w(t)} = \pm\pi i+\int_{\pi(\Gamma(b))}\frac{t+c}{w(t)}dt
\]
where $c$ is some constant and the choice of $+$ or $-$ in front of $\pi i$ depends of the orientation of $\pi(\Gamma(b))$. Finally, it is easy to see that the last integral approaches $0$ as $b$ tends to $a_1$.

\subsection{Cauchy-type Integrals on $L$}
\label{ssec:cti}

Let $\chi$ be a continuous function on $L$.  Define
\begin{equation}
\label{eq:Fchi}
F_\chi(\z) := \frac{1}{2\pi i}\oint_L\chi(\tr)\frac{w(\tr)+w(\z)}{t-z}\frac{dt}{2w(\tr)},
\end{equation}
$\z\in\RS\setminus (L\cup\{\infty^{(1)},\infty^{(2)}\})$. It is known \cite{Zver71}, and can  be verified easily by projecting onto the complex plane 
using \eqref{eq:BK6}--\eqref{eq:BK7} below, together with the classical Sokhotski-Plemelj formulae, \cite[Sec. I.4.2]{Gakhov}, see also \eqref{eq:SP}, that $F_\chi$ is a sectionally holomorphic function in $\RS\setminus (L\cup\{\infty^{(1)},\infty^{(2)}\})$ with simple poles at $\infty^{(1)}$ and $\infty^{(2)}$ that satisfies
\begin{equation}
\label{SPC}
F_\chi^+-F_\chi^- = \chi \quad \mbox{a.e. on} \quad L.
\end{equation}
By developing $1/(t-z)$ and $w(z)$ in  powers of $z$ at infinity as was done after \eqref{eq:raz}, we get that $F_\chi(z^{(1)})-\ell_\chi(z)\to0$  and $F_\chi(z^{(2)})+\ell_\chi(z)\to0$  as $z\to\infty$, where
\begin{equation}
\label{eq:ellchi}
\ell_\chi(z) = u_\chi z+v_\chi:=-z\oint_L\frac{\chi}{2w}\frac{dt}{2\pi i} + \oint_L(A-t)\frac{\chi}{2w}\frac{dt}{2\pi i}
\end{equation}
and $A$ was defined in \eqref{eq:raz}. As for the traces of $F_\chi$ on $L$, it holds that
\begin{equation}
\label{eq:BK5}
\|F_\chi^\pm\|_{2,L} \leq \const\|\chi\|_{2,L}, \quad \|\cdot\|_{2,L} := \left(\oint_L|\cdot|^2|d\Omega|\right)^{1/2},
\end{equation}
where $\const$ is independent of $\chi$. Indeed, we have that
\begin{eqnarray}
F_\chi(z^{(1)}) &=& \frac{w(\z)}{4\pi i}\oint_L\frac{\chi(\tr)}{t-z}\frac{dt}{w(\tr)} + \frac{1}{4\pi i}\oint_L\frac{\chi(\tr)}{t-z}dt \nonumber \\
{} &=& \frac{w(z)}{4\pi i}\int_\Delta\frac{\chi^+(t)+\chi^-(t)}{t-z}\frac{dt}{w^+(t)} + \frac{1}{4\pi i}\int_\Delta\frac{\chi^+(t)-\chi^-(t)}{t-z} dt \nonumber \\
\label{eq:BK6}
{} &=& \frac12\left(R_\Delta(\chi^+;z)+R_\Delta(\chi^-;z) + C_\Delta(\chi^+;z) - C_\Delta(\chi^-;z)\right), \\
\label{eq:BK7}
F_\chi(z^{(2)}) &=& \frac12\left(-R_\Delta(\chi^+;z)-R_\Delta(\chi^-;z) + C_\Delta(\chi^+;z) - C_\Delta(\chi^-;z)\right), 
\end{eqnarray}
where $\chi^\pm(t):=\chi(\tr)$, $t\in\Delta^\pm$, and the functions $R_\Delta$ and $C_\Delta$ are defined in \eqref{eq:RDelta}. Hence,
\begin{eqnarray}
\|F^+_\chi\|^2_{2,L} &\leq& \int_\Delta\left(|R_\Delta^+(\chi^+;t)|^2+|R_\Delta^-(\chi^+;t)|^2+|R_\Delta^+(\chi^-;t)|^2+|R_\Delta^-(\chi^-;t)|^2\right)\frac{|dt|}{|w^+(t)|} \nonumber \\
{} && + \int_\Delta\left(|C_\Delta^+(\chi^+;t)|^2+|C_\Delta^-(\chi^+;t)|^2+|C_\Delta^+(\chi^-;t)|^2+|C_\Delta^-(\chi^-;t)|^2\right)\frac{|dt|}{|w^+(t)|} \nonumber \\
{} &\leq& \const\left(\int_\Delta\frac{|\chi^+(t)|^2+|\chi^-(t)|^2}{|w^+(t)|}|dt|\right) = \const\|\chi\|_{2,L}^2 \nonumber
\end{eqnarray}
by Lemma~\ref{lem:A2weight}.

\section{Boundary Value Problems on $\RS$}
\label{sec:rhp}

On an elliptic Riemann surface it is possible to prescribe all but one elements of the zero/pole set of a sectionally  meromorphic function with given jump. The following proposition deals with the case when we prescribe $n$ poles at $\infty^{(1)}$ and $n-1$ zeros at $\infty^{(2)}$.

In what follows, we construct a function $\map_n$ which should rather
be denoted by $\map_{n,\kappa}$. However, we alleviate the notation and drop 
the subscript $\kappa$.

\begin{proposition}
\label{prop:map}
Let  $\kappa\in\{1,2,3\}$ be fixed.\\
\textnormal{\bf (i)}~For each $n\in\N\setminus\{1\}$ and $\gamma\in\C$ there exists $\z_n\in\RS$ such that $\z_n+(n-1)\infty^{(2)}-n\infty^{(1)}$ is the principal divisor of a function $\map_n$ which is meromorphic in $\widehat\RS_\kappa$ and has continuous traces on $L_\kappa$ that satisfy
\begin{equation}
\label{eq:littlejump}
\map_n^+=\map_n^-e^{2\pi i\gamma}.
\end{equation}
Under the normalization $\map_n(z)z^{-k_n}\to1$ as $z\to\infty$, $\map_n$ is the unique function meromorphic in $\widehat\RS_\kappa$ with 
principal divisor of the form $\w+(n-1)\infty^{(2)} - n\infty^{(1)}$, $\w\in\RS$, and continuous traces on $L_\kappa$ that satisfy \eqref{eq:littlejump}. Moreover, if $\z_n=\infty^{(1)}$ then $\z_{n-1}=\infty^{(2)}$ and $\map_n=\map_{n-1}$.\\
\textnormal{\bf (ii)}~It holds that
\begin{equation}
\label{eq:littleconjfun1}
\frac{(\map_n\map^*_n)(z)}{G_\kappa(\cp(\Delta))^{2n-1}} = \xi_{n,\kappa}
\left\{
\begin{array}{ll}
(z-z_n)/|\map(z_n)|, & \z_n\in D^{(2)}\setminus\{\infty^{(2)}\}, \bigskip \\
\cp(\Delta), & \z_n=\infty^{(2)}, \bigskip\\
(z-z_n)|\map(z_n)|, & \z_n\in L\cup D^{(1)}\setminus\{\infty^{(1)}\},
\end{array}
\right.
\end{equation}
where $|\xi_{n,\kappa}|=1$ and $G_\kappa:=\exp\{2\pi \eqm(\Delta_\kappa)\im(\gamma)\}$.\\
\textnormal{\bf (iii)}~It holds that
\begin{equation}
\label{eq:littleconjfun2}
\frac{\map_n^*(z)}{\map_n(z)} = \frac{\xi_{n,\kappa}G_\kappa}{\map^{2n-1}(z)}\Upsilon_\kappa(\z_n;z)
\left\{
\begin{array}{ll}
\displaystyle \frac{z-z_n}{\map(z)|\map(z_n)|}, & \z_n\in D^{(2)}\setminus\{\infty^{(2)}\}, \bigskip \\
\cp(\Delta)/\map(z), & \z_n=\infty^{(2)}, \bigskip\\
\displaystyle \frac{\map(z)|\map(z_n)|}{z-z_n}, & \z_n\in L\cup D^{(1)}\setminus\{\infty^{(1)}\},
\end{array}
\right.
\end{equation}
where $\{\Upsilon_\kappa(\ar;\cdot)\}_{\ar\in\RS}$ is a normal family of non-vanishing functions in $D$ that are uniformly bounded in $D$ for $\ar$ outside of any fixed neighborhood of $L$.
 \end{proposition}
\begin{proof}
For definiteness, we put $\kappa=3$. Cases where $\kappa=1,2$ are handled
similarly upon 
choosing $L_{\kappa-1}$ to be the {\bf a}-cycle, $L_\kappa$ to be the {\bf b}-cycle, and $a_{\kappa+1}$ to be the initial bound for integration. \\
{\bf (i)} Let $\z_n$ be the unique point satisfying \eqref{eq:jip}. Set
\begin{equation}
\label{eq:mapn1}
\map_n(\z) :=\gamma_n\exp\left\{\Omega_0(\z_n;\z)+(n-1)\Omega_0(\infty^{(2)};\z)+2\pi i(\gamma-j_n)\Omega_1(\z)\right\}, \quad \z\in\widetilde\RS_2,
\end{equation}
where $\gamma_n$ is some constant to be chosen later. Then $\map_n$ is a holomorphic and non-vanishing function in $\widetilde\RS_2$ except for a pole of order $n$ at $\infty^{(1)}$,  a zero of order $n-1$ at $\infty^{(2)}$, and a simple zero at $\z_n$. Denote by $\chi$ the multiplicative jump of $\map_n$ on $L_2\cup L_3$. That is, $\map_n^+=\map_n^-\chi$. Then it follows from \eqref{eq:bvOmega1} that
\begin{equation}
\label{eq:chi1}
\chi = \exp\left\{2\pi i(\gamma-j_n)\right\} = \exp\left\{2\pi i\gamma\right\} \quad \mbox{on} \quad L_3.
\end{equation}
Moreover, we deduce from \eqref{eq:bvOmegaa}, \eqref{eq:bvOmega1}, and \eqref{eq:jip} that
\begin{eqnarray}
\chi &=& \exp\left\{-2\pi i\left(\Omega_1(\z_n)+(n-1)\Omega_1(\infty^{(2)})-n\Omega_1(\infty^{(1)})+(\gamma-j_n)\frac{\beta_3}{\beta_2}\right)\right\} \nonumber \\
\label{eq:chi2}
{} &=& \exp\{-2\pi il_n\} = 1 \quad \mbox{on} \quad L_2\setminus\{\z_n\}.
\end{eqnarray}
Clearly, \eqref{eq:chi2} extends to $\z_n$ as well by continuity when the latter belongs to $L_2$. Thus, $\map_n$ is, in fact, meromorphic in $\widehat\RS_3$ and satisfies \eqref{eq:littlejump}.

Choose $\gamma_n$ so that $\map_n(z)z^{-n}\to1$ as $z\to\infty$. Let $\widetilde\map_n$ be a function meromorphic in $\widehat\RS_3$ with continuous traces on $L_3$ satisfying \eqref{eq:littlejump}, such that $(\widetilde\map_n)=\w+(n-1)\infty^{(1)}-n\infty^{(2)}$ for some $\w\in\RS$, and normalized so that $\widetilde\map_n(\z)z^{-k_n}\to1$ as $z\to\infty$, where $k_n=n$ if $\w\neq\infty^{(1)}$ and $k_n=n-1$ otherwise. Then the ratio $\map_n/\widetilde\map_n$ is continuous across $L_3$ and therefore  is a rational function over $\RS$. Since $(\map_n/\widetilde\map_n)=\z_n-\w$ and there are no rational functions over $\RS$ with only one pole, $\w=\z_n$ and $\widetilde\map_n$ is a constant multiple of $\map_n$. Due to the normalization at $\infty^{(1)}$, it holds that $\widetilde\map_n=\map_n$. That is, $\map_n$ is the unique function with the prescribed properties. The claim for $\z_n=\infty^{(1)}$ follows from \eqref{eq:mapn1} and the remark made after \eqref{eq:jip}.\\
{\bf (ii)} Suppose that $\z_n\in\RS\setminus\{a_1,\infty^{(1)},\infty^{(2)}\}$. Then \eqref{eq:mapn1} combined with \eqref{eq:relation} and \eqref{eq:lambdan} yields that
\begin{equation}
\label{repr1}
\map_n(\z) = \gamma_n\map_{a_1}^{n-1}(\z) \exp\left\{\Omega_0(\z_n;\z)+(\lambda_n-\omega_2)\Omega_1(\z)\right\}.
\end{equation}
For brevity, let us put $\gamma_\Delta:=\cp(\Delta)$. Using \eqref{eq:saympmap} and \eqref{eq:reprP} we can equivalently write
\begin{equation}
\label{repr2}
\map_n(\z) = \frac{(z-z_n)\map_{a_1}^{n-1}(\z)\exp\left\{(\lambda_n-\omega_2)\Omega_1(\z)-\Omega_0(\z_n;\z^*)\right\}}{(\xi_{a_1}\gamma_\Delta)^{1-n}\exp\left\{(\lambda_n-\omega_2)\Omega_1(\infty^{(1)})-\Omega_0(\z_n;\infty^{(2)})\right\}}.
\end{equation}
Then it follows from the symmetries $\map_{a_1}(\z)\map_{a_1}(\z^*)\equiv1$ and $\Omega_1(\z)+\Omega_1(\z^*)\equiv0$ that
\begin{equation}
\label{limitmiddle}
\map_n(\z)\map_n(\z^*) = \frac{(z-z_n)(a_1-z_n)(\xi_{a_1}\gamma_\Delta)^{2n-1}}{\xi_{a_1}\gamma_\Delta \exp\left\{2(\lambda_n-\omega_2)\Omega_1(\infty^{(1)})-2\Omega_0(\z_n;\infty^{(2)})\right\}}
\end{equation}
where we used \eqref{eq:reprP} once more. Set
\[
\alpha(\z) := C_*^{-1}(a_1-z)\exp\left\{(\omega_2-2\lambda(\z))\Omega_1(\infty^{(1)})+2\Omega_0(\z;\infty^{(2)})\right\},
\]
where $C_*$ was defined in \eqref{eq:C1}. Clearly, we can rewrite \eqref{limitmiddle} as
\begin{equation}
\label{limitmiddle1}
\map_n(\z)\map_n(\z^*) = (\xi_{a_1}\gamma_\Delta )^{2n-1}(z-z_n)\alpha(\z_n).
\end{equation}
Now, we deduce from \eqref{eq:rr} and \eqref{eq:linearfactor2} that
\begin{eqnarray}
\alpha(\z) &=&C_*^{-1}(a_1-z)\exp\left\{(\omega_2-2\lambda(\z))\Omega_1(\infty^{(1)})-2\int_{\infty^{(1)}}^\z d\Omega_0(a_1,\infty^{(2)};\tr)\right\} \nonumber \\
\label{eq:tildealpha}
{} &=& -\exp\left\{-2\lambda(\z)\Omega_1(\infty^{(1)})+\omega_2\left(\Omega_1(\z)-\Omega_1(\infty^{(2)})\right)\right\}\map_{a_1}(\z).
\end{eqnarray}
Representation \eqref{eq:tildealpha} combined with \eqref{eq:jumplambda}, \eqref{eq:bvOmega1}, and \eqref{eq:greenboundary} yields that $\alpha$ is a continuous function in $\RS\setminus\{\infty^{(1)}\}$ that vanishes at $\infty^{(2)}$, blows up at $\infty^{(1)}$, and is otherwise non-vanishing and finite. It further follows from \eqref{eq:tildealpha} and \eqref{eq:Omega1infty12} that
\[
\left|\frac{\alpha(\z)}{\map_{a_1}(\z)}\right| = \exp\left\{-2\pi\eqm(\Delta_2)\im\left(\int_{\infty^{(2)}}^\z d\Omega_1\right)-\re\left(\lambda(\z)\left[\eqm(\Delta_3)-\eqm(\Delta_2)\frac{\beta_3}{\beta_2}\right]\right)\right\}.
\]
The latter expression can be simplified using \eqref{lambda}, \eqref{eq:Omega1infty12}, and elementary algebra to
\begin{equation}
\label{modulus}
\left|\alpha(\z)/\map_{a_1}(\z)\right| = \exp\left\{2\pi\eqm(\Delta_3)\im(\gamma)\right\}=G_3.
\end{equation}
Hence, \eqref{eq:littleconjfun1} holds by \eqref{limitmiddle1} and \eqref{modulus} with
\begin{equation}
\label{xi1}
\xi_{n,3} := \xi_{a_1}^{2n-1}\alpha(\z_n)/|\alpha(\z_n)|.
\end{equation}

When $\z_n=\infty^{(2)}$, we get as in \eqref{repr1} and \eqref{repr2} that
\begin{equation}
\label{repr3}
\map_n(\z) = (\xi_{a_1}\gamma_\Delta )^n\map_{a_1}^n(\z) \exp\left\{\lambda_n\Omega_1(\z)-\lambda_n\Omega_1(\infty^{(1)})\right\}
\end{equation}
and therefore
\begin{equation}
\label{xi2}
\map_n(\z)\map_n(\z^*) = (\xi_{a_1}\gamma_\Delta )^{2n}\exp\left\{-2\lambda(\infty^{(2)})\Omega_1(\infty^{(1)})\right\} =: \gamma_\Delta^{2n}\xi_{n,3}G_3,
\end{equation}
where it can be shown as in \eqref{modulus} that $|\exp\left\{-2\lambda(\infty^{(2)})\Omega_1(\infty^{(1)})\right\}|=G_3$.

Finally, suppose that $\z_n=a_1$. Then we deduce as in \eqref{repr2} only using \eqref{eq:reprPa1} instead of \eqref{eq:reprP} that
\begin{equation}
\label{repr4}
\map_n(\z) = \frac{(z-a_1)\map^{n-1}(\z)\exp\left\{(\lambda_n-\omega_2)\Omega_1(\z)-\Omega_0(a_1;\z^*)\right\}}{(\xi_{a_1}\gamma_\Delta )^{1-n}\exp\left\{(\lambda_n-\omega_2)\Omega_1(\infty^{(1)})-\Omega_0(a_1;\infty^{(2)})\right\}}.
\end{equation}
Further, we get as in \eqref{limitmiddle} only by using \eqref{eq:reprPa1} again, that
\[
\map_n(\z)\map_n(\z^*) = \frac{(z-a_1)(b-a_1)(\xi_{a_1}\gamma_\Delta )^{2(n-1)}}{C(b)\exp\left\{2(\lambda_n-\omega_2)\Omega_1(\infty^{(1)})-2\Omega_0(a_1;\infty^{(2)})\right\}}.
\]
Since $\Omega_0(a_1;\infty^{(2)})=-\int_{\infty^{(2)}}^bd\Omega_0(a_1,\infty^{(1)};\tr)$ by definition, we get from \eqref{eq:linearfactor1}, \eqref{Ca1tilde}, and \eqref{eq:tildealpha} that
\begin{equation}
\label{eq:gamman24}
\map_n(\z)\map_n(\z^*) = \frac{-(z-a_1)(\xi_{a_1}\gamma_\Delta )^{2(n-1)}C_*}{\exp\left\{2(\lambda_n-\omega_2)\Omega_1(\infty^{(1)})\right\}} = (\xi_{a_1}\gamma_\Delta )^{2n-1}(z-a_1)\alpha(a_1),
\end{equation}
which finishes the proof of \eqref{eq:littleconjfun1} upon setting
\begin{equation}
\label{xi3}
\xi_{n,3}:=\xi_{a_1}^{2n-1}\alpha(a_1)/|\alpha(a_1)|.
\end{equation}
{\bf (iii)} For $\z_n$ with finite canonical projection, we deduce from \eqref{repr2} and \eqref{repr4} that
\begin{equation}
\label{ratio}
\frac{\map_n(\z^*)}{\map_n(\z)} = \frac{\exp\{2(\omega_2-\lambda_n)\Omega_1(\z)-\Omega_0(\z_n;\z)+\Omega_0(\z_n;\z^*)\}}{\map_{a_1}^{2(n-1)}(\z)}.
\end{equation}

Suppose first that $\z_n=a_1$. Then we get from \eqref{ratio} by using \eqref{connectionmaps} and \eqref{eq:reprPa1} that
\[
\frac{\map_n^*(z)}{\map_n(z)} = \left(\frac{\xi_{a_1}}{\map(z)}\right)^{2n-2}\frac{1}{C(b)}\frac{b-a_1}{z-a_1}\exp\left\{2(\omega_2-\lambda_n)\Omega_1(z^{(1)})+2\Omega_0(a_1;z^{(2)})\right\}.
\]
Hence, \eqref{eq:littleconjfun2} takes place with
\[
\Upsilon_3(a_1;z) := \frac{1}{C(b)}\frac{b-a_1}{\xi_{a_1}\alpha(a_1)}\exp\left\{2(\omega_2-\lambda(a_1))\Omega_1(z^{(1)})+2\Omega_0(a_1;z^{(2)})\right\}
\]
by \eqref{xi3}. Since
\[
\frac{1}{C(b)}\frac{b-a_1}{\xi_{a_1}\alpha(a_1)} = \gamma_\Delta\exp\left\{-2(\omega_2-\lambda(a_1))\Omega_1(\infty^{(1)})+2\int_{\infty^{(2)}}^bd\Omega_0(a_1,\infty^{(1)};\tr)\right\}
\]
by \eqref{Ca1tilde}, \eqref{eq:linearfactor1}, \eqref{eq:C1}, and \eqref{eq:tildealpha} and $\Omega_0(a_1;z^{(2)}) = \int_b^{z^{(2)}}d\Omega_0(a_1,\infty^{(1)};\tr)$ by the very definition, we get that
\[
\Upsilon_3(a_1;z) = \gamma_\Delta\exp\left\{2\left(\omega_2-\lambda(a_1)\right)\left(\Omega_1(z^{(1)})-\Omega_1(\infty^{(1)})\right)+2\int_{\infty^{(2)}}^{z^{(2)}}d\Omega_0(a_1,\infty^{(1)};\tr)\right\}.
\]

Suppose now that $\z_n\in (L\setminus\{a_1\})\cup(D^{(1)}\setminus\{\infty^{(1)}\})$. Then we obtain from \eqref{ratio} that
\[
\frac{\map_n^*(z)}{\map_n(z)} = \left(\frac{\xi_{a_1}}{\map(z)}\right)^{2n-2}\frac{a_1-z_n}{z-z_n}\exp\left\{2(\omega_2-\lambda_n)\Omega_1(z^{(1)})+2\Omega_0(\z_n;z^{(2)})\right\},
\]
where we used \eqref{connectionmaps} and \eqref{eq:reprP}. Therefore, \eqref{eq:littleconjfun2} holds with
\begin{eqnarray}
\Upsilon_3(\ar;z) &:=& \frac{a_1-a}{\xi_{a_1}\alpha(\ar)}\exp\left\{2(\omega_2-\lambda(\ar))\Omega_1(z^{(1)})+2\Omega_0(\ar;z^{(2)})\right\} \nonumber \\
{} &=& \gamma_\Delta\exp\left\{2\left(\omega_2-\lambda(\ar)\right)\left(\Omega_1(z^{(1)})-\Omega_1(\infty^{(1)})\right)+2\int_{\infty^{(2)}}^{z^{(2)}}d\Omega_0(\ar,\infty^{(1)};\tr)\right\} \nonumber
\end{eqnarray}
by \eqref{xi1} and where we used the definition of $\alpha$ (see the line above \eqref{eq:tildealpha}), \eqref{eq:C1}, and \eqref{eq:generalthirdkind} to derive the second equality. Treating $d\Omega_0(\ar;\infty^{(1)};\tr)$ as being identically zero when $\ar=\infty^{(1)}$, we can define 
\[
\Upsilon_3(\infty^{(1)};z) := \gamma_\Delta\exp\left\{2\left(\omega_2-\lambda(\infty^{(1)})\right)\left(\Omega_1(z^{(1)})-\Omega_1(\infty^{(1)}\right)\right\}.
\]
Then for each $\ar\in L\cup D^{(1)}$ the function $\Upsilon_3(\ar;\cdot)$ is holomorphic and non-vanishing in $D$ such that $\Upsilon(\ar;\infty)=\gamma_\Delta$. Moreover, the continuity of $\lambda$ as a function of $\ar$ in $L\cup D^{(1)}$, \eqref{eq:OmegaInfF}, and \eqref{eq:locallyuniform} imply that
\begin{equation}
\label{normalfamily}
\Upsilon_3(\tr;\cdot) \rightrightarrows \Upsilon_3(\ar;\cdot)
\end{equation}
in $D$ as $\tr\to\ar$, $\ar,\tr\in L\cup D^{(1)}$.

Assume next that $\z_n\in D^{(2)}\setminus\{\infty^{(2)}\}$. Then we deduce from \eqref{ratio} that
\[
\frac{\map_n^*(z)}{\map_n(z)} = \left(\frac{\xi_{a_1}}{\map(z)}\right)^{2n}\frac{z-z_n}{a_1-z_n}\map_{a_1}^2(z^{(1)})\exp\left\{2(\omega_2-\lambda_n)\Omega_1(z^{(1)})-2\Omega_0(\z_n;z^{(1)})\right\},
\]
where as before we used \eqref{connectionmaps} and \eqref{eq:reprP}. Thus, \eqref{eq:littleconjfun2} holds with
\[
\Upsilon_3(\ar;z) := \frac{\xi_{a_1}}{\alpha(\ar)(a_1-a)}\exp\left\{-2\lambda(\ar)\Omega_1(z^{(1)})-2\Omega_0(\ar;z^{(1)})+2\Omega_0(\infty^{(2)};z^{(1)})\right\}
\]
due to \eqref{xi1} and \eqref{eq:relation}. Because
\[
\frac{\xi_{a_1}}{\alpha(\ar)(a_1-a)} = \frac{1}{\gamma_\Delta}\exp\left\{2\lambda(\ar)\Omega_1(\infty^{(1)})+2\int_{a_1}^{\infty^{(1)}}d\Omega_0(\ar,\infty^{(2)};\tr)\right\} \nonumber
\]
by \eqref{eq:tildealpha}, \eqref{eq:linearfactor1}, \eqref{eq:C1}, \eqref{eq:rr}, and since $\Omega_0(\ar;z^{(1)})-\Omega_0(\infty^{(2)};z^{(1)}) = \int_{a_1}^{z^{(1)}}d\Omega_0(\ar,\infty^{(2)};\tr)$ by \eqref{eq:rr} again, we get that
\begin{equation}
\label{upsilon3ad2}
\Upsilon_3(\ar;z) =\frac{1}{\gamma_\Delta}\exp\left\{-2\lambda(\ar)\left(\Omega_1(z^{(1)})-\Omega_1(\infty^{(1)})\right)-2\int_{\infty^{(1)}}^{z^{(1)}}d\Omega_0(\ar,\infty^{(2)};\tr)\right\}.
\end{equation}

Finally, assume that $\z_n=\infty^{(2)}$. Then we get from \eqref{repr3} that
\[
\frac{\map_n(\z^*)}{\map_n(\z)} = \frac{1}{\map_{a_1}^{2n}(\z)}\exp\left\{-2\lambda(\infty^{(2)})\Omega_1(\z)\right\}
\]
and therefore \eqref{eq:littleconjfun2} holds with
\[
\Upsilon_3(\infty^{(2)};z) =\frac{1}{\gamma_\Delta}\exp\left\{-2\lambda(\infty^{(2)})\left(\Omega_1(z^{(1)})-\Omega_1(\infty^{(1)})\right)\right\}
\]
by \eqref{xi2}. Clearly, for each $\ar\in D^{(2)}$ the function $\Upsilon_3(\cdot;\infty^{(j)})$ is holomorphic and non-vanishing in $D$ such that $\Upsilon_3(\ar;\infty)=1/\gamma_\Delta$. Moreover, the continuity of $\lambda$ as a function of $\ar$ in $D^{(2)}$, \eqref{eq:OmegaInfF}, and \eqref{eq:locallyuniform} imply that \eqref{normalfamily} holds for $\ar,\tr\in D^{(2)}$ as well. It only remains to observe that if $\tr\to\ar\in L$, $\tr\in D^{(2)}$, then the limiting function is given by \eqref{upsilon3ad2} used with this given~$\ar$.
\end{proof}

In Section~\ref{subs:jip} we explained that $\z_n+(n-1)\infty^{(2)}-n\infty^{(1)}$ is the principal divisor of a rational function over $\RS$ when $\gamma$ is an integer or an integer multiple of $\beta_\kappa/\beta_{\kappa+1}$. In the former case $\map_n$ is exactly this rational function ($e^{2\pi i\gamma}=1$ and therefore $\map_n$ has no jump across $L_\kappa$), but in the latter case it is not. In fact, $\map_n$ is then the product of the rational function with 
principal divisor $\z_n+(n-1)\infty^{(2)}-n\infty^{(1)}$ by a function holomorphic in $\widehat\RS_\kappa$ and having multiplicative jumpacross $L_k$
as in \eqref{eq:littlejump}.

\section{Szeg\H{o}-type Functions on $\RS$}
\label{sec:sn}

\subsection{Proof of Proposition~\ref{prop:sn}} The ground work for the proof of Proposition~\ref{prop:sn} was done in Sections~\ref{sec:ci} and~\ref{sec:rhp}.  Here, we only need to combine the results of these sections. 

Fix $\kappa\in\{1,2,3\}$ and let $G_{h,\kappa}$, $S_{h,\kappa}$, and $\map_n(=\map_{n,\kappa})$ be as in Propositions~\ref{prop:map} and~\ref{prop:szego}, where $\gamma$ in Proposition~\ref{prop:map} equals to $-m_0/(2\pi i\beta_\kappa)$ with $m_0$ defined in \eqref{eq:gm}. Set
\begin{equation}
\label{definitionSn}
S_n(\z) = \left\{
\begin{array}{ll}
\map_n(\z)/S_{h,\kappa}(z), & \z\in D^{(1)} \smallskip \\
G_{h,\kappa}\map_n(\z)S_{h,\kappa}(z), & \z\in D^{(2)}.
\end{array}
\right.
\end{equation}
Fix $\tr\in L$ and let $D^{(2)}\ni\z\to\tr$ so that $z\to t\in\Delta^\mp$ (recall that $z=\pi(\z)$ and $t=\pi(\tr)$). Then 
\begin{equation}
\label{eq:atot}
S_n(\z) \to S_n^-(\tr) \quad \mbox{and} \quad S_n(\z)=\map(\z)G_{h,\kappa}S_{h,\kappa}(z) \to \map^-(\tr)G_{h,\kappa}S_{h,\kappa}^\mp(t)
\end{equation}
by the very definition of $S_n$. Let now $D^{(1)}\ni\ar\to\tr$. Then $a\to t\in\Delta^\pm$, $S_n(\ar) \to S_n^+(\tr)$, and
\begin{eqnarray}
S_n(\ar) &=& \frac{\map(\ar)}{S_{h,\kappa}(a)} \to \frac{\map_n^+(\tr)}{S_{h,\kappa}^\pm(t)} = \left\{
\begin{array}{ll}
\map_n^-(\tr)G_{h,\kappa}S_h^\mp(t)/h(t), & t\in\Delta\setminus\Delta_\kappa, \bigskip \\
\map_n^-(\tr)\exp\left\{-\frac{m_0}{\beta_\kappa}\right\}\widetilde G_{h,\kappa}S_{h,\kappa}^\mp(t)/h(t), & t\in\Delta_\kappa 
\end{array}
\right\} \nonumber \\
{} &=& \frac{S_n^-(\tr)}{h(t)} \nonumber
\end{eqnarray}
by \eqref{eq:littlejump}, \eqref{eq:szego}, and \eqref{eq:atot}.  Moreover, it follows easily from Propositions~\ref{prop:szego} and~\ref{prop:map}(i) that $S_n$ satisfies all the functional properties required by Proposition~\ref{prop:sn}. Hence, we are left to show uniqueness of $S_n$. Suppose that $\tilde S_n$ 
is another such function with principal divisor of the form
$\w + (n-1)\infty^{(2)} - n\infty^{(1)}$. Then $S_n/\tilde S_n$ is a rational 
function on $\RS$ by the principle of analytic continuation, and
it has at most one pole namely $\w$. Therefore it is a constant 
as there are no rational functions over $\RS$ with one pole. 
The fact that $S_n=S_{n-1}$ and $\z_{n-1}=\infty^{(2)}$ whenever $\z_n=\infty^{(1)}$ follows from the analogous claim in Proposition~\ref{prop:map}(i). 
\qed

\subsection{Proof of Proposition~\ref{prop:snasymp}}
By the very definition of $S_n$, we have that
\[
S_nS_n^* = \map_nS_{h,\kappa}^{-1}G_{h,\kappa}\map_n^*S_{h,\kappa} = G_{h,\kappa}\map_n\map_n^*.
\]
Observe that
\[
G_{h,\kappa} = \exp\left\{-m_1+m_0\frac{\beta_\kappa^1}{\beta_\kappa}\right\} = G_h\exp\left\{-\eqm(\Delta_\kappa)\frac{m_0}{\beta_\kappa}\right\}
\]
by \eqref{eqandbetas} and since
\[
-m_1+m_0a_0 = \frac{1}{\pi i}\int_\Delta(a_0-t)\frac{\log h(t)}{w^+(t)}dt = \int\log hd\eqm
\]
by \eqref{eq:eqmeas}. As we use Proposition~\ref{prop:map} with $\gamma=-m_0/(2\pi i\beta_\kappa)$, it holds that
\[
|G_{h,\kappa}/G_h| = \exp\left\{-\re\left(\eqm(\Delta_\kappa)\frac{m_0}{\beta_\kappa}\right)\right\} = \exp\left\{-2\pi \eqm(\Delta_\kappa)\im(\gamma)\right\} = G_\kappa^{-1}.
\]
Hence, \eqref{eq:conjfunSn1} follows from \eqref{eq:littleconjfun1} with $\xi_n:=\xi_{n,\kappa}G_{h,\kappa}G_\kappa/G_h$. The fact that $\xi_n$ does not depend on $\kappa$ follows from  uniqueness of $S_n$.

By the same token, we get that $S_n^*/S_n=(G_{h,\kappa}S_{h,\kappa}^2)(\map_n^*/\map_n)$. As $S_{h,\kappa}$ is a holomorphic and non-vanishing function in $D$ with continuous and non-vanishing trace on $\partial D$, \eqref{eq:conjfunSn2} follows from \eqref{eq:littleconjfun2} with $\Upsilon(\ar;\cdot):=S_{h,\kappa}^2(\cdot)\Upsilon_\kappa(\ar;\cdot)$. Again, $\Upsilon(\ar;\cdot)$ does not depend on $\kappa$ by uniqueness of $S_n$. \qed

\subsection{An Auxiliary Estimate} 
\label{ss:specialjip}

For the proof of Theorem~\ref{thm:pade} in the case of $\z_n$ approaching $L$, we need an estimate of the ratio $S_n^*/S_n$ on $L^+$ ($L$ approached from $D^{(1)}$), see \eqref{boundtildeXn} below. 

We start by constructing a special rational functions of degree 2 on $\RS$. The unique solvability of \eqref{eq:auxjip} means that $\Omega_1$ and its boundary values from each side on $L_2\cup L_3$ define an isomorphism from $\RS$ onto the quotient surface $\C/(\Z+(\beta_3/\beta_2)\Z)$. In particular,
\[
\delta := \min\left\{\left|\Omega_1(\tr)-\Omega_1(\infty^{(1)})\right|:~\tr\in L^+\right\}>0.
\]
Let $O$ be any neighborhood of $L$ which is disjoint from $O_{\infty^{(1)}}$, the neighborhood of $\infty^{(1)}$ given by
\begin{equation}
\label{defOinfini}
O_{\infty^{(1)}}:=\left\{\z\in\RS:~\left|\Omega_1(\z)-\Omega_1(\infty^{(1)})\right|<\delta/6\right\},
\end{equation}
and such that for any $\z\in O$ there exists $\tr\in\RS\setminus O$ satisfying
\begin{equation}
\label{definitiontn1}
\left|\Omega_1(\z)-\Omega_1(\tr)\right|=\delta/3.
\end{equation}
Assume that $n\in\N$ is such that $\z_n\in O$. Denote by $\tr_{n1}$ a point in $\RS\setminus O$ satisfying \eqref{definitiontn1} with $\z$ replaced by $\z_n$. Then, since $\Omega_1(O_{\infty^{(1)}})$ is a disk of diameter $\delta/3$, there exist $\w_n,\tr_{n2}\in\partial O_{\infty^{(1)}}$ such that
\begin{equation}
\label{choixwntn2}
\Omega_1(\z_n)-\Omega_1(\tr_{n1}) = \Omega_1(\tr_{n2}) - \Omega_1(\w_n).
\end{equation}
In other words, $\z_n+\w_n-\tr_{n1}-\tr_{n2}$ is a principal divisor by \eqref{eq:abel}, see Figure~\ref{fig:thechoice} for the geometric interpretation of the above construction. 
\begin{figure}[!ht]
\centering
\includegraphics[scale=.4]{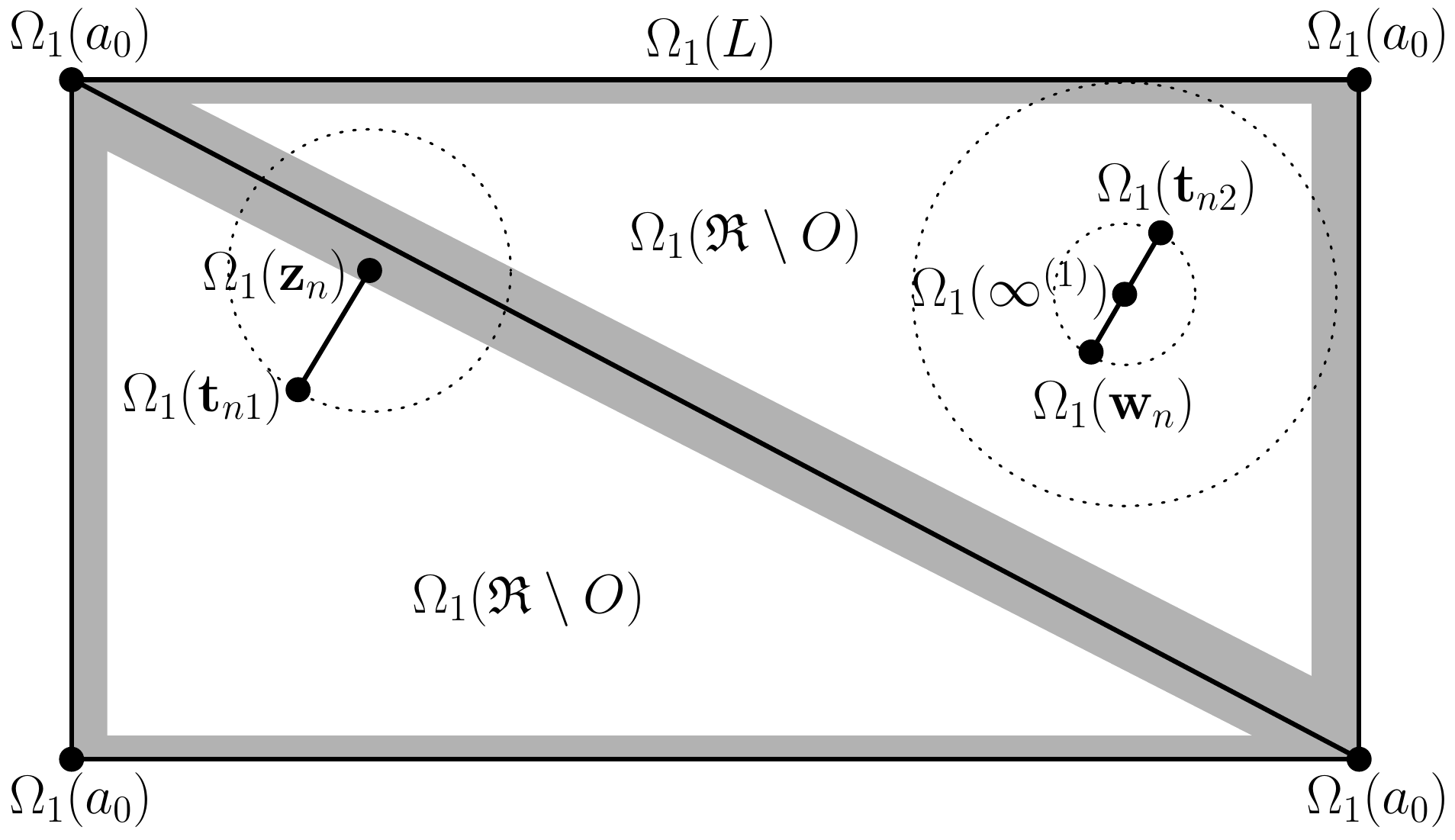}
\caption{\small Schematic geometric interpretation of the construction of the principle divisor $\z_n+\w_n-\tr_{n1}-\tr_{n2}$ representing the image of $\RS$ under $\Omega_1$ as a rectangle (in general, it is a centrally symmetric curvilinear rectangle). The shaded region is $\Omega_1(O)$ and the boundary and the diagonal represent the image of $L$ under $\Omega_1$.}
\label{fig:thechoice}
\end{figure}
Let $U_n$ be the rational function over $\RS$ with this principal divisor normalized so $U_n(\infty^{(1)})=1$. Let us stress that the zero of $U_n$ which is not $\z_n$ and the two poles do not belong to $O$. Observe also that
\[
U_n(\z)/U_n(\z^*) = \prod_{\ar\in\{\z_n,\w_n\}} \exp\left\{\Omega_0(\ar;\z)-\Omega_0(\ar;\z^*)\right\}\prod_{j=1}^2 \exp\left\{\Omega_0(\tr_{nj};\z^*)-\Omega_0(\tr_{nj};\z)\right\}
\]
by the properties of $\Omega_0(\ar;\cdot)$ (see the paragraph between \eqref{eq:OmegaArF} and \eqref{eq:bvOmegaa}). In particular, it follows from \eqref{ratio} that
\begin{equation}
\label{almostthere}
\left|\frac{\map_n(\tr^*)U_n(\tr)}{\map_n(\tr)U_n(\tr^*)}\right| \leq \const \exp\left\{\Omega_0(\w_n;\tr)-\Omega_0(\w_n;\tr^*)\right\}\prod_{j=1}^2 \exp\left\{\Omega_0(\tr_{nj};\tr^*)-\Omega_0(\tr_{nj};\tr)\right\}
\end{equation}
for $\tr\in L^+$, where we used \eqref{eq:lambdanbounded} and the fact that
$|\varphi_{a_1}^{\pm}|=1$ on $L$. 
Since $\Omega_0(\ar;\tr)$ is a continuous function of $\tr\in L^+$ for $\ar\notin L$ and differentials continuously depend on the parameter $\ar$, see \eqref{eq:locallyto0again}, a compactness argument shows that
\begin{equation}
\label{followingLaurent}
1/C_O \leq \left|\Omega_0(\ar;\tr)-\Omega_0(\ar;\tr^*)\right| \leq C_O
\end{equation}
on $L^+$ for all $\ar\notin O$ and some finite non-zero constant $C_O=C_0(O)$.
Combining \eqref{almostthere} and \eqref{followingLaurent} 
with the definition of $S_n$ and the boundedness of $S_{h,\kappa}$, 
see \eqref{definitionSn}, we get that
\begin{equation}
\label{boundtildeXn}
\left|\frac{S_n(\tr^*)U_n(\tr)}{S_n(\tr)U_n(\tr^*)}\right| \leq \const
\end{equation}
for all $t\in L^+$ and $n$ such that $\z_n\in O$, where the constant depends only on $O$.

\subsection{Proof of Proposition~\ref{prop:asympzn}}

By definition, see \eqref{eq:jip-r3}, the point $\z_n$ is the solution of the Jacobi inversion problem
\begin{equation}
\label{431}
\int_{\infty^{(2)}}^{\z_n}d\Omega_1 \equiv n\eqm(\Delta_3) - (n\eqm(\Delta_2)+\gamma)\frac{\beta_3}{\beta_2} = v(n), \quad \mbox{(mod periods)},
\end{equation}
where $v(t):=t\eqm(\Delta_3) - \left(t\eqm(\Delta_2)+\gamma\right)(\beta_3/\beta_2)$. Denote by ${\bf V}$ the set of the limit points of $\{v(n)\}$ in the Jacobi variety Jac$(\RS):=\C/\{n+m(\beta_3/\beta_2)\}$. Since the function $\int_{\infty^{(2)}}^{\z}d\Omega_1$ is a holomorphic bijection of $\RS$ onto Jac$(\RS)$, the set ${\bf Z}$ is equal to $\RS$, is a finite set of points in $\RS$, or is a finite union of pairwise disjoint arcs on $\RS$ if and only if ${\bf V}=$Jac$(\RS)$, is a finite set of points in Jac$(\RS)$, or is a finite union of pairwise disjoint arcs on Jac$(\RS)$, respectively. Clearly, ${\bf V}\subset\{v(t)\}_{t\in\R}$. The structure of the sets ${\bf V}$ and $\{v(t)\}_{t\in\R}$ on Jac$(\RS)$ for an elliptic Riemann surface $\RS$ was analyzed in \cite[Sec. 5]{Suet03}, depending on arithmetic properties of the numbers  $\eqm(\Delta_k)$, $k\in\{1,2,3\}$. More precisely, it was shown that ${\bf V}=$Jac$(\RS)$ when the numbers $\eqm(\Delta_k)$ are rationally independent; that ${\bf V}$ is a finite set of points when the numbers $\eqm(\Delta_k)$ are rational; and that ${\bf V}$ is the union of a finite number of pairwise disjoint arcs when the numbers $\eqm(\Delta_k)$ are rationally dependent but at least one of them is irrational.

The second conclusion of the proposition easily follows from \eqref{431}. Indeed, if $\z_n=\z_{n+m}$ for some $n,m\in\N$, then
\[
m\left(\eqm(\Delta_3) - \eqm(\Delta_2)\frac{\beta_3}{\beta_2}\right) \equiv 0 \quad \mbox{(mod periods)}.
\]
By comparing first imaginary and then real parts of the equation above, we see that $m\eqm(\Delta_k)\in\N$. Hence, whenever at least one of the numbers $\eqm(\Delta_k)$ is irrational all the points $\z_n$ are mutually distinct. However, if all $\eqm(\Delta_k)$ are rational, then $\z_n=\z_{n+m}$ for all $n\in\N$, where $m$ is the least common multiple of their denominators.

It remains for us to show that the set of non-collinear triples $(a_1,a_2,a_3)\in \C^3$ whose corresponding $\eqm(\Delta_k)$ are rationally dependent has Lebesgue measure zero. Our point of departure is a characterization, given in \cite{Kuzmina}, of the Chebotar\"ev center $a_0$ for triples of the form $(a_1,a_2,a_3)=(0,e^{i\alpha},\rho^2e^{-i\alpha})$ with $\alpha\in (0,\pi/2)$ and $0<\rho<1$. For such triples we indicate the dependence  on $(\alpha,\rho)$ by writing $\Delta_k(\alpha,\rho)$ for the arcs constituting the Chebotar\"ev continuum, and we set $\eqm(\Delta_j(\alpha,\rho))$, $j=1,2,3$, for the mass of the equilibrium measure on $\Delta_j(\alpha,\rho)$. Note that $\eqm(\Delta_2(\alpha,\beta))>\eqm(\Delta_3(\alpha,\beta))$ since $\rho<1$ \cite[Thm. 1.4]{Kuzmina}. We also define
\begin{equation}
\label{deflambda}
\left\{
\begin{array}{lcl}
\lambda_1 &:=& \eqm(\Delta_2(\alpha,\rho))+\eqm(\Delta_3(\alpha,\rho)),\smallskip \\
\lambda_2 &:=& \eqm(\Delta_2(\alpha,\rho))-\eqm(\Delta_3(\alpha,\rho)).
\end{array}
\right.
\end{equation}

Let $\C^*/\pm$ be the quotient 
of $\C\setminus\{0\}$ by the equivalence
relation $(z_1\sim z_2)\Longrightarrow z_1=\pm z_2$.
Denote by $U\subset (\R_+)^2$ the set of pairs $(t_1,t_2)$ with
$0<t_2<t_1$ and $t_1+t_2<2$.
From \cite[Thm. 1.6]{Kuzmina},
and its proof, we deduce that $\lambda_1$, $\lambda_2$, and $a_0$ 
are the last components
of a unique $6$-tuple 
$(p,k,\mu, \lambda_1,\lambda_2,a_0)\in
(\C\setminus\{0,\pm1\})\times(\C^*/\pm)\times \C\times U
\times \C$ 
satisfying\footnote{The square root in the second equation of \eqref{CTK} is apparently missing in the statement just quoted; the need for it can be checked from the proof, {\it cf.} equation (1.36) of that reference.}
\begin{equation}
\label{CTK}
\left\{
\begin{array}{l}
p-\left(1+(e^{-i\alpha}+\rho^{-2}e^{i\alpha})a_0+\rho^{-2}a_0^2\right)^{1/2}=0,\\
k-\left(\frac{p+1-(e^{-i\alpha}+\rho^{-2}e^{i\alpha})a_0/2}{2p}\right)^{1/2}=0,\\
{\bf cn}(\mu,k)-\frac{1-p}{1+p}=0,\\
\mu-\lambda_1{\bf K}(k)-i\lambda_2{\bf K}(k')=0,\\
2{\bf K}(k)\left(
\frac{a_0p^{1/2}}{\rho(1+p)}-\frac{\Theta_4'\left(\frac{\mu}{2{\bf K}(k)}|
\tau\right)}{2{\bf K}(k)\Theta_4\left(\frac{\mu}{2{\bf K}(k)}|\tau\right)}\right)
+i\pi\lambda_2=0,
\end{array}
\right.
\end{equation}
where ${\bf K}(k)$ is the complete elliptic integral of the first kind with (complex) modulus $k$ and  $k'=(1-k^2)^{1/2}$ is the complementary modulus, $\Theta_4(z\,|\,\tau)$ is the Jacobi theta function with period 1 and quasi-period $\tau:=i{\bf K}(k')/{\bf K}(k)$, while ${\bf cn}(\cdot,k)$ is the Jacobi elliptic function with periods $4{\bf K}(k)$ and $2{\bf K}(k)+i2{\bf K}(k')$. The principal branch of the square root is used in the first equation, where the quantity under square root there cannot be negative because $a_0$ necessarily lies in the open triangle $K(0, e^{i\alpha}, \rho^2e^{-i\alpha})$. Consequently $p$ actually lies in $\mathcal{H}_+\setminus\{1\}$, where $\mathcal{H}_+$ indicates the open right half plane. The branch used in the second equation is immaterial. The third equation addresses the equivalence class of $\mu$ modulo periods, but the fourth selects a unique representative for $\mu$ because $(t_1,t_2)\in U$. Of necessity, it holds that $k^2\neq1$, that ${\bf K}(k),{\bf K}(k')\neq0,\infty$, and that $\mbox{\rm Im}(i{\bf K}(k')/{\bf K}(k))\neq0,\infty$ \cite[eqns. (1.33)\&(1.60)]{Kuzmina}. In particular, the fourth equation in  \eqref{CTK} yields that
\begin{equation}
\label{calclambda}
\lambda_1=\frac{{\rm Re}\,\left(\overline{\mu}\;{\bf K}(k')\right)}{{\rm Re}\,\left(\overline{{\bf K}(k)}{\bf K}(k')\right)} \quad \mbox{and} \quad \lambda_2=\frac{{\rm Im}\,\big(\overline{\mu}\;{\bf K}(k)\big)}{{\rm Re}\,\left(\overline{{\bf K}(k)}{\bf K}(k')\right)}.
\end{equation}
Set $Z:=\rho e^{-i\alpha}+(\rho e^{-i\alpha})^{-1}$, and note that the map $(\alpha,\rho)\mapsto Z$ is a real analytic homeomorphism from $(0,\pi/2)\times(0,1)$ onto the positive quadrant $\mathcal{Q}:=\{z=x+iy:\,x>0,\,y>0\}$. If we further let $A_0:=a_0/\rho$, then \eqref{CTK} and \eqref{calclambda} provide us with four relations:
\begin{equation}
\label{CTKred}
\left\{
\begin{array}{ccl}
0 &=& p^2-1-ZA_0-A_0^2,\smallskip\\
0 &=& 2k^2p-p-1+ZA_0/2,\smallskip\\
0 &=& {\bf cn}(\mu,k)-(1-p)/(1+p),\smallskip\\
0 &=& \displaystyle 2{\bf K}(k)\left(\frac{A_0p^{1/2}}{1+p}-\frac{\Theta_4'\left(\frac{\mu}{2{\bf K}(k)}|\tau\right)}{2{\bf K}(k)\Theta_4\left(\frac{\mu}{2{\bf K}(k)}|\tau\right)}\right)+i\pi\frac{{\rm Im}\,\big(\overline{\mu}\;{\bf K}(k)\big)}{{\rm Re}\,\left(\overline{{\bf K}(k)}{\bf K}(k')\right)}.
\end{array}
\right.
\end{equation}
From the first two equations in \eqref{CTKred} we obtain
\begin{equation}
\label{A0}
A_0=(p^2-3-2p+4k^2p)^{1/2}
\end{equation}
where the principal branch of the square root is used 
(again the quantity under square root cannot be negative 
 when \eqref{CTKred} holds for $a_0\in K(0,e^{i\alpha},\rho^2e^{-i\alpha})$), 
and we are left with the following system of equations
\begin{equation}
\label{CTKredr}
H_1(p,k,Z) = H_2(p,k,\mu) = H_3(p,k,\mu) = 0
\end{equation}
where
\[
\left\{
\begin{array}{lcl}
H_1(p,k,Z) &:=& 2k^2p-p-1+Z(p^2-3-2p+4k^2p)^{1/2}/2, \smallskip\\
H_2(p,k,\mu) &:=& {\bf cn}(\mu,k)-(1-p)/(1+p), \smallskip \\
H_3(p,k,\mu) &:=& \displaystyle 2{\bf K}(k)\left(\frac{(p^2-3-2p+4k^2p)^{1/2}p^{1/2}}{1+p}-\frac{\Theta_4'\left(\frac{\mu}{2{\bf K}(k)}|\tau\right)}{2{\bf K}(k)\Theta_4\left(\frac{\mu}{2{\bf K}(k)}|\tau\right)}\right)+i\pi\frac{{\rm Im}\,\big(\overline{\mu}\;{\bf K}(k)\big)}{{\rm Re}\,\left(\overline{{\bf K}(k)}{\bf K}(k')\right)}.
\end{array}
\right.
\]
Set $\mathcal{M}$ to be the open subset of $\C^*/\pm$ comprised of
$k\neq\pm1$ 
for which $\mbox{\rm Im}(i{\bf K}(k')/{\bf K}
(k))\neq0,\infty$ ($\C^*/\pm$ being endowed with the quotient topology), 
and let $\mathcal{U}$ be the open subset of 
the analytic manifold
$\Bigl(\mathcal{H}_+\setminus\{1\}\Bigr)\times\mathcal{M}
\times \C$ 
consisting of triples $(p,k,\mu)$ for which
the quantity $-(4k^2p-2p-2)/(4k^2p-2p-3+p^2)$ ({\it i.e.} the value of
$Z$ when $H_1(p,k,Z)=0$) lies in $\mathcal{Q}$
and  
such that the right hand sides of equations \eqref{calclambda} define a 
member of $U$.

Then, solutions $(p,k,\mu,\lambda_1,\lambda_2,a_0)$
to \eqref{CTK} as described above project injectively 
onto the real analytic\footnote{Note that $H_3$ is not complex analytic}
variety $\mathcal{V}\subset\mathcal{U}$ of those  
$(p,k,\mu)$ such that
$H_2(p,k,\mu)=H_3(p,k,\mu)=0$.
Note that $\mathcal{V}$ is distinct from $\mathcal{U}$ since for fixed
$p,k$ the set of $\mu$ with $H_2(p,k,\mu)=0$ is discrete.
By a theorem of Lojaciewicz \cite[Thm. 5.2.3]{KrPa}, as 
$\mathcal{U}$ has real dimension $6$, we get that $\mathcal{V}$ decomposes 
into a disjoint union
$\cup_{j=0}^5 \mathcal{V}_j$ where $\mathcal{V}_j$ 
is a real analytic submanifold of $\mathcal{U}$ of dimension $j$.
Pick $j\in\{0,\ldots,5\}$ and consider the map
$\Phi_j:\mathcal{V}_j\times \mathcal{Q}\to\C$
given by $\Phi_j(p,k,\mu,Z)=H_1(p,k,Z)$.
From \eqref{CTKredr} and \eqref{A0}, we see that
the partial derivative\footnote{We understand by
${\bf D}_Z\Phi_j$ the derivative of $Z\mapsto \Phi_j(p,k,\mu,Z)$ 
as a map from an open subset of $\R^2$
({\it i.e.} $\mathcal{Q}$) into $\R^2\sim\C$.
As $H_1$ is holomorphic in $Z$, this derivative is just multiplication by
the complex number $\partial\Phi_j/\partial Z(p,k,\mu,Z)$ viewed as a 
real linear map on $\R^2$.} ${\bf D}_Z\Phi_j$ is 
multiplication  by 
$A_0/2$, which is bijective at any $(p,k,\mu,Z)$
where $\Phi_j(p,k,\mu,Z)=0$ since $A_0\neq 0$.
Hence, by the transversality theorem \cite[Ch. 2]{GuilleminPollack}, 
it holds for almost every $Z\in \mathcal{Q}$ that the 
partial map $\Phi_{j,Z}(p,k,\mu):= \Phi_j(p,k,\mu,Z)$, defined on 
$\mathcal{V}_j$ with values in $\C$, is transverse to the 
submanifold $\{0\}\subset \C$ which has codimension $2$. 
This means that, for a.e. $Z$, the set $\Phi_{j,Z}^{-1}(0)$ is either
empty or a submanifold
of $\mathcal{V}_j$ of codimension $2$ whose tangent space at 
$(p,k,\mu)\in \Phi_{j,Z}^{-1}(0)$ is the kernel of 
the derivative 
${\bf D}\Phi_{j,Z}(p,k,\mu):{\bf T}_{(p,k,\mu)}\mathcal{V}_j\to\R^2$, 
where ${\bf T}_{(p,k,\mu)}\mathcal{V}_j$ indicates the tangent space to 
$\mathcal{V}_j$ at $(p,k,\mu)$. However, we know from existence and uniqueness
of a solution to \eqref{CTK} that $\Phi_{j,Z}^{-1}(0)$ consists 
of a single point, say $(p_Z,k_Z,\mu_Z)\in\mathcal{U}$. 
Therefore, we get for a.e. $Z$ that
$\Phi_{j,Z}^{-1}(0)=\emptyset$ if $j\neq2$, and that
$(p_Z,k_Z,\mu_Z)\in \mathcal{V}_2$ is such that
${\bf D}\Phi_{2,Z}(p_Z,k_Z,\mu_Z)$ is an isomorphism
from ${\bf T}_{(p_Z,k_Z,\mu_Z)}\mathcal{V}_2$ onto $\R^2$.
Subsequently, by the implicit function theorem, such $Z$ form 
an open set $\mathcal{Z}\subset \mathcal{Q}$ over which
$p_Z,k_Z,\mu_Z$ are real analytic functions of $Z$.

\emph{We claim} that $\lambda_1,\lambda_2$ cannot both be
constant on some nonempty open set $B\subset\mathcal{Z}$. 
Suppose indeed this is 
the case. Then, we see from \eqref{CTK} that $\mu_Z=F(k_Z)$ for $Z\in B$,
where $F$ is a globally defined holomorphic function on 
$(\C^*/\pm)\setminus\{\pm1\}$. In turn, the third equation in 
\eqref{CTKred}
entails that $p_Z=G(k_Z)$ where $G$ is again a globally defined holomorphic 
function on $(\C^*/\pm)\setminus\{\pm1\}$. Thus, by uniqueness of a 
solution, the correspondence $Z\mapsto k_Z$ must be injective from $B$
onto some set $E\subset\C^*/\pm$, and since this correspondence is continuous 
(in fact: real analytic) $E$ is open \cite[Thm. 36.5]{Munkres}.
This shows that the open subset of 
$\mathcal{V}_2$ consisting of triples $(p_Z,k_Z,\mu_Z)$ with $Z\in B$ is 
holomorphically parametrized by $k_Z\in E$ and is in fact a Riemann surface.
In particular, the third equation in \eqref{CTKredr} tells us that
\[
2{\bf K}(k)\left(
\frac{
(G^2(k)-3-2G(k)+4k^2G(k))^{1/2}G^{1/2}(k)}{1+G(k)}-\frac{\Theta_4'\left(\frac{
F(k)}{2{\bf K}(k)}|
\tau\right)}{2{\bf K}(k)\Theta_4\left(\frac{F(k)}{2{\bf K}(k)}|\tau\right)}\right)
+i\pi\lambda_2=0
\]
for all $k\in E$. As the left hand side is a globally defined 
holomorphic function on $(\C^*/\pm)\setminus\{\pm1\}$ 
(remember $\lambda_2$ is assigned to some constant value) it must be the 
zero function. Hence, for \emph{any} $Z\in \mathcal{H}_+\setminus\{1\}$,
the solution to \eqref{CTKredr}, which is known to exists and to be unique,
is obtained from $k$ by plugging
$\mu=F(k)$ and $p=G(k)$ while fixing $\lambda_1$, $\lambda_2$ to the constant 
values they assume on $B$, because then all the equations will be satisfied.
In particular $\lambda_1,\lambda_2$ are constant functions of 
$Z\in \mathcal{H}_+\setminus\{1\}$. 
Back to the original variables, we get that
the $\eqm(\Delta_j(\alpha,\rho))$ are constant 
(remember $\sum_j\eqm(\Delta_j(\alpha,\rho))=1$). However, 
this cannot be because
$\eqm(\Delta_1(\alpha,\rho))=\eqm(\Delta_2(\alpha,\rho))$ when
$\rho^2=1/(2\cos(2\alpha))$ for some $\alpha<\pi/6$ ({\it i.e.} $a_3$
lies on the perpendicular bisector of $[a_1,a_2]$) whereas 
$\eqm(\Delta_1(\alpha,\rho))<\eqm(\Delta_2(\alpha,\rho))$
when $2\rho^2\cos(2\alpha)<1$ ({\it i.e.} if $|a_3-a_1|<|a_3-a_2|$), see
\cite[Thm. 1.4]{Kuzmina}. \emph{This proves the claim}.

Recaping what we did in terms of variables $(\alpha,\rho)$, we find
in view of \eqref{deflambda}and \eqref{A0} that there is an 
open subset 
$\mathcal{W}\subset (0,\pi/2)\times(0,1)$, whose complement has zero measure,
over which $a_0$ and $\eqm(\Delta_j(\alpha,\rho))$, $j=1,2,3$,
are real analytic functions of 
$(\alpha,\rho)$. Moreover, the correspondence
\[
(\alpha,\rho)\mapsto  \big(\eqm(\Delta_1(\alpha,\rho)), \eqm(\Delta_2(\alpha,\rho)\big)
\]
cannot be constant on a nonempty open subset of $\mathcal{W}$.

We now consider triples of the form $(a_1,a_2,a_3)=(0,i,a)$
where $a\in \D_+:=\{z;\,|z|<1,\,{\rm Re}z>0\}$. Accordingly, we put
$\Delta_k(a)$, $k\in\{1,2,3\}$, for the analytic arcs constitutive of the
corresponding Chebotar\"ev continuum, and write $c(a)$ for the Chebotar\"ev 
center.  Observe that the map $a(\alpha,\rho):=i\rho^2 e^{-2i\alpha}$ 
is a real analytic diffeomorphism 
from $(0,\pi/2)\times(0,1)$ onto $\D_+$, and that
the triples $(0,e^{i\alpha},\rho^2 e^{-i\alpha})$ and 
$(0,i,a(\alpha,\rho))$ differ by a rotation of angle $\pi/2-\alpha$.
Thus $c(a(\alpha,\rho))=ie^{-i\alpha}a_0$ and
$\eqm(\Delta_j(a(\alpha,\rho)))=\eqm(\Delta_j(\alpha,\rho))$, $1\leq j\leq3$. 
Moreover, by what precedes,   
there is an open subset $\mathcal{T}\subset \D_+$, with 
$\D_+\setminus\mathcal{T}$ of zero measure, such that the maps $a\mapsto c(a)$
and $a\mapsto\eqm(\Delta_j(a))$, from $\mathcal{T}$ into $\C$ and $\R$ 
respectively,  are real analytic. 
Moreover, $a\mapsto 
(\eqm(\Delta_1(a)), \eqm(\Delta_2(a))$ is not constant 
over a nonempty open subset of $\mathcal{T}$.

Consider the open subset $\mathcal{A}\subset\mathcal{T}\times\D_+$ of those
$(a,b)$ such that $b$ lies in the triangle $K(0,i,a)$. 
For $(a,b)\in\mathcal{A}$,  
consider the quadratic differential
\begin{equation}
\label{defQab}
Q_{a,b}(z):=-\frac{1}{\pi}\frac{z-b}{z(z-i)(z-a)}dz^2.
\end{equation}
Note that $Q_{a,c(a)}$ is minus
the quadratic differential \eqref{QD} where $a_1=0$, $a_2=i$, and $a_3=a$,
so that $Q_{a,c(a}(z)dz^2>0$ on $\Delta_j^\circ(a)$, $j=1,2,3$.
Define further
\begin{equation}
\label{defPsi}
\Psi_1(a,b):=\int_0^b Q_{a,b}^{1/2}(z)\,dz,\qquad
\Psi_2(a,b):=\int_i^b Q_{a,b}^{1/2}(z)\,dz
\end{equation}
where, by Cauchy's theorem,
the integrals may be taken over any smooth path joining $0$ (resp. $i$)
to $b$ whose interior lies in $K(0,i,a)$, and where the branch of the square 
root is positive if $b=c(a)$ and if the path $\Delta_1(a)$ 
(resp. $\Delta_2(a)$) is used. 
By \eqref{eq:eqmeas}, we have that 
\begin{equation}
\label{egDelPsi}
\Psi_1(a,c(a))=\eqm(\Delta_1(a)),\qquad \Psi_2(a,c(a))=\eqm(\Delta_2(a)).
\end{equation}

Writing $a=x_a+iy_a$ to single out real and imaginary parts, we
introduce differential operators 
$\partial_a:=(\partial_{x_a}-i\partial_{y_a})/2$ and
$\overline{\partial}_a:=(\partial_{x_a}+i\partial_{y_a})/2$. We define
$\partial_b$ and $\overline{\partial}_b$ similarly.
Those $a$ for which $\overline{\partial}_a c(a)=0$ form
a real analytic variety, say $\mathcal{X}\subset\mathcal{T}$.
\emph{We claim} that $\mathcal{X}$ has measure zero. To prove this, 
it is enough by Lojaciewicz's theorem to show that $\mathcal{X}$ has 
no interior. Assume for a contradiction that it contains an open set
$V\neq\emptyset$, so that $c(a)$ is a holomorphic function of $a\in V$.
Then, by inspection of \eqref{defQab}-\eqref{defPsi}, the function
$\Psi_j(a,c(a))$ is in turn holomorphic on $V$ for $j=1,2$. 
However, it is real valued by \eqref{egDelPsi} hence it must be constant. 
But then,
$a\mapsto (\eqm(\Delta_1(a)),\eqm(\Delta_2(a)))$ is constant over $V$
which is impossible, as pointed out earlier. 
\emph{This proves the claim}. Thus, $\mathcal{A}:=\mathcal{T}\setminus\mathcal{X}$ is open and 
$\D_+\setminus\mathcal{A}$ has zero measure. 

Next, since $\Psi_j(a,b)$ is holomorphic in $a$,$b$, it holds that
$\overline{\partial}_a\Psi_j(a,b)=\overline{\partial}_b\Psi_j(a,b)=0$,
so by \eqref{egDelPsi} and the chain rule
\begin{equation}
\label{derb1}
\overline{\partial}_a\eqm(\Delta_1(a))
=\left(\partial_b\Psi_1(a,b)_{|b=c(a)}\right)\overline{\partial}_a c(a)=
\frac{\overline{\partial}_a c(a)}{2}\int_0^{c(a)} \frac{Q_{a,c(a)}^{1/2}(z)}{z-c(a)}
\,dz
\end{equation}
and likewise
\begin{equation}
\label{derb2}
\overline{\partial}_a\eqm(\Delta_2(a))
=\frac{\overline{\partial}_a c(a)}{2}\int_i^{c(a)} \frac{Q_{a,c(a)}^{1/2}(z)}{z-C(a)}
\,dz.
\end{equation}
Fix $\Delta_1(a)$ to be the integration path in
\eqref{derb1}, and let $z_3$ be the intersection of the straight line 
through $a$, $c(a)$ with the segment $[0,i]$. Since $a$ is strictly closer to 
$0$  than $i$, it follows from \cite[Thm. 4.1.]{Kuzmina} that 
$\Delta_1(a)$ is included in the closure of
the triangle $K(0,z_3,c(a))$ 
but not in its boundary. Therefore, since $Q_{a,c(a)}^{1/2}dz>0$ on
$\Delta_1^\circ(a)$, the integral in \eqref{derb1} lies in 
$\overline{C(-c(a),0,z_3-c(a))}$, the complex conjugate of the open
positive cone $C(-c(a),0,z_3-c(a))$ with vertex $0$ generated by the triangle
$K(-c(a),0,z_3-c(a))$. Similarly, if we let $z_1$ be the intersection of 
the straight line through $0$, $c(a)$ with the segment $[i,a]$, we get
from \cite[Thm. 4.1.]{Kuzmina} that $\Delta_2(a)$ is contained\footnote{
$\Delta_2(a)$ lies in the closure of $K(i,c(a),z_1$ or of $K(i,c(a),z_3)$
according whether $0$ is closer to $i$ than $a$ or not; if $|i-a|=|i|=1$,
then $\Delta_2(a)$ is the segment $[c(a),i]$.} in the 
closure of $K(i,c(a),z_1)\cup K(i,c(a),z_3)$, hence 
the integral in \eqref{derb2} lies in the closure of
$\overline{C(z_1-c(a),0,z_3-c(a))}$. Since the latter is disjoint from
$\overline{C(-c(a),0,z_3-c(a))}\cup\left(-\overline{C(-c(a),0,z_3-c(a))}\right)$, we deduce that the integrals in \eqref{derb1} and \eqref{derb2} define
two complex numbers that are linearly independent over $\R$. If in addition 
$a\in\mathcal{A}$,  we conclude since $\overline{\partial}_ac(a)\neq0$ that 
$\overline{\partial}_a\eqm(\Delta_1(a))$ and
$\overline{\partial}_a\eqm(\Delta_2(a))$ are in turn linearly independent 
over $\R$. By definition of $\overline{\partial}_a$, this means that
the map $\Lambda(a):=(\eqm(\Delta_1(a)), \eqm(\Delta_1(a)))$ from
$\D_+$ into $\R^2$ has nonsingular derivative at each point of $\mathcal{A}$,
in particular its restriction to $\mathcal{A}$ is locally a bi-Lipschitz 
homeomorphism. From this, as $\mathcal{A}$ is open of full measure in 
$\D_+$, it is 
elementary to check that  if 
$E\subset\Lambda(\D_+))$ has measure zero (resp. is dense in $\Lambda(\D_+)$) 
then
its inverse image $\Lambda^{-1}(E)$ has measure zero (resp. is dense in 
$\D_+$). 

Now, since
$\eqm(\Delta_1(a))+\eqm(\Delta_2(a))+\eqm(\Delta_3(a))=1$, the 
$\eqm(\Delta_j(a))$ are rationally dependent if and only if there exist
integers $n_1,n_2$, not both zero,  for which
$n_1\eqm(\Delta_1)+ n_2\eqm(\Delta_2)\in\Q$. To each nonzero pair 
of integers $(n_1,n_2)$, we can pick real numbers
$t_1$, $t_2$ with $n_1t_2-n_2t_1\neq0$, hence
the subset of $\R^2$ comprised of $(x,y)$ such that $n_1x+n_2y\in\Q$
has measure zero, being the inverse image under 
$(x,y)\mapsto (n_1x+n_2y,t_1x+t_2y)$ of those points whose first 
coordinate is rational. Consequently the set $\mathcal{Y}\subset\R^2$ of those
$(x,y)$  for which there exists a nonzero pair of integers $(n_1,n_2)$
such that  $n_1x+n_2y\in\Q$ has measure zero as countable union of sets 
of measure zero. Note 
that $\mathcal{Y}$ contains the dense subset of rational pairs.
Thus, by properties of the map $\Lambda$ 
we just proved, the set of $a\in\D_+$ for which the triple
$(0,i,a)$ has $\Q$-linearly independent (resp dependent, rational)
$\eqm(\Delta_j(a))$ has full measure (resp. is dense) in $\D_+$. 
Because the $\eqm(\Delta_k)$ are invariant under 
nonsingular affine 
transformations of the triple $(a_1,a_2,a_3)$ and their conjugates
\cite[Thm. 5.1.2]{Ransford}, it follows easily that, for \emph{any}
pair $(a_1,a_2)\in\C^2$, the set of non-collinear $a_3$ for which the triple
$(a_1,a_2,a_3)$ has $\Q$-linearly independent (resp. $\Q$-linearly dependent, 
rational)
$\eqm(\Delta_j)$ has complement of measure zero
(resp. is dense) in $\C$.  Proposition \ref{prop:asympzn} is now
a consequence of Fubini's theorem. \qed

\section{Asymptotics of Pad\'e Approximants}
\label{sec:proofs}

\subsection{Integral Representation of the Error}
\label{subs:error}

Let $f_h$ be given by \eqref{eq:CauchyT} and $\pi_n=p_n/q_n$ be the $n$-th Pad\'e approximant to $f_h$. Then we deduce from \eqref{eq:linsys} that
\[
\oint_\Gamma z^k(q_nf_h-p_n)(z)dz = 0, \quad k\in\{0,\ldots,n-1\},
\]
by Cauchy integral formula applied in the exterior of $\Gamma$, where $\Gamma$ is any positively oriented Jordan curve encompassing $\Delta$. Applying Cauchy integral formula once more, this time in the interior of $\Gamma$, we get that
\[
0 = \oint_\Gamma z^k(q_nf_h)(z)dz = \oint_\Gamma z^kq_n(z)\frac{1}{\pi i}\int_\Delta\frac{h(t)}{t-z}\frac{dt}{w^+(t)}dz.
\]
Further, using the Fubini-Tonelli and Cauchy integral theorems, we obtain that
\begin{equation}
\label{eq:ortho}
0 = \frac{1}{\pi i}\int_\Delta h(t)\oint_\Gamma\frac{z^kq_n(z)}{t-z}dz\frac{dt}{w^+(t)} = 2\int_\Delta t^kq_n(t)\frac{h(t)dt}{w^+(t)},  \quad k\in\{0,\ldots,n-1\}.
\end{equation}
Thus, polynomials $q_n$, the denominators of Pad\'e approximants $\pi_n$, satisfy non-Hermitian orthogonality relations on $\Delta$ with respect to the weight $h/w^+$.

For each polynomial $q_n$ we define its function of the second kind by the rule
\begin{equation}
\label{eq:fsk}
R_n(z) := \frac{1}{\pi i}\int_\Delta\frac{q_n(t)h(t)}{t-z}\frac{dt}{w^+(t)}, \quad z\in D.
\end{equation}
It can be seen by developing $1/(t-z)$ in powers of $z$ at infinity and using \eqref{eq:ortho} that
\begin{equation}
\label{eq:wRn-zeroing}
(wR_n)(z) = O(z^{-n+1}) \quad \mbox{as} \quad z\to\infty.
\end{equation}
Moreover, it also easily follows from \eqref{eq:ortho} that
\[
\int_\Delta\frac{q_n(t)-q_n(z)}{t-z}q_n(t)h(t)\frac{dt}{w^+(t)} = 0
\]
and therefore
\[
R_n(z) = \frac{1}{q_n(z)}\frac{1}{\pi i}\int_\Delta\frac{q_n^2(t)h(t)}{t-z}\frac{dt}{w^+(t)}.
\]

Applying Cauchy integral theorem to $q_n(q_nf_h-p_n)$ on the bases of \eqref{eq:linsys}, we get that
\[
e_n(z) := (f_h-\pi_n)(z)= \frac{q_n(z)(q_nf_h-p_n)(z)}{q_n^2(z)} = \frac{1}{q_n^2(z)}\frac{1}{2\pi i}\oint_\Gamma\frac{q_n(\tau)(q_nf_h-p_n)(\tau)}{z-\tau}d\tau
\]
for $z$ in the exterior of $\Gamma$. Hence, we derive from Cauchy integral formula and Fubini-Tonelli theorem that
\[
e_n(z) =\frac{1}{q_n^2(z)}\frac{1}{\pi i}\int_\Delta h(t)\frac{1}{2\pi i}\oint_\Gamma\frac{q_n^2(\tau)}{(z-\tau)(t-\tau)} d\tau\frac{dt}{w^+(t)}
\]
and hence
\begin{equation}
\label{eq:error}
e_n(z) = \frac{1}{q_n^2(z)}\frac{1}{\pi i}\int_\Delta\frac{q_n^2(t)h(t)}{t-z}\frac{dt}{w^+(t)} = \frac{R_n(z)}{q_n(z)}.
\end{equation}
Thus, to describe the behavior of the error of approximation, we need to analyze the asymptotic behavior of $q_n$ and $R_n$.

\subsection{Boundary Value Problem} 

According to \eqref{eq:RDelta}, it holds that
\[
R_\Delta(q_nh;z) = \frac{(wR_n)(z)}{2}, \quad z\in D.
\]
Since $q_nh$ is Dini-continuous on $\Delta$, $wR_n$ has unrestricted continuous boundary values on $\partial D$.  Furthermore, it follows from \eqref{eq:onDelta} that
\begin{equation}
\label{eq:wRn-bvp}
(wR_n)^++(wR_n)^- = 2q_nh \quad \mbox{on} \quad \Delta.
\end{equation}
Below, we turn this boundary value problem on $\Delta$ into a boundary value problem on $L$.

Firstly, set
\[
A_n(\z) := \frac{(wR_n)(z)}{S_n(\z)}, \quad \z\in D^{(2)},
\]
where $S_n$ is the function granted by Proposition~\ref{prop:sn}.
Since $S_n$ vanishes at $\infty^{(2)}$ with order $n-1$ when $\z_n\neq\infty^{(2)}$ and with order $n$ otherwise, and $wR_n$ vanishes at infinity with order at least $n-1$ by \eqref{eq:wRn-zeroing}, $A_n$ is a holomorphic function in $D^{(2)}$ except for a single simple pole at $\z_n$ when $\z_n\in D^{(2)}\setminus\{\infty^{(2)}\}$ and a possible simple pole at $\infty^{(2)}$ when $\z_n=\infty^{(2)}$. Moreover, $A_n$ has continuous trace on $L^-\setminus\{\z_n\}$. 

Secondly, put
\[
B_n(\z) := \frac{2q_n(z)}{S_n(\z)}, \quad \z\in\RS\setminus L.
\]
Then $B_n$ is a holomorphic function in $\RS\setminus(L\cup\{\z_n\})$ with a simple pole at $\z_n$ when $\z_n\notin L\cup\{\infty^{(1)}\}$ and a possible simple pole at $\infty^{(1)}$ when $\z_n=\infty^{(1)}$ and $\deg(q_n)=n$. Moreover, $B_n$ has continuous traces on both sides of $L\setminus\{\z_n\}$.

Thus, by the very definition of functions $A_n$ and $B_n$, we have that
\begin{equation}
\label{eq:bs1}
A_n^-(\tr)-B_n^+(\tr) = \frac{(wR_n)^\pm(t)}{S_n^-(\tr)} - \frac{2q_n(t)}{S_n^+(\tr)} = \frac{(wR_n)^\pm(t)}{S_n^-(\tr)} - \frac{2(q_nh)(t)}{S_n^-(\tr)} = -\frac{(wR_n)^\mp(t)}{S_n^-(\tr)},
\end{equation}
where, as usual, $t=\pi(\tr)$ and we used \eqref{eq:jumpL} and \eqref{eq:wRn-bvp}.

Thirdly, define $A_n(\z)=A_n(\z^*)$, $\z\in D^{(1)}$, where $\z^*$ is the point conjugate to $\z$. Clearly, $A_n$ enjoys in $D^{(1)}$ the same properties as in $D^{(2)}$. As to the boundary values on $L$, it holds that $A_n^+(\tr)=A_n^-(\tr^*)$ and therefore
\[
A_n^-(\tr) = \frac{(wR_n)^\pm(t)}{S_n^-(\tr)} \quad \mbox{and} \quad A_n^+(\tr) = \frac{(wR_n)^\mp(t)}{S_n^-(\tr^*)},
\]
where $t\in\Delta^\pm$. Hence, boundary value problem \eqref{eq:bs1} can be rewritten as
\begin{equation}
\label{eq:bs2}
A_n^-(\tr)-B_n^+(\tr) = -A_n^+(\tr)\frac{S_n^-(\tr^*)}{S_n^-(\tr)}, \quad \tr\in L.
\end{equation}
Finally, define
\[
X_n(\z) = \frac{S_n(\z^*)}{S_n(\z)}, \quad \z\in\RS\setminus L.
\]
Then $X_n$ is a sectionally  holomorphic function in $\RS\setminus(L\cup\{\z_n\}\cup\{\infty^{(2)}\})$, it vanishes at $\infty^{(1)}$ with order at least $2(n-1)$ and has continuous traces on $L\setminus\{\z_n\}$. Moreover, it holds that $X_n^+(\tr)=S_n^-(\tr^*)/S_n^+(\tr)$. Thus, we get from \eqref{eq:bs2} and \eqref{eq:jumpL} that
\begin{equation}
\label{eq:bvp-L}
A_n^- = B_n^+ - \frac{(A_nX_n)^+}{h\circ\pi}, \quad \mbox{on} \quad L,
\end{equation}
which is our final boundary value problem.

\subsection{Proof of Theorem~\ref{thm:bernstein-szego}}
Since $h\equiv1/p$, \eqref{eq:bvp-L} becomes
\[
A_n^- = (B_n-PA_nX_n)^+,
\]
where $P:=p\circ\pi$ is the lift of $p$ onto $\RS$. Observe that the left-hand side of the equality above is given by a function meromorphic in $D^{(2)}$ and the right-hand side is given by a function meromorphic in $D^{(1)}$ and holomorphic at $\infty^{(1)}$ for all $2(n-1)>\deg(p)$ (recall that $X_n$ vanishes at $\infty^{(1)}$ with order at least $2(n-1)$). As they have continuous boundary values on $L\setminus\{\z_n\}$ from within the respective domains and at $\z_n$ they have a polar singularity, the function
\[
\Phi_n := \left\{
\begin{array}{ll}
B_n-PA_nX_n, & \mbox{in} \quad D^{(1)}, \smallskip \\
A_n, & \mbox{in} \quad D^{(2)},
\end{array}
\right.
\]
is rational over $\RS$. Observe now that
\[
(PA_nX_n)(\z) = \frac{(pwR_n)(z)}{S_n(\z)}, \quad \z\in D^{(1)},
\]
and hence $\Phi_n$ has at most one simple pole at $\z_n$. However, there are no rational functions over $\RS$ with one pole and therefore $\Phi_n$ is a constant. As $pX_n$ vanishes at $\infty^{(1)}$ , $\Phi_n\equiv2$ by the normalization of $S_n$ and the definition of $B_n$. Summarizing, we derived that
\begin{equation}
\label{final}
2 \equiv A_n(z^{(2)}) = \frac{(wR_n)(z)}{S_n(z^{(2)})} \quad \mbox{and} \quad 2 \equiv \frac{2q_n(z)}{S_n(z^{(1)})}-2\frac{p(z)S_n(z^{(2)})}{S_n(z^{(1)})}.
\end{equation}
Hence, Theorem~\ref{thm:bernstein-szego} follows from \eqref{eq:error}. \qed

\subsection{Proof of Theorem~\ref{thm:pade}}

Let $p_n$ be the best uniform approximant to $1/h$ on $\Delta$ among all polynomial of degree at most $n$. Then the norms $\|p_n\|_\Delta$ are uniformly bounded and \begin{equation}
\label{BW}
\left|p_n(z)\map^{-n}(z)\right| \leq \const,\qquad z\in D,
\end{equation} 
by the Bernstein-Walsh inequality. Set $P_n:=p_n\circ\pi$ and define
\begin{equation}
\label{RHPhi}
\Phi_n := \left\{
\begin{array}{ll}
B_n-P_nA_nX_n, & \mbox{in} \quad D^{(1)}, \smallskip \\
A_n, & \mbox{in} \quad D^{(2)}.
\end{array}
\right.
\end{equation}
Analyzing $\Phi_n$ as in the proof of Theorem~\ref{thm:bernstein-szego}, we see that $\Phi_n$ is a sectionally meromorphic function on $\RS\setminus L$ with at most one pole, necessarily at $\z_n$, and continuous traces on both sides of $L\setminus\{\z_n\}$ that satisfy
\begin{equation}
\label{jump1}
\Phi_n^+ - \Phi_n^- = \epsilon_n(A_nX_n)^+,
\end{equation}
where $\epsilon_n:=(1/h-p_n)\circ\pi$. Observe that $\max_L|\epsilon_n|=\omega_n$ by the definition of $\omega_n$, see \eqref{omegan}. 

Let $O$ be a neighborhood of $L$ in $\RS$ as in
Section~\ref{ss:specialjip}. Consider first those indices $n$ for which $\z_n\notin O$. Assume in addition that $\z_n\neq\infty^{(1)}$. Let $T_n$ be
the constant 0 if $\z_n=\infty^{(2)}$, and if $\z_n\neq\infty^{(2)}$ 
the unique rational function on $\RS$ vanishing at $\infty^{(1)}$ and
having two simple poles at $\z_n$ and $\infty^{(2)}$, normalized so that 
$T_n(\z)/z\to1$ as $\z\to\infty^{(2)}$. Put
\begin{equation}
\label{firstEn}
2E_n := F_n - \ell_n + 2u_nT_n,
\end{equation}
where $F_n(\z):=F_{\epsilon_n(A_nX_n)^+}(\z)$, $\ell_n:=\ell_{\epsilon_n(A_nX_n)^+}$, and $u_n:=u_{\epsilon_n(A_nX_n)^+}$, see \eqref{eq:Fchi}--\eqref{eq:ellchi}. Then $E_n$ is a sectionally meromorphic function on $\RS\setminus L$
with at most one pole at $\z_n$ and well-defined boundary values on both sides of $L$ that satisfy a.e.
\begin{equation}
\label{jump2}
2E_n^+ - 2E_n^- = \epsilon_n(A_nX_n)^+.
\end{equation}
Hence, we derive from \eqref{jump1} and \eqref{jump2} that $\Phi_n = B_n(\infty^{(1)}) + 2E_n$ by the principle of meromorphic continuation and since there are no rational functions over $\RS$ with one pole. Moreover, it holds that $B_n(\infty^{(1)})=2$ when $\deg(q_n)=n$ since $q_n$ is monic and $S_n(\z)=z^n+\dots$ at $\infty^{(1)}$ and $B_n(\infty^{(1)})=0$ otherwise.

If $\deg(q_n)<n$, then
\begin{equation}
\label{An1}
\|A_n^-\|_{2,L} = \|\Phi_n^-\|_{2,L} = \|2E_n^-\|_{2,L}.
\end{equation}
Clearly, for $\z_n\notin O$ the values of $T_n$ on $L$ form a uniformly bounded family of continuous functions. Hence, it follows from \eqref{eq:BK5},
\eqref{eq:ellchi} and \eqref{firstEn} that
\begin{equation}
\label{An2}
\|E_n^-\|_{2,L}\leq \const \omega_n\|(A_nX_n)^+\|_{2,L}.
\end{equation}
Moreover, since for such $\z_n$ the traces $X_n^+$ form a normal family on $L$ by \eqref{eq:conjfunSn2} and $A_n^+(\tr)=A_n^-(\tr^*)$ on $L$ by the definition of $A_n$, \eqref{An1} and \eqref{An2} yield that
\begin{equation}
\label{An3}
\|A_n^+\|_{2,L} \leq \const \omega_n\|A_n^+\|_{2,L}.
\end{equation}
However $\omega_n\to0$ and therefore \eqref{An3} cannot be true for large $n$. That is, for all $n$ large enough and $\z_n\in\RS\setminus O$, it holds that
$\z_n\neq\infty^{(1)}$ and $\deg(q_n)=n$. Thus, for such $n$, we have that $B_n(\infty^{(2)})=2$. Then, we get by repeating the steps \eqref{An1}--\eqref{An3} that
\begin{equation}
\label{An4}
\|A_n^+-2\|_{2,L} \leq \const \omega_n\|A_n^+\|_{2,L},
\end{equation}
which yields that $\|A_n^+\|_{2,L} \to 2$ as $n\to\infty$ for all admissible $n$. In particular, it holds that
\begin{equation}
\label{estimate}
\|E_n^\pm\|_{2,L} \leq \const \omega_n.
\end{equation}
In another connection, we deduce as in \eqref{final} that
\begin{equation}
\label{asymptotics1}
\left\{
\begin{array}{lll}
\displaystyle q_n(z)/S_n(z^{(1)}) &=& \displaystyle 1 + \frac{p_n(z)S_n(z^{(2)})(1+E_n(z^{(2)}))}{S_n(z^{(1)})} + E_n(z^{(1)}), \bigskip \\
\displaystyle (wR_n)(z)/S_n(z^{(2)}) &=& \displaystyle 2\left(1+E_n(z^{(2)})\right), 
\end{array}
\right.
\end{equation}
which implies \eqref{asymptotics} by \eqref{eq:error}, \eqref{estimate} and \eqref{eq:conjfunSn2} (note that $S_nE_n$ is holomorphic in $D^{(2)}$).

Suppose now that $\z_n=\infty^{(1)}$. If $\deg(q_n)=n-1$ then $q_{n-1}=q_n$ and $R_{n-1}=R_n$. As $S_n=S_{n-1}$ and $\z_{n-1}=\infty^{(2)}$ 
by Proposition~\ref{prop:sn}, the asymptotics of $q_n$ and $R_n$ is described by what precedes. Hence, in what follows we can assume that $n$ belongs to an
infinite subsequence such that $\z_n=\infty^{(1)}$ and $\deg(q_n)=n$. As before, one can check that 
\begin{equation}
\label{newas}
\frac{\Phi_n(\z)}{\eta_n} = 2 + \frac{2z+F_n(\z)-\ell_n(z)}{\eta_n}, \quad \z\in\RS\setminus L,
\end{equation}
where $B_n(z)=2z+2\eta_n+\mathcal{O}(1/z)$. Recall that $\ell_n(z)=u_nz+v_n$ by the very definition \eqref{eq:ellchi} and $F_n(\z)=-\ell_n(z)+\mathcal{O}(1/z)$ near $\infty^{(2)}$. Thus, $u_n=1$ and therefore
\begin{equation}
\label{An5}
1\leq \const\omega_n\|A_n^+\|_{2,L}
\end{equation}
as in \eqref{An2}. On the other hand, we get from \eqref{newas} as in \eqref{An4} that
\begin{equation}
\label{An6}
\|\eta_n^{-1}A_n^+-2\|_{2,L}  = \|\eta_n^{-1}(F_n^--\ell_n)\|_{2,L} \leq \const\omega_n\|\eta_n^{-1}A_n^+\|_{2,L}.
\end{equation}
Hence, we have by \eqref{An5} that
\begin{equation}
\label{An7}
\|\eta_n^{-1}A_n^+\|_{2,L}\to2 \quad \mbox{and} \quad |\eta_n^{-1}|\leq\const\omega_n.
\end{equation}
Set
\[
E_n(\z) := \frac{2z+F_n(\z)-\ell_n(z)}{2\eta_n}, \quad \z\in\RS\setminus L.
\]
Then $\|E_n^\pm\|_{2,L}\leq\const\omega_n$ by \eqref{An6} and \eqref{An7}. Moreover, it holds that
\[
\left\{
\begin{array}{lll}
\displaystyle q_n(z)/S_n(z^{(1)}) &=& \displaystyle \eta_n\left(1 + \frac{p_n(z)S_n(z^{(2)})(1+E_n(z^{(2)}))}{S_n(z^{(1)})} + E_n(z^{(1)})\right), \bigskip \\
\displaystyle (wR_n)(z)/S_n(z^{(2)}) &=& \displaystyle 2\eta_n\left(1+E_n(z^{(2)})\right), 
\end{array}
\right.
\]
which yields \eqref{asymptotics} by \eqref{eq:error}.

Assume next that $n$ ranges over an infinite subsequence with $\z_n\in O$.
Define $\widetilde\Phi_n := \Phi_nU_n$, where the functions $U_n$ were constructed in Section~\ref{ss:specialjip}. Recall that $U_n(\infty^{(1)})=1$, $(U_n)=\z_n+\w_n-\tr_{n1}-\tr_{n2}$, and $\w_n,\tr_{n1},\tr_{n2}\notin O$. Therefore, $\widetilde\Phi_n$ is a meromorphic function on $\RS\setminus L$, with a zero at $\w_n$, two poles at $\tr_{n1}$ and $\tr_{n2}$, and $\widetilde\Phi_n(\infty^{(1)})=B_n(\infty^{(1)})$. Furthermore, since $U_n$ is holomorphic across $L$, it holds that
\begin{equation}
\label{sautPhit}
\widetilde\Phi_n^+ - \widetilde\Phi_n^- = \epsilon_n(\widetilde A_n\widetilde X_n)^+
\end{equation}
by \eqref{jump1}, where we set
\[
\begin{array}{lll}
\widetilde A_n(z^{(2)}) &:=& A_n(z^{(2)})U_n(z^{(2)}), \smallskip \\
\widetilde A_n(z^{(1)}) &:=& \widetilde A_n(z^{(2)}),
\end{array}
\quad \mbox{and} \quad 
\begin{array}{lll}
\widetilde B_n(z^{(1)}) &:=& B_n(z^{(1)})U_n(z^{(1)}), \smallskip \\
\widetilde X_n(z^{(1)}) &:=& X_n(z^{(1)})U_n(z^{(1)})/U_n(z^{(2)}),
\end{array}
\]
for $z\in D$. As in \eqref{firstEn}, define
\[
2\widetilde E_n := \widetilde F_n - \widetilde\ell_n + 2\widetilde u_n \widetilde T_n,
\] 
where $\widetilde F_n:=F_{\epsilon_n(\widetilde A_n\widetilde X_n)^+}$, $\widetilde\ell_n:=\ell_{\epsilon_n(\widetilde A_n\widetilde X_n)^+}$, $\widetilde u_n:=u_{\epsilon_n(\widetilde A_n\widetilde X_n)^+}$, and $\widetilde T_n$ is a rational function over $\RS$ with a zero at $\infty^{(1)}$, two poles at $\infty^{(2)}$ and $\tr_{n2}$, and normalized so $\widetilde T_n(z^{(2)})/z\to$ as 
$z\to\infty$. Note from \eqref{SPC} that across $L$
\begin{equation}
\label{sautEnt}
\widetilde E_n^+ - \widetilde E_n^- = \epsilon_n(\widetilde A_n\widetilde X_n)^+.
\end{equation}
 As $\Omega_1$ is Lipschitz on $\widetilde\RS_2$
with respect to any fixed Riemannian metric on $\RS$,
it follows from \eqref{choixwntn2} and \eqref{definitiontn1} (where $\z=\z_n$)
that $\widetilde T_n(\w_n)$ is bounded independently of $n$ (along the 
considered subsequence), see Figure~\ref{fig:thechoice}. 
Moreover, $|w_n|$ remains bounded in $D$ because
$\w_n\in\partial O_{\infty^{(1)}}$, see \eqref{defOinfini}. 
Hence, from \eqref{eq:Fchi}, \eqref{eq:ellchi}, and
since  the traces $\widetilde X_n^+$ are uniformly bounded on $L$ 
by \eqref{boundtildeXn}, we deduce that
\begin{equation}
\label{estimEtwn}
|\widetilde E_n(\w_n)| \leq \const\omega_n\|\widetilde A_n^+\|_{2,L}.
\end{equation}
Likewise, as $\tr_{n2}\notin O$, it holds that 
$|\widetilde T_n|$ is bounded on $L$, therefore by 
\eqref{eq:BK5} and \eqref{eq:ellchi} again
\begin{equation}
\label{estimEtn}
\|\widetilde E_n^+\|_{2,L} \leq \const\omega_n\|\widetilde A_n^+\|_{2,L}.
\end{equation}
Now, from \eqref{sautPhit} and \eqref{sautEnt} we get that
 $\widetilde\Phi_n-2\widetilde E_n$ has no jump across $L$, and since it 
can only have poles at $\tr_{n1}$, $\tr_{n2}$ it must be a scalar multiple of 
$U_n$. Checking values at $\infty^{(1)}$ and $\w_n$
(remember $\widetilde\Phi_n(\w_n)=U_n(\w_n)=0$ and 
$\widetilde\Phi_n(\infty^{(1)})=B_n(\infty^{(1)})$), we conclude that
\begin{equation}
\label{calcPhit}
\widetilde\Phi_n=B_n(\infty^{(1)})+2\widetilde E_n+c_n(1-U_n),
\end{equation}
where $c_n=2\widetilde E_n(\w_n)-B_n(\infty^{(1)})$. 
If we had $B_n(\infty^{(1)})=0$ for all $n$ large enough 
(within the considered subsequence), it would hold that
$c_n=2\widetilde E_n(\w_n)$ and, by \eqref{RHPhi}
and \eqref{calcPhit}, that
$\widetilde A_n^- = 2\widetilde E_n^- + c_n(1-{U_n}_{|L})$. 
Because $U_n$ is bounded on $L$ independently of $n$
as $\tr_{n2}\notin O$, and since 
$\|\widetilde A_n^-\|_{2,L}=\|\widetilde A_n^+\|_{2,L}$ by construction,
this would entail with  \eqref{estimEtwn},\eqref{estimEtn}, \eqref{sautEnt}
that
\begin{equation}
\label{finalAn}
\|\widetilde A_n^+\|_{2,L} \leq \const\omega_n\|\widetilde A_n^+\|_{2,L}.
\end{equation}
But $\omega_n\to0$, thus \eqref{finalAn} is impossible for $n$ large enough, 
therefore $B_n(\infty^{(1)})=2$ and $\deg(q_n)=n$ for all such $n$.
Repeating the arguments leading to \eqref{finalAn}, this time 
with $B_n(\infty^{(1)})=2$, we get that
\begin{equation}
\label{thelastbound}
|2+c_n|\leq \const\omega_n\|\widetilde A_n^+\|_{2,L} \quad \mbox{and} \quad \|\widetilde A_n^+-2U_n\|_{2,L} \leq \const\omega_n\|\widetilde A_n^+\|_{2,L}.
\end{equation}
The last inequality implies that $\|\widetilde A_n^+\|_{2,L}$ is
uniformly bounded. Rewrite \eqref{calcPhit} as
\[
\widetilde\Phi_n = 2U_n+2\widetilde E_n + (2+c_n)(1-U_n)
\]
or equivalently
\[
\Phi_n = 2 + 2\left(\widetilde E_n + (1+c_n/2)(1-U_n)\right)U_n^{-1} =:2\left(1 + E_n\right).
\]
By its very definition, $E_n$ has a pole at $\z_n$ but is holomorphic at 
$\w_n$ (even though $U_n^{-1}$ is not). Formula \eqref{asymptotics} now follows exactly as \eqref{asymptotics1}. 
To show \eqref{errorintegral}, observe from the definition of $E_n$ and since
$|U_n^{-1} (l_n\circ\pi)|$ is bounded on $L$ that
\[
\|2E_n^\pm (l_n\circ\pi)\|_{2,L} \leq \const\|\widetilde E_n^\pm + (1+c_n/2)(1-U_n)\|_{2,L} \leq \const\omega_n,
\]
where the last estimate follows from \eqref{estimEtn}, \eqref{thelastbound}, and the boundedness of $U_n$ on $L$. \qed

\subsection{Proof of Corollary~\ref{convsubs}}
Assume first that $\z_n\in D^{(2)}\setminus\{\infty^{(2)}\}$ for $n$ large enough, $n\in\N_1$, so that  $\z\in D^{(2)}\cup L$. For such $n$, the function
\[
\frac{(\xi-z_n)E_n^*(\xi)}{w(\xi)|\map(\z_n)|}, \quad \xi\in D,
\]
is holomorphic and vanishes at $\infty$. Hence, it converges locally uniformly to zero there by \eqref{errorintegral} (to handle the case where $\z_n\to\infty^{(2)}$, observe that $(\xi-z_n)/|\map(z_n)|$ is bounded on $\Delta$  independently of $z_n$).
In turn, $E_n$ is also holomorphic in $D$ and converges to zero locally uniformly there.  Indeed, if $\z\notin L$, this follows directly from \eqref{errorintegral} (where we can choose  $l_n\equiv1$) and the Cauchy representation formula. On the other hand, if $\z\in L$, observe that
\[
\frac{(\xi-z_n)E_n(\xi)}{\map(\xi)}, \quad \xi\in D,
\]
is also holomorphic. Moreover, its $L^1(\partial D)$-norm does not exceed $\const\omega_n$ by \eqref{errorintegral}, the Schwarz inequality, and since $\|w^{1/2}/\map\|_{L^2(\partial D)}\leq \const$ The latter immediately entails that this function tends to zero locally uniformly in $D$ by the Cauchy representation formula. However, as $|(\xi-z_n)/\map(\xi)|$ is bounded away from zero on compact subsets of $D$ uniformly with respect to $z_n$ (remember $z_n\to z\in\Delta$), $E_n$ converges to zero locally uniformly in $D$ as well.

Gathering what we did, it can be concluded from \eqref{eq:conjfunSn2} and  \eqref{asymptotics} that $f_h-\pi_n$ converges locally uniformly to zero on $D$ and in fact geometrically fast because $|\map|>1+\varepsilon_K$ on any compact $K\subset D$. If $\z_n=\infty^{(2)}$ for each $n$, then $E^*_n$ may have a pole at $\infty$ but $E_n^*/w$ is holomorphic in $D$ and vanishes at $\infty$. Thus, it converges locally uniformly to zero in $D$ by \eqref{errorintegral} and the Cauchy formula. As before $E_n$ also converges  locally uniformly to zero in $D$ and, from \eqref{eq:conjfunSn2} and \eqref{asymptotics}  again, we get the desired conclusion.

Assume now that $\z_n\in D^{(1)}$ for $n$ large enough, $n\in\N_1$, so that $\z\in D^{(1)}\cup L$. If $\z\in L$ and $K$ is compact in $D$, then $|\xi-z_n|\geq c_k>0$ for $\xi\in K$ and $n$ large. Moreover, the functions
\[
\frac{(\xi-z_n)E_n(\xi)}{\map(\xi)} \quad \mbox{and} \quad \frac{(\xi-z_n)E_n^*(\xi)}{\map(\xi)}, \quad \xi\in D,
\]
are  holomorphic in $D$, and using \eqref{errorintegral} and the Schwarz inequality as before we see that they go to zero locally uniformly in $D$. In particular $E_n$ and $E_n^*$ tend to zero uniformly on $K$. Thus, we conclude again from \eqref{eq:conjfunSn2} and \eqref{asymptotics} that $f_h-\pi_n$ converges to zero geometrically fast on $K$. The argument when $\z\in D^{(1)}$ and $K\subset D\setminus\{z_n\}$ is similar.

Finally, observe from \eqref{asymptotics} that
\begin{equation}
\label{Rouche}
\left|(f_h-\pi_n)-\frac{2S_n^*}{wS_n}\right|=
\left|\frac{2S_n^*}{wS_n}\right|\left|\frac{E_n^*-E_n-
\mathcal{O}(|\map|^{-n})}{1+E_n+\mathcal{O}(|\map|^{-n})}\right|.
\end{equation}
Now if $\z\in D^{(1)}$ and $\D(z,r)\subset D$ is a disk of radius $r$ centered at $z$ (the set $\{z:|z|>r\}$ if $z=\infty$) with boundary circle $\T(z,r)$, it follows from what precedes that $E_n$ and $E^*_n$ tend to zero on $\T(z,r)$. Hence, the second factor on the right-hand side of  \eqref{Rouche} also converges to zero on $\T(z,r)$ as $\N_1\ni n\to\infty$. Thus, by Rouch\'e's theorem, $f_h-\pi_n$ has exactly one pole in $\D(z,r)$ for $n$ large enough because $2S_n/(wS^*_n)$ has exactly one pole there (namely $z_n$) and none of these two functions can have a zero in $\D(z,r)$ by \eqref{eq:conjfunSn2}, \eqref{asymptotics}, and the fact that $E_n$ and $E^*_n$ converge to zero on $\T(z,r)$. This achieves the proof of the corollary.
\qed

\bibliographystyle{plain}
\bibliography{BaYa12}

\end{document}